\documentclass{amsart}[12pt, article]
\usepackage{amssymb}
\usepackage{amsmath}
\usepackage{mathrsfs}
\usepackage{amsfonts}
\usepackage{txfonts}
\usepackage{amssymb}
\usepackage{color}

\newcommand\NN{{\mathbb N}}
\newcommand\RR{{{\mathbb R}}}

\newcommand\CC{{\mathbb C}}
\def\SS {\mathbb{S}}

\newcommand\cA{{\mathcal A}}
\newcommand\cB{{\mathcal B}}
\newcommand\cD{{\mathcal D}}
\newcommand\cE{{\mathcal E}}
\newcommand\cH{{\mathcal H}}
\newcommand\cF{{\mathcal F}}
\newcommand\cL{{\mathcal L}}

\newcommand\cS{{\mathcal S}}

\newcommand\cW{{\mathcal W}}
\newcommand\cX{{\mathcal X}}

\newcommand\pP{{\bf P}}
\newcommand\iI{{\bf I}}
\newcommand\ra{{\rangle}}
\newcommand\la{{\langle}}

\newtheorem{theo}{Theorem}[section]
\newtheorem{lemm}[theo]{Lemma}
\newtheorem{defi}[theo]{Definition}
\newtheorem{coro}[theo]{Corollary}
\newtheorem{prop}[theo]{Proposition}
\newtheorem{rema}[theo]{Remark}

\begin{document}

\title[Qualitative properties of solutions]
{ Boltzmann equation without angular cutoff in the whole space: \\
III, Qualitative properties of solutions}

\author{R. Alexandre}
\address{R. Alexandre,
Department of Mathematics, Shanghai Jiao Tong University
\newline\indent
Shanghai, 200240, P. R. China
\newline\indent
and \newline\indent Irenav, Arts et Metiers Paris Tech, Ecole
Navale,
\newline\indent
Lanveoc Poulmic, Brest 29290 France }
\email{radjesvarane.alexandre@ecole-navale.fr}
\author{Y. Morimoto }
\address{Y. Morimoto, Graduate School of Human and Environmental Studies,
Kyoto University
\newline\indent
Kyoto, 606-8501, Japan} \email{morimoto@math.h.kyoto-u.ac.jp}
\author{S. Ukai}
\address{S. Ukai, 17-26 Iwasaki-cho, Hodogaya-ku, Yokohama 240-0015, Japan}
\email{ukai@kurims.kyoto-u.ac.jp}
\author{C.-J. Xu}
\address{C.-J. Xu, School of Mathematics and Statistics , Wuhan University 430072,
Wuhan, P. R. China
\newline\indent
and \newline\indent
Universit\'e de Rouen, UMR 6085-CNRS,
Math\'ematiques
\newline\indent
Avenue de l'Universit\'e,\,\, BP.12, 76801 Saint
Etienne du Rouvray, France } \email{Chao-Jiang.Xu@univ-rouen.fr}
\author{T. Yang}
\address{T. Yang, Department of mathematics, City University of Hong Kong,
Hong Kong, P. R. China} \email{matyang@cityu.edu.hk}

\subjclass[2000]{35A02, 35B40, 35B65,35H10, 35S05, 35Q20, 82B40}

\date{30 October 2010}

\keywords{Boltzmann equation, non-cutoff cross sections,
hypoellipticity, uniqueness, non-negativity, convergence rate.}

\begin{abstract}
This is a continuation of our series of works for the inhomogeneous
Boltzmann equation. We study qualitative properties of
classical solutions, precisely, the full regularization in all variables, uniqueness, non-negativity
and convergence rate to the equilibrium. Together with the results of Parts I and II about the well posedness of the Cauchy problem around Maxwellian, we conclude this series with a satisfactory mathematical theory for Boltzmann equation without angular cutoff.

\end{abstract}

\maketitle

\tableofcontents

\section{Introduction}\label{sect-IV-1}

Following our series of works
\cite{amuxy4-2,amuxy4-3}, extending results from
\cite{amuxy3b,amuxy3}, this Part III is concerned with qualitative properties associated with solutions to the Cauchy problem for the
inhomogeneous Boltzmann equation
\begin{equation}\label{IV-1.1}
f_t+v\cdot\nabla_x f=Q(f, f)\,,\,\,\,\,\,\,f|_{t=0}=f_0.
\end{equation}
We refer the reader for the complete framework, definitions and
bibliography, to our previous papers \cite{amuxy4-2,amuxy4-3}. General details about Boltzmann equation for non cutoff cross sections can be found in \cite{alex-review,cip,villani2}. Let us just recall herein that the
Boltzmann bilinear collision operator is given by
\[
Q(g, f)=\int_{\RR^3}\int_{\mathbb S^{2}}B\left({v-v_*},\sigma
\right)
 \left\{g'_* f'-g_*f\right\}d\sigma dv_*\,,
\]
where $f'_*=f(t,x,v'_*), f'=f(t,x,v'), f_*=f(t,x,v_*), f=f(t,x,v)$,
and for $\sigma\in \mathbb S^{2}$, the pre- and post-collisional velocities are
linked by the relations
$$
v'=\frac{v+v_*}{2}+\frac{|v-v_*|}{2}\sigma,\,\,\, v'_*
=\frac{v+v_*}{2}-\frac{|v-v_*|}{2}\sigma \ .
$$
The non-negative cross section
 $B(z, \sigma)$ depends only on $|z|$ and the scalar product
$\frac{z}{|z|}\,\cdot\, \sigma$. As in the previous parts, we assume
that it takes the form
\begin{equation*}
B(|v-v_*|, \cos \theta)=\Phi (|v-v_*|) b(\cos \theta),\,\,\,\,\,
\cos \theta=\frac{v-v_*}{|v-v_*|} \, \cdot\,\sigma\, , \,\,\,
0\leq\theta\leq\frac{\pi}{2},
\end{equation*}
where
\begin{equation}\label{4-2-1.2}
\Phi(|z|)=\Phi_\gamma(|z|)= |z|^{\gamma},\quad
 b(\cos \theta)\approx \theta^{-2-2s} \ \
 \mbox{when} \ \ \theta\rightarrow 0+,
\end{equation}
for some $\gamma>-3$ and $0<s<1$.

\smallskip

In the present work, we are concerned with qualitative properties of classical solutions to the Boltzmann equation, under the previous assumptions. By qualitative properties, we mean specifically regularization properties, positivity, uniqueness of solutions and asymptotic trend to global equilibrium.

Let us recall that in a close to equilibrium framework,  the existence of such classical solutions was proven in our series of papers \cite{amuxy4-2,amuxy4-3} and using a different method, by Gressmann and Strain \cite{gr-st-0,gr-st,gr-st-1}. We refer also to \cite{amuxy6} for bounded local solutions.

The first qualitative property which will be addressed here is concerned with regularization properties of classical solutions, that is, the immediate smoothing effect on the solution. For the homogeneous Boltzmann equation,  after the works of Desvillettes \cite{D95,desv-ajout1,desv-ajout2}, this issue has now a long history \cite{al-saf,al-saf-1,chen-li-xu1,desv-wen1,HMUY,MU,MUXY-DCDS,ukai}. All these works deal with smoothed type kinetic part for the cross sections, which therefore rules out the more physical assumption above, that is, including the singular behavior for relative velocity near $0$. We refer the reader to our forthcoming work \cite{amuxy5} for this issue.

Regularization effect for the inhomogeneous Boltzmann equation was studied in our previous works \cite{amuxy-nonlinear-3,amuxy3}, but for Maxwellian type molecules or smoothed kinetic parts for the cross section. Nevertheless, we have introduced many technical tools, some of which are helpful for tackling the singular assumption above. In particular, by improving the pseudo-differential calculus and functional estimates from \cite{amuxy-nonlinear-3,amuxy3}, we shall be able to prove our regularity result.

We shall use the following standard weighted Sobolev space defined, for $k, \ell \in \RR$, as
$$
H^k_\ell = H^k_\ell (\RR^3_v ) = \lbrace f\in \cS ' (\RR^3_v ) ; \ W_\ell f \in H^k (\RR^3_v ) \rbrace
$$
and for any open set  $\Omega \subset \RR^3_x$
$$
H^k_\ell (\Omega \times \RR^3_{v} ) = \lbrace f\in \cD ' (\Omega \times \RR^3_{ v} ) ; \ W_\ell f
\in H^k (\Omega \times \RR^3_{ v} ) \rbrace
$$
where $W_\ell (v) = \langle v\rangle^\ell = (1+ | v|^2 )^{\ell/2}$ is always the weight for $v$ variables. Herein, $(\cdot, \cdot)_{L^2} = (\cdot,\cdot)_{L^2 (\RR^3_v)}$ denotes the usual scalar product in $L^2=L^2 (\RR^3)$ for $v$ variables. Recall that $L^2_\ell=H^0_\ell$.

\begin{theo}\label{theo-IV-1.1}
Assume \eqref{4-2-1.2} holds true, with  $0<s<1$, $\gamma>\max\{-3,
-3/2-2s\}$, $0<T\leq+\infty$. Let $\Omega$ be an open domain of $\RR^3_x$. Let $f\in L^\infty([0, T]; H^5_\ell(\Omega\times\RR^3))$, for any $\ell\in\NN$, be a solution of Cauchy problem
\eqref{IV-1.1}. Moreover, assume that $f$ satisfies the following local coercivity estimate : for any compact $K\subset\Omega$ and $0<T_1<T_2<T$, there exist two constants $\eta_0>0, C_0>0$ such that
\begin{equation}\label{coercivity-nonlinear-a}
-(Q(f,\, h),\,\,h)_{L^2(\RR^7)}\geq \eta_0\|h\|^2_{H^s_{\gamma/2}(\RR^7)}-C_0
\|h\|^2_{L^2_{\gamma/2+s}(\RR^7)}\,
\end{equation}
for any $h\in C^1_0(]T_1, T_2[; C^\infty_0(K; H^{+\infty}_\ell(\RR^3)))$. Then we have
$$
f\in C^\infty(]0, T[\times\Omega; \cS(\RR^3))\, .
$$
\end{theo}


Classical solutions satisfying such a local coercivity estimate do exist \cite{amuxy4-2,amuxy4-3}, see Corollary \ref{coro-coercivity-1} in next section.

Our next result is related to uniqueness of solutions. We shall consider function spaces
with exponential decay in the velocity variable, for $m\in\RR$
$$
\tilde{\cE}^m_0(\RR^6)=\Big\{g\in\cD'(\RR^6_{x, v});\, \exists \,\rho>0\,
\, \mbox{ s.t.} \,\, e^{\rho <v>^2} g\in L^\infty(\RR_x^3; H^m(\RR^3_{v}) )\Big\},
$$
and for $T>0$
\begin{eqnarray*}
\tilde {\mathcal E}^m([0,T]\times{\mathbb R}^6_{x, v})&=&\Big\{f\in
C^0([0,T];{\mathcal D}'({\mathbb R}^6_{x, v}));\, \exists \,\rho>0
\\
&&\hskip 0.5cm \mbox{ s.t. } \,\, e^{\rho \langle v \rangle^2} f\in L^\infty([0,
T]\times \RR^3_x ;\,\, H^m({\mathbb R}^3_{v})) \Big\}.
\end{eqnarray*}

\begin{theo}\label{theo-IV-1.2}
Assume that $0<s<1$ and $\max \lbrace -3, -3/2-2s \rbrace <\gamma < 2-2s $. Let $f_0\geq 0$ and $ f_0 \in \tilde \cE^{0}_0(\RR^6)$.
Let $ 0<T < +\infty$ and
suppose that $f \in
\tilde { \mathcal E}^{2s}([0,T]\times{\mathbb R}^6_{x, v})$
is a non-negative solution to the Cauchy problem \eqref{IV-1.1}. Then any solution in the function space  $\tilde { \mathcal E}^{2s}([0,T]\times{\mathbb R}^6_{x, v})$ coincides with $f$.
\end{theo}

\noindent
{\bf Remark.}
{\it

1) Note that the solutions considered above are not necessarily classical ones. Moreover, Theorem \ref{theo-IV-1.2} does not require the coercivity. On the other hand, if we suppose the coercivity, then we can get the uniqueness in the function space
$\tilde { \mathcal E}^{s}([0,T]\times{\mathbb R}^6_{x, v})$, without the non-negativity assumption, see precisely Theorem \ref{uniqueness} in Section \ref{sect-IV-7}.

2) We can also remove the restriction $\gamma+2s<2$, if we consider the small perturbation around Maxwellian, see precisely Theorem \ref{uniqueness-small-pur} in Section \ref{sect-IV-7}.

3) Finally, in the soft potential case $\gamma+2s\leq 0$, we can refine the above uniqueness results which can be applied to the solution of Theorem 1.4 of \cite{amuxy4-2}, see precisely Theorem \ref{unique-x-global} in Section \ref{sect-IV-7}.
}
\smallskip

Our next issue is about the non-negativity of solutions. We shall use the following modified weighted Sobolev spaces: For $k\in\NN,\,\,\ell\in\RR$
\begin{align*}
{\tilde \cH}^k_\ell(\RR^6)&=\Big\{f\in\cS'(\RR^6_{x, v})\, ;\,\,
\|f\|^2_{{\tilde \cH}^k_\ell(\RR^6)}=\sum_{|\alpha|+|\beta|\leq N}\|\tilde W_{\ell-|\beta|}
\partial^\alpha_\beta f\|^2_{L^2(\RR^6)}<+\infty\,\Big\}\, ,
\end{align*}
where $\tilde W
_\ell=(1+|v|^2)^{| s+\gamma/2|\,\ell/2}$.

Combining with the existence results of \cite{amuxy4-2,amuxy4-3} and the above Theorem \ref{theo-IV-1.2}, one has

\begin{theo}\label{theo-IV-1.3}
Let $0<s<1$, $\gamma>\max\{-3, -3/2-2s\}$, $k\geq 6$. There exist $\varepsilon_0 >0 $ and $\ell_0$ such that the Cauchy problem \eqref{IV-1.1} admits a unique global solution $f=\mu+{\mu}^{1/2}\,g$ for initial datum $f_0=\mu+{\mu}^{1/2}\,g_0$ satisfying

\noindent{\bf 1) } $g\in {L^\infty([0, +\infty[;  {H}^{k}_{\ell_0}(\RR^6)))}$, if $\gamma+2s> 0$ and $
\|g_0\|_{{H}^{k}_{\ell_0}(\RR^6)}\leq \varepsilon_0$.

\noindent{\bf 2) } $g\in {L^\infty([0, +\infty[;{\tilde \cH}^{k}_{\ell_0}(\RR^6))}$, if $\gamma+2s\leq 0$ and $
\|g_0\|_{{\tilde \cH}^{k}_{\ell_0}(\RR^6)}\leq \varepsilon_0$.

If $f_0=\mu+{\mu}^{1/2}\,g_0\geq 0$, then the above  solution $f=\mu+{\mu}^{1/2}\,g\geq 0$.

\end{theo}

\noindent
{\bf Remark.}
{\it
The existence of global solution was proved in \cite{amuxy4-2,amuxy4-3}, while the uniqueness follows from Theorem \ref{theo-IV-1.2}, more precisely Theorem \ref{uniqueness-small-pur}, in Section \ref{sect-IV-7}.

}
\smallskip

One of the basic issues in the mathematical
theory for  Boltzmann equation theory is about the convergence of solutions
to equilibrium. This topic has been recently renewed and
complemented by proofs of optimal  convergence rates in the whole
space, see for example
\cite{Duan-U-Y-Zhao,glassey,kawashima,villani2,Y-Yu} and references
therein. This is closely related to the
study of the hypocoercivity theory that
 is about the interplay of a conservative operator and
a degenerate diffusive operator which gives the convergence
to the equilibrium. Note that this kind of interplay also gives
the full regularization.

For later use, denote
$$
\mathcal{N}=\mbox{span}\{\mu^{\frac 12}, \mu^{\frac 12}v_i, \mu^{\frac12} |v|^2,\quad i=1,2,3\},
$$
as the null space of the linearized Boltzmann collision operator, and $\pP$ the projection operator
to $\mathcal{N}$ in $L^2(\RR^3_v)$.

 For the problem considered in this paper,
we have the following convergence rate estimates.

\begin{theo} \label{theo-IV-1.4}
Let $0<s<1$ and $f=\mu +\mu^{1/2}\, g $ be a global solution of the Cauchy problem \eqref{IV-1.1} with initial datum $f_0=\mu+{\mu}^{1/2}\,g_0$. We have the
following two cases:

\noindent{\bf 1)} Let $\gamma+2s>0$, $ N\ge
6, \ell>3/2+2s+\gamma$. There exists $\varepsilon_0>0$
such that if $
\|g_0\|^2_{L^1(\RR^3_x; L^2(\RR^3_v))} + \|
g_0\|^2_{H^N_\ell(\RR^6)}\leq \varepsilon_0 $
and $
g\in L^\infty([0, +\infty[\,;\,\,H^{N}_\ell(\RR^6))$, then we have for
all $t>0$,
$$
\|g(t)\|^2_{L^2(\RR^6)}= \|
\pP g(t)\|^2_{L^2(\RR^6)}+\|({\bf I}-\pP)g(t)\|^2_{L^2(\RR^6)}\lesssim (1+t)^{-3/2},
$$
and
$$
\sum_{1\leq|\alpha|\leq N}\|\partial^\alpha
\pP g(t)\|^2_{L^2(\RR^6)}+\sum_{|\alpha|\leq
N}\|\partial^\alpha({\bf I}-\pP)g(t)\|^2_{L^2(\RR^6)}\lesssim (1+t)^{-5/2}.
$$

\noindent{\bf 2)} Let~~ $ \max\{-3, -\frac 32 -2s\} < \gamma \le -2s, N\ge 6, \ell\ge N+1$. There exists $\varepsilon_0>0$
such that if $\|g_0\|^2_{{\tilde \cH}^{N}_\ell(\RR^6)}\leq \varepsilon_0$ and $
g\in
L^\infty([0, +\infty[\,;\,\,{\tilde \cH}^{N}_\ell(\RR^6))$, then we have
for all $t>0$,
$$
\sup_{x\in \RR^3}\|g(t)\|^2_{H^{N-3}(\RR^3_v)}\lesssim (1+t)^{-1}.
$$
\end{theo}

We emphasize that the above convergence rate for the hard potential
case is optimal in the sense that it is the same for the linearized
problem through either spectrum analysis in \cite{Pao}, or direct
Fourier transform using the compensating function introduced in
\cite{kawashima}. However, the convergence rate for soft potential
is not optimal. In fact, how to obtain an optimal convergence rate
even for the cutoff soft potential is still an unsolved problem
\cite{strain-guo, ukai-asano}.

We also would like to mention that the above convergence rate is
for  the whole space setting. If the problem is instead
considered on the torus with small perturbation, then the
exponential decay for hard potential can be obtained, and this point
is a direct consequence of the energy estimates given in \cite{amuxy4-3} by
using Poincar\'e inequality (this is for example the case considered
in \cite{gr-st}).

\smallskip

Before presenting the plan of the paper we want to give some comments on our proofs. First of all, our proof of regularization property applies to the classical
solutions obtained in \cite{amuxy4-2,amuxy4-3}. Note that from those existence theorems, one can show that if the initial data
 satisfying $\|g_0\|_{H^{k}_{l}}\le \epsilon_k$ for $k\ge 6$ and
 $l\ge l_0$ for some $l_0$, the solution is also in $H^k$  when $\epsilon_k$ is small. However, the current existence theory
  does not yield that  $g\in H^{k+N}$, under the condition that $g_0\in H^{k+N}$  for  $N>0$ if $||g_0||_{H^{k+N}}$ is not small.
   Therefore, we can not just mollify the initial data to
 study the full regularity by working formally on the smooth solution.
 Instead, we need analytic tools from peudo-differential theory
  and harmonic analysis to study the gain of regularity rigorously. In fact, it is a standard technic for the hypoellipticity of linear differential operators \cite{hormander,M-X,oleinik}. The same comments apply for the uniqueness and positivity
issues for which we give also rigorous proofs.
\smallskip

The paper is organized as follows. In Section \ref{sect-IV-2}, we give the functional analysis of the collision operator, including upper bounds, commutators estimates and coercivity. In Section \ref{sect-IV-3}, we prove Theorem \ref{theo-IV-1.1} giving the regularization of solutions. Section \ref{sect-IV-7} is devoted to precise versions of uniqueness results related to Theorem \ref{theo-IV-1.2}, while Section \ref{sect-IV-8}   proves the non-negativity of solutions. Finally the last Section proves Theorem \ref{theo-IV-1.4} about the convergence of solutions to equilibrium.

{\bf Notations:} Herein, letters $f$, $g, \cdots$ stand for various suitable  functions, while $C$, $c, \cdots$ stand for various numerical constants, independent from functions $f$, $g, \cdots$ and which may vary from line to line. Notation $A\lesssim B$ means that there exists a constant $C$ such that $A \le C B$, and similarly for $A\gtrsim B$. While $A \sim B$ means that there exist two generic constants $C_1, C_2 >0$ such that
$$C_1A \leq B \leq C_2 A.$$

\section{Functional analysis of the collision operator}\label{sect-IV-2}
\smallskip \setcounter{equation}{0}

In this section, we study the upper bound and commutators
estimates for the collision operator $Q(\,\cdot,\,\cdot\,)$. Since
it is only an operator with respect to velocity variable, in this section,
our analysis is on $\RR^3_v$, forgetting variable $x$. In what follows, we denote $\tilde\Phi_{\gamma}$ by $\tilde\Phi_\gamma (z) = (1+|z|^2)^{\gamma /2}$. $Q_{\tilde\Phi_\gamma}$ will denote the collision operator defined with the modified kinetic factor $\tilde\Phi_{\gamma}$.

\subsection{Upper bound estimate}
For $0< s<1$,
$\gamma\in\RR$, we proved the following upper bounded estimate
(Theorem 2.1 of \cite{amuxy-nonlinear-3})
\begin{equation}\label{upper-I}
| (Q_{\tilde{\Phi}_\gamma} (f,  g),\, h )|  \lesssim ||f||_{L^1_{\ell^++(\gamma+2s)^+}} || g||_{H^{m+s}_{(\ell+\gamma+2s)^+}} \|h\|_{H^{s-m}_{-\ell}}\,,
\end{equation}
for any $m, \ell\in\RR$, and the estimate of commutators with weight (Lemma 2.4 of \cite{amuxy-nonlinear-3})
\begin{equation}\label{commutators-I}
\Big|\Big(W_\ell\,Q_{\tilde{\Phi}_\gamma} (f, g)-Q_{\tilde{\Phi}_\gamma}
(f, W_\ell g),\, h\Big)\Big|\lesssim
\|f\|_{L^1_{\ell+(2s-1)^++\gamma^+}}
\|g\|_{H^{(2s-1+\epsilon)^+}_{\ell+(2s-1)^++\gamma^+}}\,  \|h\|_{L^{2}}\,,
\end{equation}
for any $\ell\in\RR$.

For the singular type of kinetic factors considered herein
$|v-v_*|^\gamma$, we need to take into account the singular behavior
close to $0$. Therefore, we decompose the kinetic factor in two
parts. Let $0\leq \phi (z)\leq 1$ be a smooth radial function with
value $1$ for $z$ close to $0$, and $0$ for large values of $z$. Set
$$
\Phi_\gamma (z) = \Phi_\gamma (z) \phi (z) + \Phi_\gamma (z) (1-\phi (z)) = \Phi_c (z) + \Phi_{\bar c} (z).
$$
And then correspondingly we can write
$$
Q (f, g) = Q_c (f, g) + Q_{\bar c} (f, g),
$$
where the kinetic factor in the collision operator is defined
according to the decomposition respectively. Since $\Phi_{\bar c}
(z)$ is smooth, and $\Phi_{\bar c} (z)\leq \tilde{\Phi}_\gamma(z)$,
$Q_{\bar c} (f, g)$ has similar properties as for
$Q_{\tilde{\Phi}_\gamma} (f, g)$ as regards upper bounds and
commutators estimatations, which means that \eqref{upper-I} and
\eqref{commutators-I} hold true for $Q_{\bar c} (f, g)$.

{}From now on, we concentrate on the study the singular part $Q_{c}
(f, g)$, referring for the smooth part $Q_{\bar c} (f, g)$ to
\cite{amuxy-nonlinear-3}. Note that in \cite{amuxy4-2}, the same decomposition was also used, but for the modified operator $\Gamma (f,g)$. Here,
the absence of the gaussian factor slightly adds some more
difficulties.
\begin{prop} \label{upper-c}
Let $0< s<1,  \gamma > \max\{-3, -2s-3/2\}$ and $m \in\  [s-1,s]$. Then we have
\[
| (Q_c (f,  g),h )|  \lesssim \|f\|_{L^2} ||g||_{H^{s+ m}} \|h\|_{H^{s -m}}\,.
\]
\end{prop}

\begin{rema}\label{weight-version123} As will be clearer from the proof below,
the following precise estimates are also available: if $\gamma +2s
>0$, we have
$$
| (Q_c (f,  g), h )|  \lesssim || f||_{L^1} || g||_{H^{s+ m}} \|h\|_{H^{s-m}}\,.
$$
and moreover if $\gamma +2s > -1$, we have
\[
| (Q_c (f,  g); h )|  \lesssim  \|f\|_{L^{ 3/2}}|| g||_{H^{s+ m}} \|h\|_{H^{s-m}}\,.
\]
\end{rema}

For the proof of Proposition \ref{upper-c}, we shall follow some of
the arguments form \cite{amuxy4-2}. First of all, by using the formula from
the Appendix of \cite{al-1}, and as in \cite{amuxy4-2}, one has
\begin{align*}
( Q_c(f, g), h ) =& \iiint b \Big({\xi\over{ | \xi |}} \cdot \sigma \Big) [ \hat\Phi_c (\xi_* - \xi^- ) - \hat \Phi_c (\xi_* ) ] \hat f (\xi_* ) \hat g(\xi - \xi_* ) \overline{{\hat h} (\xi )} d\xi d\xi_*d\sigma .\\
= & \iiint_{ | \xi^- | \leq {1\over 2} \la \xi_*\ra }  \cdots\,\, d\xi d\xi_*d\sigma
+ \iiint_{ | \xi^- | \geq {1\over 2} \la \xi_*\ra } \cdots\,\, d\xi d\xi_*d\sigma \,\\
=& A_1(f,g,h)  +  A_2(f,g,h) \,\,.
\end{align*}
Then, we write $A_2(f,g,h)$ as
\begin{align*}
A_2 &=  \iiint b \Big({\xi\over{ | \xi |}} \cdot \sigma \Big) {\bf 1}_{ | \xi^- | \ge {1\over 2}\la \xi_*\ra }
\hat\Phi_c (\xi_* - \xi^- ) \hat f (\xi_* ) \hat g(\xi - \xi_* ) \overline{{\hat h} (\xi )} d\xi d\xi_*d\sigma .\\
&- \iiint b \Big({\xi\over{ | \xi |}} \cdot
 \sigma \Big){\bf 1}_{ | \xi^- | \ge {1\over 2}\la \xi_*\ra } \hat \Phi_c (\xi_* ) \hat f (\xi_* ) \hat g(\xi - \xi_* ) \overline{{\hat h} (\xi )} d\xi d\xi_*d\sigma \\
&= A_{2,1}(f,g,h) - A_{2,2}(f,g,h)\,.
\end{align*}
While for $A_1$, we use the Taylor expansion of $\hat \Phi_c$ at
order $2$ to have
$$
A_1 = A_{1,1} (f,g,h) +A_{1,2} (f,g,h)
$$
where
$$
A_{1,1} = \iiint b\,\, \xi^-\cdot (\nabla\hat\Phi_c)( \xi_*)
{\bf 1}_{ | \xi^- | \leq {1\over 2} \la \xi_*\ra }  \hat f (\xi_* ) \hat g(\xi - \xi_* ) \bar{\hat h} (\xi ) d\xi d\xi_*d\sigma,
$$
and $A_{1,2} (F,G,H)$ is the remaining term corresponding to the second order term in the Taylor expansion of $\hat\Phi_c$. The $A_{i,j}$ with
$i,j=1,2$ are estimated by the following lemmas.
\begin{lemm}\label{A-1}
We have
\[
| A_{1,1} |+| A_{1,2}|  \lesssim \|f\|_{L^2} || f||_{H^{s+ m}} \|h\|_{H^{s -m}}\,.
\]
\end{lemm}
\begin{proof}
Considering firstly $A_{1,1}$, by writing
\[
\xi^- = \frac{|\xi|}{2}\left(\Big(\frac{\xi}{|\xi|}\cdot \sigma\Big)\frac{\xi }{|\xi|}-\sigma\right)
+ \left(1- \Big(\frac{\xi}{|\xi|}\cdot \sigma\Big)\right)\frac{\xi}{2},
\]
we see that the integral corresponding to the first term on the right hand side vanishes because of the symmetry
on $\SS^2$.
Hence, we have
\[
A_{1,1}= \iint_{\RR^6} K(\xi, \xi_*)
\hat f (\xi_* ) \hat g(\xi - \xi_* ) \bar{\hat h} (\xi ) d\xi d\xi_* \,,
\]
where
\[
K(\xi,\xi_*) = \int_{\SS^2}
 b \Big({\xi\over{ | \xi |}} \cdot \sigma \Big)
\left(1- \Big(\frac{\xi}{|\xi|}\cdot \sigma\Big)\right)\frac{\xi}{2}\cdot
(\nabla\hat\Phi_c)( \xi_*)
{\bf 1}_{ | \xi^- | \leq {1\over 2} \la \xi_*\ra } d \sigma \,.
\]
Note that $| \nabla \hat \Phi_c (\xi_*) | \lesssim {1\over{\la
\xi_*\ra^{3+\gamma +1}}}$, from the Appendix of \cite{amuxy4-2}. If $\sqrt 2
|\xi| \leq \la \xi_* \ra$, then $|\xi^-| \leq \la \xi_* \ra/2$ and
this imply  the fact that $0 \leq \theta \leq \pi/2$, and we have
\begin{align*}
|K(\xi,\xi_*)| &\lesssim \int_0^{\pi/2} \theta^{1-2s} d \theta\frac{ \la \xi\ra}{\la \xi_*\ra^{3+\gamma +1}}
\lesssim \frac{1  }{\la \xi_*\ra^{3+\gamma}}\left(
\frac{\la \xi \ra}{\la \xi_*\ra}\right) \,.
\end{align*}
On the other hand, if $\sqrt 2 |\xi| \geq \la \xi_* \ra$, then
\begin{align*}|K(\xi,\xi_*)| &\lesssim \int_0^{\pi\la \xi_*\ra /(2|\xi|)} \theta^{1-2s} d \theta\frac{ \la \xi\ra}{\la \xi_*\ra^{3+\gamma +1}}
\lesssim \frac{1  }{\la \xi_*\ra^{3+\gamma}}\left(
\frac{\la \xi \ra}{\la \xi_*\ra}\right)^{2s-1}\,.
\end{align*}
Hence we obtain
\begin{align}\label{later-use1}
|K(\xi,\xi_*)|
\lesssim
\frac{1  }{\la \xi_*\ra^{3+\gamma}}\left\{
\left(
\frac{\la \xi \ra}{\la \xi_*\ra}\right){\bf 1}_{\sqrt 2 |\xi| \leq \la \xi_* \ra}
+{\bf 1}_{ \sqrt 2 |\xi| \geq  \la \xi_* \ra \geq |\xi|/2}
+
\left(
\frac{\la \xi \ra}{\la \xi_*\ra}\right)^{2s}
{\bf 1}_{\la \xi_* \ra \leq |\xi|/2}\right\}\,.
\end{align}
Notice that
\begin{equation}\label{equivalence-relation}
\left \{ \begin{array}{ll}
\la \xi \ra \lesssim \la \xi_* \ra \sim \la \xi-\xi_*\ra   &\mbox{on supp ${\bf 1}_{\la \xi_* \ra\geq \sqrt 2 |\xi|}$}\\
\la \xi \ra \sim \la \xi-\xi_*\ra   &\mbox{on supp ${\bf 1}_{\la \xi_*\ra \leq |\xi |/2 } $} \\
\la \xi \ra \sim \la \xi_* \ra \gtrsim   \la \xi-\xi_*\ra  &
\mbox{on supp ${\bf 1}_{\sqrt 2 |\xi| \geq \la \xi_*\ra \geq | \xi|/2 }$\,.}
\end{array}
\right.
\end{equation}
Replacing the factors
$\la\xi \ra/\la \xi_*\ra$ and $(\la\xi \ra/\la \xi_*\ra)^{2s}$ on the right hand side of \eqref{later-use1}
by
\[
\left(\frac{\la \xi -\xi_*\ra}{\la \xi_* \ra}\right)^{s+m} \left(\frac{\la \xi \ra}{\la \xi_* \ra}\right)^{s-m}
\qquad \mbox{and } \qquad \frac{\la \xi \ra^{s+m} \la \xi-\xi_*\ra^{s-m}}{\la \xi_*\ra^{2s}}\,,
\]
respectively, we obtain
\begin{align}\label{kernel-estimate}
|K(\xi,\xi_*)| &\lesssim
\frac{\la \xi\ra^{s-m}  \la \xi-\xi_*\ra^{s+m} }{\la \xi_*\ra^{3+\gamma +2s}}
+ \frac{{\bf 1}_{ \la \xi -\xi *\ra \lesssim \la \xi_{*}  \ra}}{\la \xi_*\ra^{3+\gamma +s-m}
\la \xi-\xi_*\ra^{s+m} }   \la \xi\ra^{s-m}  \la \xi-\xi_*\ra^{s+m}\,.
\end{align}
Putting $\tilde {\hat g}(\xi)=  \la \xi \ra^{s+m}  \hat g(\xi), \tilde {\hat h}(\xi)=  \la \xi \ra^{s -m} \hat h(\xi)$,
we have by the Cauchy-Schwarz inequality
\begin{align*}
|A_{1,1}|^2 &\lesssim \int_{\RR^6}
\frac{|\hat f(\xi_*)| }{\la \xi_*\ra^{3+\gamma +2s}}|\tilde {\hat g}(\xi-\xi_*)|^2 d\xi d\xi_*
\int_{\RR^6}
\frac{|\hat f(\xi_*)| }{\la \xi_*\ra^{3+\gamma +2s}}|\tilde {\hat h}(\xi)|^2 d\xi d\xi_*\\
&+
 \int_{\RR^6}
\frac{|\hat f(\xi_*)|^2 }{\la \xi_*\ra^{6+2\gamma +2s-2m}} \frac{{\bf 1}_{ \la \xi -\xi *\ra \lesssim \la \xi_{*}  \ra}}{
\la \xi-\xi_*\ra^{2s+2m} }d\xi d\xi_*
\int_{\RR^6}
|\tilde {\hat g}(\xi-\xi_*)|^2 |\tilde {\hat h}(\xi)|^2 d\xi d\xi_*\\
&= \cA \cB +  \cD \cE\,.
\end{align*}
Since $\la \xi_*\ra^{-(3+\gamma +2s)} \in L^2$, the Cauchy-Schwarz inequality again shows
\[
\cA \lesssim \left(  \int_{\RR^3}\frac{|\hat f(\xi_*)|}{\la \xi_*\ra^{3+\gamma +2s}} d\xi_* \right) \|g\|_{H^{s+m}}
\lesssim \|f\|_{L^2}\|g\|_{H^{s+m}}
\,,\enskip \cB \lesssim \|f\|_{L^2}\|h\|_{H^{s-m}}\,.
\]
Note that
\[
\int\frac{{\bf 1}_{ \la \xi -\xi *\ra \lesssim \la \xi_{*}  \ra}}{
\la \xi-\xi_*\ra^{2s+2m} } d\xi \enskip \lesssim
\left \{ \begin{array}{lcl}
\displaystyle  \frac{1}{\la \xi_*\ra^{-3+2s+2m} }& \mbox{if} & s+m <3/2\\
 \log \la \xi_*\ra & \mbox{if} & s+m \ge 3/2\,.
\end{array}
\right. \]
Since $3+2(\gamma +2s) >0$ and $6+2 \gamma+ 2(s-m) >0$, we get
$\cD \le \|f \|_{L^2}^2 $, which concludes the desired bound for $A_{1,1}$.

Remark that if $\gamma +2s >0$ then we obtain $|A_{1,1} | \lesssim
\|F \|_{L^1} \|G\|_{H^{s+m}} \|H\|_{H^{s-m}}$ because
 $\|\hat F\|_{L^\infty} \leq \|F\|_{L^1}$.
If $0 \geq \gamma +2s > -3/2$ then we can just estimate
$|A_{1,1} | \lesssim \|F \|_{L^2} \|G\|_{H^{s+m}} \|H\|_{H^{s-m}}$.
If $0 \geq \gamma +2s > -1$
then
$|A_{1,1} | \lesssim \|F \|_{L^{3/2} }\|G\|_{H^s} \|H\|_{H^s}$.
Those follow from  the H\"older inequality and $\|\hat F\|_{L^p} \leq \|F\|_{L^q}$ with
$1/p +1/q=1$.

Now we consider  $A_{1,2} (f, g, h)$, which comes from the second order term of the Taylor expansion. Note that
$$
A_{1,2} = \iiint  b \Big({\xi\over{ | \xi |}} \cdot \sigma \Big)\int^1_0 d\tau (\nabla^2\hat \Phi_c) (\xi_* -\tau\xi^- ) \cdot\xi^- \cdot\xi^- \hat F (\xi_* ) \hat G(\xi - \xi_* ) \bar{\hat H} (\xi ) d\sigma d\xi d\xi_*\, .
$$
Again from the Appendix of \cite{amuxy4-2}, we have
$$
| (\nabla^2\hat \Phi_c) (\xi_* -\tau\xi^- ) | \lesssim {1\over{\la  \xi_* -\tau \xi^-\ra^{3+\gamma +2}}}
\lesssim
 {1\over{\la \xi_*\ra^{3+\gamma +2}}},
$$
because $|\xi^-| \leq \la \xi_*\ra/2$.
Similar to $A_{1,1}$, we can obtain
\[
|A_{1,2}| \lesssim
 \iint_{\RR^6} \tilde K(\xi, \xi_*)
\hat f (\xi_* ) \hat g(\xi - \xi_* ) \bar{\hat h} (\xi ) d\xi d\xi_* \,,
\]
where $\tilde K(\xi,\xi_*)$ has the following upper bound
\begin{align}\label{later-use2}
\tilde K(\xi,\xi_*) &\lesssim \int_0^{\min(\pi/2, \,\, \pi\la \xi_*\ra /(2|\xi|))} \theta^{1-2s} d \theta
\frac{ \la \xi\ra^2}{\la \xi_*\ra^{3+\gamma +2}}\\
&\lesssim
\frac{1  }{\la \xi_*\ra^{3+\gamma}}\left\{
\left(
\frac{\la \xi \ra}{\la \xi_*\ra}\right)^2{\bf 1}_{\sqrt 2 |\xi| \leq \la \xi_* \ra}
+{\bf 1}_{ \sqrt 2 |\xi| \geq  \la \xi_* \ra \geq |\xi|/2}
+\left(
\frac{\la \xi \ra}{\la \xi_*\ra}\right)^{2s}
{\bf 1}_{\la \xi_* \ra \leq |\xi|/2}\right\}\,,\notag
\end{align}
from which we obtain the same inequality as \eqref{kernel-estimate} for $\tilde K(\xi,\xi_*)$.
Hence we obtain
 the desired bound for $A_{1,2}$. And this completes
the proof of the lemma.
\end{proof}

\begin{lemm}\label{A-2}
We have also
\[
 | A_{2,1} |+| A_{2,2}|
 \lesssim \|f\|_{L^2} || f||_{H^{s+ m}} \|h\|_{H^{s -m}}\,.
\]
\end{lemm}

\begin{proof}
In view of the definition of $A_{2,2}$, the fact that $|\xi| \sin(\theta/2) =|\xi^-| \geq \la \xi_*\ra/2$ and
$\theta \in [0,\pi/2]$ imply $\sqrt 2 |\xi| \geq \la \xi_*\ra$.
We can then
directly compute the spherical integral appearing inside $A_{2,2}$ together with $\Phi$
as follows:
\begin{align}\label{A-2-2}
&\int  b \Big({\xi\over{ | \xi |}} \cdot \sigma \Big)\Phi(\xi_*) {\bf 1}_{ | \xi^- | \ge {1\over 2}\la \xi_*\ra } d\sigma
 \lesssim  {1\over{\la \xi_* \ra^{3+\gamma }}} \frac{\la  \xi\ra^{2s} }{\la \xi_*\ra^{2s}}{\bf 1}_{\sqrt 2 |\xi| \geq \la \xi_* \ra} \\
 &\lesssim
\frac{\la \xi\ra^{s-m}  \la \xi-\xi_*\ra^{s+m} }{\la \xi_*\ra^{3+\gamma +2s}}
+ \frac{{\bf 1}_{ \la \xi -\xi *\ra \lesssim \la \xi_{*}  \ra}}{\la \xi_*\ra^{3+\gamma +s-m}
\la \xi-\xi_*\ra^{s+m} }   \la \xi\ra^{s-m}  \la \xi-\xi_*\ra^{s+m}\,,\notag
\end{align}
which yields the desired estimate for $A_{2, 2}$.

We now turn to
$$
A_{2,1}=  \iiint b\,\, {\bf 1}_{ | \xi^- | \ge {1\over 2} \la  \xi_*\ra }\hat \Phi_c (\xi_* - \xi^-) \hat f (\xi_* ) \hat g(\xi - \xi_* ) \bar{\hat h} (\xi ) d\sigma d\xi d\xi_* .
$$
Firstly, note that we can  work on the set $| \xi_* \,\cdot\,\xi^-| \ge {1\over 2} | \xi^-|^2$. In fact, on the complementary of this
set, we have
 $| \xi_* \,\cdot\,\xi^-| \leq {1\over 2} | \xi^-|^2$ so that
 $|\xi_* -\xi^-| \gtrsim | \xi_*|$, and in this case,
we can proceed in the same way as for $A_{2,2}$. Therefore, it suffices to estimate
\begin{align*}
A_{2,1,p}&=  \iiint b\,\, {\bf 1}_{ | \xi^- | \ge {1\over 2} \la \xi_*\ra }{\bf 1}_{| \xi_* \,\cdot\,\xi^-| \ge {1\over 2} | \xi^-|^2}\hat \Phi_c (\xi_* - \xi^-) \hat f (\xi_* ) \hat g(\xi - \xi_* ) \overline{{\hat h} (\xi )} d\sigma d\xi d\xi_* \,.
\end{align*}
By
\[
{\bf 1}= {\bf 1}_{\la \xi_* \ra \geq |\xi|/2} {\bf 1}_{\la\xi-\xi_* \ra \leq \la \xi_* - \xi^- \ra}
+  {\bf 1}_{\la \xi_* \ra \geq |\xi|/2} {\bf 1}_{\la\xi-\xi_* \ra > \la \xi_* - \xi^-\ra}
+  {\bf 1}_{\la \xi_* \ra < |\xi|/2}
\]
we decompose
\begin{align*}
A_{2,1,p}
=
A_{2,1,p}^{(1)} + A_{2,1,p}^{(2)} +A_{2,1,p}^{(3)} \,.
\end{align*}
On the sets for
above integrals, we have $\la \xi_* -\xi^- \ra \lesssim \,
\la \xi_* \ra$, because $| \xi^- | \lesssim | \xi_*|$
that follows from  $| \xi^-|^2 \le 2 | \xi_* \cdot\xi ^-| \lesssim |\xi^-|\, | \xi_*|$.
Furthermore, on the sets for $A_{2,1,p}^{(1)}$ and $A_{2,1,p}^{(2)}$  we have $\la \xi \ra \sim \la \xi_* \ra$,
so that $\sup \Big(b\,\, {\bf 1}_{ | \xi^- | \ge {1\over 2} \la \xi_*\ra } {\bf 1}_{\la \xi_* \ra \geq |\xi|/2}\Big) \lesssim
{\bf 1}_{|\xi^- |\leq |\xi|/\sqrt 2}$ and
$\la \xi_* -\xi^- \ra \lesssim \,
\la \xi \ra$.
Hence we have, in view of $s-m \geq 0$,
\begin{align*}
|A_{2,1,p}^{(1)} | ^2  \lesssim& \iiint \frac{
|\hat \Phi_c (\xi_* - \xi^-) |^2 |\hat f (\xi_* )|^2  }
{\la \xi_* -\xi^- \ra^{2s-2m}}\frac{ {\bf 1}_{\la\xi-\xi_* \ra \leq \la \xi_* - \xi^- \ra}}{\la\xi-\xi_* \ra^{2s+2m}}d\xi d\xi_* d \sigma\\
& \times \iiint  |\la\xi-\xi_* \ra^{s+m}\hat g(\xi - \xi_* )|^2 |\la\xi \ra^{s-m}{\hat h} (\xi ) |^2 d\sigma d\xi d\xi_*\,.
\end{align*}
If $\gamma +2s >0$ then by the change of variables $\xi_*-\xi^- \rightarrow u$ we have
\begin{align*}
|A_{2,1,p}^{(1)} | ^2  \lesssim&  \|\hat F\|^2_{L^\infty} \int
{\la u \ra^{-( 6 +2\gamma+ 2s-2m)}}    \int \frac{ {\bf 1}_{\la w \ra \leq \la u \ra}}{\la  w \ra^{2s+2m}}dw  du\,\,  \|G\|^2_{H^{s+m}} \|H\|_{H^{s-m}}^2 \\
\lesssim &
\|F\|^2_{L^1} \|G\|^2_{H^{s+m}} \|H\|_{H^{s-m}}^2 \,.
\end{align*}
If $ \gamma + 2s >-3/2$ then  with $u = \xi_* -\xi^-$ we have
\begin{align*}
|A_{2,1,p}^{(1)} | ^2  \lesssim&  \int |\hat f(\xi_*)|^2 \left\{
\sup_{u}
{\la u \ra^{-( 6 +2\gamma+ 2s-2m)}}    \int \frac{ {\bf 1}_{\la \xi^+ -u \ra \leq \la u \ra}}{\la  \xi^+ -u \ra^{2s+2m}}d\xi^+\right\} d\xi_*\,\,
 \|g\|^2_{H^{s+m}} \|h\|_{H^{s-m}}^2\\
\lesssim &
\|f\|^2_{L^2}  \|g\|^2_{H^{s+m}} \|h\|_{H^{s-m}}^2 \,,
\end{align*}
because $d\xi \sim d \xi^+$ on the support of ${\bf 1}_{|\xi^- |\leq |\xi|/\sqrt 2}$\,.
In the  case $\gamma +2s >-1$,
  by the H\"older inequality and the change of variables $u = \xi_* -\xi^-$ we have
\begin{align*}
|A_{2,1,p}^{(1)} | ^2  \lesssim& \left( \int |\hat f(\xi_*)|^3 d\xi_*\right)^{2/3} \\
&\times  \left(\int \left(
{\la u \ra^{-( 6 +2\gamma+ 2s-2m)}}    \int \frac{ {\bf 1}_{\la \xi^+ -u \ra \leq \la u \ra}}{\la  \xi^+ -u \ra^{2s+2m}}d\xi^+\right)^3
 du\right)^{1/3}
\,\,   \|g\|^2_{H^{s+m}} \|h\|_{H^{s-m}}^2 \\
\lesssim &
\|f\|^{2}_{L^{3/2}}  \|g\|^2_{H^{s+m}} \|h\|_{H^{s-m}}^2 \,.
\end{align*}
As for $A_{2,1,p}^{(2)}$ we have by the Cauchy-Schwarz inequality
\begin{align*}
|A_{2,1,p}^{(2)} | ^2  \lesssim& \iiint \frac{
|\hat \Phi_c (\xi_* - \xi^-) | |\hat f (\xi_* )|  }
{\la \xi_* -\xi^- \ra^{2s}}|{\la\xi-\xi_* \ra^{s+m}} \hat g(\xi -\xi_*)|^2 d \sigma d\xi d\xi_* \\
& \times \iiint \frac{
|\hat \Phi_c (\xi_* - \xi^-) | |\hat f (\xi_* )|  }
{\la \xi_* -\xi^- \ra^{2s}} |\la\xi \ra^{s-m}{\hat h} (\xi ) |^2 d\sigma d\xi d\xi_*\,.
\end{align*}
Since we have
\[
\iint \frac{
|\hat \Phi_c (\xi_* - \xi^-) | |\hat f (\xi_* )|  }
{\la \xi_* -\xi^- \ra^{2s}} d \xi_* d\sigma \lesssim \left \{ \begin{array}{ll}
\|f\|_{L^1} & \mbox{if} \enskip \gamma +2s >0\\
\|f\|_{L^2} & \mbox {if} \enskip \gamma +2s >-3/2\,,\\
\|f\|_{L^{3/2}} & \mbox {if} \enskip  \gamma +2s >-1\,,
\end{array} \right.
\]
we have the desired estimates for $A_{2,1,p}^{(2)}$.

On the set $A_{2,1,p}^{(3)}$ we have $\la \xi \ra \sim \la \xi - \xi_*\ra$. Hence
\begin{align*}
|A_{2,1,p}^{(3)} | ^2  \lesssim& \iiint b\,\, {\bf 1}_{ | \xi^- | \ge {1\over 2} \la \xi_*\ra }\frac{
|\hat \Phi_c (\xi_* - \xi^-) | |\hat f (\xi_* )|  }
{\la \xi\ra^{2s}}|{\la\xi-\xi_* \ra^{s+m}} \hat g(\xi -\xi_*)|^2 d \sigma d\xi d\xi_* \\
& \times \iiint b\,\, {\bf 1}_{ | \xi^- | \ge {1\over 2} \la \xi_*\ra }\frac{
|\hat \Phi_c (\xi_* - \xi^-) | |\hat f (\xi_* )|  }
{\la \xi \ra^{2s}} |\la\xi \ra^{s-m}{\hat h} (\xi ) |^2 d\sigma d\xi d\xi_*\,.
\end{align*}
We use the change of variables in $\xi_*$,  \, $u= \xi_* -\xi^-$.
Note that $| \xi ^-| \ge {1\over 2} \la u +\xi^-\ra $ implies  $|\xi^-| \geq \la u\ra/\sqrt {10}$.
If $\gamma + 2s >0$ then we have
\begin{align*}
\iint b\,\, {\bf 1}_{ | \xi^- | \ge {1\over 2} \la \xi_*\ra }\frac{
|\hat \Phi_c (\xi_* - \xi^-) | |\hat f (\xi_* )|  }
{\la \xi\ra^{2s}} d \sigma d\xi_*
&\lesssim \|\hat f\|_{L^\infty} \int \left (\frac{|\xi|}{\la u\ra}\right)^{2s}
\la u \ra^{-(3+\gamma)} \la \xi\ra^{-2s} du\\
&\lesssim \|\hat f\|_{L^\infty}\,.
\end{align*}
On the other hand, if $\gamma +2s >-3/2$ ( or  $0 \geq \gamma +2s >-1$ )
then this integral is upper bounded  by
\begin{align*}
&\iint b\,\, {\bf 1}_{ | \xi^- | \ge {1\over 2} \la \xi_*\ra }\frac{
|\hat \Phi_c (\xi_* - \xi^-) |}{\la \xi \ra^{2s/p} \la \xi_* -\xi^-\ra^{2s/q}}\frac{\la \xi_*\ra^{2s/q} |\hat f (\xi_* )|  }
{\la \xi \ra^{2s/q}}
 d \sigma d\xi_*
\\
&\leq \left( \iint b\, {\bf 1}_{ | \xi^- | \ge {1\over 2} \la \xi_* \ra }\frac{
|\hat \Phi_c (\xi_* - \xi^-) |^p}{\la \xi \ra^{2s} \la \xi_* -\xi^-\ra^{2sp/q}} d\sigma d \xi_*\right)^{1/p}
\left(\iint b\,   {\bf 1}_{ | \xi^- | \ge {1\over 2} \la \xi_*\ra }\frac{\la \xi_*\ra^{2s} |\hat f (\xi_* )|^q }
{\la \xi \ra^{2s}} d \sigma d\xi_* \right)^{1/q}\\
&\leq \left( \iint b\, {\bf 1}_{ | \xi^- |  \gtrsim \la u \ra }\frac{
|\hat \Phi_c (u) |^p}{\la \xi \ra^{2s} \la u \ra^{2sp/q}} d\sigma d u\right)^{1/p}
\|\hat f\|_{L^q} \lesssim \int \frac{du}{\la u \ra^{p(3+\gamma+2s)}} \|f\|^p\,,
\end{align*}
where $1/p +1/q =1, p= 2$ ( or $ p=3/2$).
Hence
we also obtain the desired estimates for $A_{2,1,p}^{(3)}$. The proof of the lemma is complete
\end{proof}

Proposition \ref{upper-c} is then a direct consequence of Lemmas
\ref{A-1} and \ref{A-2}, while the statements of Remark
\ref{weight-version123} are mentioned in the proof of the two
previous lemmas.

\subsection{Estimate of commutators with weights}
The following estimation on commutators will now be proved. Because
of the weight loss related to the Bolzmann equation, test functions
involve these weights, and therefore, this estimation is quite
necessary.
\begin{prop}\label{IV-prop-2.5}
Let $0<s<1, \gamma>\max\{-3, -2s -3/2\}$.  For any $ \ell, \beta, \delta \in \RR$
$$
\Big|\Big(W_\ell\,Q_c(f, g)-Q_c
(f, W_\ell g),\, h\Big)\Big|\lesssim
\|f\|_{L^2_{\ell-1-\beta-\delta}}
\|g\|_{H^{(2s-1+\epsilon)^+}_\beta}||h||
_{L^2_ \delta}.
$$
\end{prop}
The next two lemmas are a preparation for the complete proof of this
Proposition.
\begin{lemm} \label{estimates-facile}
If $\lambda <3/2$ then
\begin{align}\label{good}
\iint_{|v-v_*| \leq 1}  |f(v_*)|\, \frac{ |g(v)|^2}{|v-v_*|^\lambda}dv dv_* \lesssim \|f\|_{L^2} \|g\|^2_{L^2}\,.
\end{align}
If $3/2< \lambda <3$ then
$$
\iint_{|v-v_*| \leq 1}  |f(v_*)|\, \frac{ |g(v)|^2}{|v-v_*|^\lambda}dv dv_*  \lesssim \|f\|_{L^2} \|g\|_{H^{\frac{\lambda}{2}-\frac 34}}\,.
$$
\end{lemm}

\begin{proof}
Since $|v_*|^{-\lambda} {\bf 1}_{|v_*| \leq 1} \in L^2$ for $\lambda <3/2$, it follows from
the Cauchy-Schwarz inequality that if $\lambda < 3/2$ then
\begin{align*}
\iint_{|v-v_*| \leq 1}  |f_*|\, \frac{ |g|^2}{|v-v_*|^\lambda}dv dv_*
&\le \int |g|^2 \Big( \int_{|v-v_*| \leq 1}  |v- v_*|^{-2\lambda}dv_* \Big)^{1/2}
\Big(\int |f_*|^2 dv_*\Big)^{1/2}dv\\
&\lesssim \|f\|_{L^2} \|g\|^2_{L^2}\,.
\end{align*}
It follows from the Hardy-Littlewood-Sobolev inequality that if $3/2 < \lambda <3$ then
\begin{align*}
\iint_{|v-v_*| \leq 1}  |f_*|\, \frac{ |g|^2}{|v-v_*|^\lambda}dv dv_*
&\lesssim  \|f\|_{L^2}\,\,  \|g^2\|_{L^{p} }\enskip \enskip \mbox{with} \enskip \frac{1}{p} = \frac 3 2 - \frac{\lambda}{3} <1\\
&
\lesssim  \|f\|_{L^2}\,\,   \|g\|^2_{H^{\frac{\lambda}{2}-\frac 34}}\,
\end{align*}
because of the Sobolev embedding theorem.
\end{proof}

\begin{lemm} \label{estimates-square}
Let $0<s<1$ and $\gamma >\max\{-3, -2s -3/2\}$. Then
$$
\iiint b \Phi_c^{\gamma} |f(v_*)| |g(v)-g(v')|^2 dv dv_*d\sigma  \lesssim \|f\|_{L^2} \|g\|_{H^{s}}^2.
$$
\end{lemm}

\begin{proof}
Note that
\begin{align*}
\Big(Q_c^{\Phi_{\gamma}} (|f|, g),g\Big) =&- \frac{1}{2}\iiint b \Phi_c^{\gamma} |f_*| |g-g'|^2 dv dv_*d\sigma \\
&+ \frac{1}{2}\iiint b \Phi_c^{\gamma} |f_*| \Big( g'^2 -g^2\Big) dv dv_*d\sigma \,.
\end{align*}
Since Proposition \ref{upper-c} with $m=0$ is applicable to the left hand side,
it suffices to consider the second term of the right hand side.
It follows from the cancellation lemma of \cite{al-1} (more precisely the formula  (29)  there) that
\begin{align*}
\iiint b \Phi_c^{\gamma} |f_*| \Big( g'^2 -g^2\Big) dv dv_*d\sigma =  \int |f(v_*)| S(v-v_*) g(v)^2 dv dv_*\,,
\end{align*}
where
\begin{align*}
S(v-v_*) =& \Phi_{\gamma}(v-v_*)\phi(v-v_*) \Big( 2\pi \int_0^{\pi/2} b(\cos \theta) \sin \theta
\left(\frac{1}{\cos^{3+\gamma+1}( \theta/2)}- 1\right) d\theta\Big)\\
&+  \Phi_{\gamma}(v-v_*)\Big( 2\pi \int_0^{\pi/2}
\frac{b(\cos \theta) \sin \theta  }{\cos^{3+\gamma+1}( \theta/2)}\left(\phi\Big(
\frac{v-v_*}{\cos (\theta/2)}\Big) -  \phi(v-v_*) \right) d\theta\Big)\,.
\end{align*}
The integral of the second term on the right hand side can be written as $\tilde \phi(v-v_*)$ whose support
is contained in $\{ 0 < |v-v_*| \lesssim 1\}$.
Since  $s > -\gamma/2 -3/4$, the estimation for the first term just follows from Lemma
\ref{estimates-facile} because the case $\gamma = -3/2$ can be treated as $\gamma -\varepsilon$
for any small $\varepsilon >0$.
\end{proof}

\noindent{\em Proof of Proposition \ref{IV-prop-2.5}.} We write
\begin{align*}\label{J-estimate}
\Big(W_\ell Q_c(f,g) - Q_c(f, W_\ell g), h \Big)
=& \iiint b\Phi_c f_*'g'\Big(W_\ell(v) - W_\ell(v')
\Big) h dv dv_* d\sigma \notag\\
= &
 \iiint b\Phi_c \Big(W_\ell(v') - W_\ell(v)
\Big)  f_* g'h' dv dv_* d\sigma \\
&+ \iiint b\Phi_c\Big(W_\ell(v') - W_\ell(v)
\Big) f_* \Big(g-g'\Big) h'dv dv_* d\sigma \notag \\
=& J_1 + J_2\,. \notag
\end{align*}
Set $v_\tau = v + \tau(v'-v)$ for $\tau \in [0,1]$ and notice that
\begin{align*}
|W_\ell (v') -W_\ell (v) | \lesssim  \int_0^1 W_{\ell-1} (v_\tau) d\tau |v-v_*| \sin (\theta/2)\,.
\end{align*}
On the support of $\phi(v-v_*)$ we have for a large $C >0$
\begin{align*}
  &\la v_*\ra \lesssim  \frac{1}{2}   + \frac{1} {C}(|v_*|- |v -v_*| )\leq        \frac{1}{2}   + \frac{1} {C}(|v_*|- |v_\tau -v_*| )\leq \la v_\tau \ra \\
  &\leq 1+ |v_*| +|v_\tau-v_*| \leq 1 + |v_*| +|v-v_*|
\lesssim \la v_* \ra\,,
\end{align*}
so that $\la v_\tau \ra \sim \la v_* \ra \sim \la v \ra\sim \la v' \ra$.
The Cauchy-Schwarz inequality shows
\begin{align*}
|J_2|^2 &\lesssim \iiint b(\cos \theta)  \sin^{2s+2\varepsilon} (\theta/2)\Phi_c^{\gamma+2s} |\la v_* \ra^{\ell-1-\beta-\delta } f_*|
\, |\la v' \ra^{\delta} h'|^2 dv dv_* d\sigma\\
&\times \Big(\iiint b(\cos \theta)  \sin^{2 -2s-2\varepsilon}(\theta/2)\Phi_c^{\gamma+2-2s} |\la v_* \ra^{\ell-1-\beta-\delta} f_*||\la v\ra^\beta g- (\la v\ra^\beta g)'|^2 dv dv_* d\sigma\\
&\  + \iiint b(\cos \theta)  \sin^{2 -2s-2\varepsilon}(\theta/2)\Phi_c^{\gamma+2-2s} |\la v_* \ra^{\ell-1-\beta-\delta} f_*|
\big(\la v\ra^\beta - \la v'\ra^\beta\big)^2 |g'|^2 dv dv_* d\sigma\Big)\\
&= J_{2,1} \times \Big( J_{2,2}^{(1)} + J_{2,2}^{(2)}\Big)\,.
\end{align*}
Take  the change of variables $v \rightarrow v'$ for $J_{2,1} $.  Since $-(\gamma+ 2s) < 3/2$, it follows from \eqref{good} that
\[ J_{2,1} \lesssim \iint_{|v'-v_*|\lesssim 1}  |\la v_* \ra^{\ell-1-\beta-\delta } f_*|
\,\frac{ |\la v' \ra^{\delta} h'|^2 }{ |v'-v_*|^{-\gamma-2s}}  dv' dv_* \lesssim
\|f \|_{L^2_{\ell-1{-\beta}-\delta} }\|h\|^2_{L^2_\delta}
\]
Apply  Lemma \ref{estimates-square} with $s =( 2s-1+ \varepsilon )^+$ and $\gamma = \gamma +2-2s$
to $J_{2,2}^{(1)}$. Then
\[
J_{2,2}^{(1)} \lesssim
 \|f\|_{L^2_{\ell-1{-\beta}-\delta} } \|g\|^2 _{H^{(2s-1+ \varepsilon)^+}_\beta}
\]
because $
\max\{(2s-1 +\varepsilon)^+ , -
\frac{\gamma}{2}+s -1 -\frac 3 4\} = (2s-1 +\varepsilon)^+$. Since \eqref{good} also implies
\[
J_{2,2}^{(2)} \lesssim \iint_{|v'-v_*|\lesssim 1}   |v'-v_*|^{\gamma+4-2s} |\la v_* \ra^{\ell-1-\beta-\delta-2} f_*|
|\la v'\ra^\beta g'|^2 dv' dv_*  \lesssim \|f\|_{L^2_{\ell-3{-\beta}-\delta} } \|g\|^2 _{L^2_\beta}\,,
\]
we obtain the desired bound for $J_2$.
As for $J_1$ we use the Taylor expansion
\[
\la v' \ra^\ell - \la v \ra^\ell = \nabla \Big( \la v' \ra^\ell\Big)\cdot (v'-v)
+ \frac{1}{2} \int_0^1 \nabla^2 \Big( \la v_\tau \ra^\ell \Big)d\tau(v'-v)^2 \,.
\]
Then, it follows from the symmetry that the integral corresponding to the first term vanishes, so that
we have
\begin{align*}
\Big|J_1\Big|\lesssim & \left|\iiint b(\cos \theta)\Phi_c \sin^2( \theta/2)|v'-v_*|^{2}\nabla^2 \Big( \la v_\tau\ra^\ell\Big)| f_* g'h '|dv dv_* d\sigma \right|\\
\lesssim &  \iint_{|v'-v_*|\lesssim 1} |v'-v_*|^{\gamma+2} |\la v_* \ra^{\ell-2-\beta-\delta } f_*|
\, |(\la v \ra^{\beta} g)'| |(\la v \ra^{\delta}h)'|  dv' dv_* \,\\
\lesssim& \|f\|_{L^2_{\ell-2{-\beta}-\delta} } \|g\|^2 _{L^2_\beta}\|h\|^2 _{L^2_\delta}\,,
\end{align*}
which completes the proof of the of Proposition \ref{IV-prop-2.5}.

Now using \eqref{commutators-I} with $Q_{\bar c}(f, g)$ and the Proposition \ref{IV-prop-2.5}, we get
\begin{prop} Let $0<s<1, \gamma>\max\{-3, -2s -3/2\}$. For any $\ell\in\RR$,
\begin{equation}\label{IV-4.3}
\Big|\Big(W_\ell\,Q(f, g)-Q
(f, W_\ell g),\, h\Big)\Big|\lesssim
\|f\|_{L^2_{\ell+3/2+(2s-1)^++\gamma^++\epsilon}}
\|g\|_{H^{(2s-1+\epsilon)^+}_{\ell+(2s-1)^++\gamma^+}}
||h||_{L^2}.
\end{equation}
\end{prop}

We can now prove the upper bound estimate with weights.

\begin{prop}\label{prop-IV-2.5b}
Let $0<s<1, \gamma>\max\{-3, -2s-3/2\}$.
Then we have, for any $\ell\in\RR$ and
$m\in [s-1, s]$,
$$
\big | \big( Q(f,\, g) ,h \big) \big|\lesssim
\big(\|f\|_{L^1_{\ell^++(\gamma+2s)^+}}+
 \|f\|_{L^{2}}\big) \| g \|_{H^{\max\{s+m,\, (2s-1+\epsilon)^+\}}_{(\ell +\gamma+2s)^+}}
 \|h\|_{H^{s-m}_{-\ell}}\, .
 $$
\end{prop}

\begin{proof} Using \eqref{upper-I}, for any $m, \ell\in\RR$,
$$
\Big|\Big(Q_{\bar c} (f,  g),\, h \Big)\Big|  \lesssim ||f||_{L^1_{\ell^++(\gamma+2s)^+}} || g||_{H^{m+s}_{(\ell+\gamma+2s)^+}} \|h\|_{H^{s-m}_{-\ell}}\,.
$$
On the other hand, for any $\ell\in\RR$ we have
$$
\Big|\Big(Q_c(f, g),\, h \Big)\Big| \leq \Big|\Big(Q_c(f,  W_\ell g),\,  W_{-\ell} h \Big)\Big|
+ \Big|\Big(W_\ell Q_c(f, g) - Q_c(f, W_\ell g),\,   W_{-\ell}  h \Big)\Big|\,,
$$
then Proposition \label{weight-radja-version123} implies, for $m \in
[s-1,s]$
$$
\Big|\Big(Q_c(f,  W_\ell g),\,  W_{-\ell} h \Big)\Big|\lesssim
\|f \|_{L^2} \|g\|_{H^{s+m}_\ell} \|h\|_{H^{s-m}_{-\ell}}
$$
and Proposition \ref{IV-prop-2.5}, for any $\ell, \beta, \delta\in\RR$
$$
\Big|\Big(W_\ell Q_c(f, g) - Q_c(f,  W_\ell g),\, W_{-\ell}h\Big)\Big|
\lesssim
\|f \|_{L^2_{\ell-1-\beta-\delta}} \|g\|_{H^{(2s-1+\epsilon)^+}_\beta} \|W_{-\ell}h\|_{L^{2}_{\delta}}\,.
$$
We choose $\delta=0, \beta=\ell$, since for $m \in [s-1,s]$,
$s-m\geq 0$, ending the proof of Proposition.
\end{proof}

\subsection{Coercivity of collision operators}
We study now the coercivity estimate for a small perturbation of $\mu$.
For any $0<s<1$ and $\gamma>-3$, we recall the non-isotropic norm
associated with the cross-section $B(v-v_*, \sigma)$ introduced in
\cite{amuxy4-2}
\begin{equation}\label{trip-norm}
||| g|||^2_{\Phi_\gamma}= \int B(v-v_*, \sigma) \mu_*\, \big(g'-g\,\big)^2
 +\int B(v-v_*, \sigma)g^2_* \big(\sqrt{\mu'}\,\, - \sqrt{\mu}\,\big)^2=J_1(g)+J_2(g)
\end{equation}
where the integration  is
over~~$\RR^3_v\times\RR^3_{v_\ast}\times\SS^2_\sigma$. The following
link with weighted Sobolev norm was shown previously
\begin{align}\label{IV-2.2}
C_1 \left\{\left\|
g\right\|^2_{H^s_{\gamma/2}(\RR^3_v)}+\left\|
g\right\|^2_{L^2_{s+\gamma/2}(\RR^3_v)}\right\}
\leq ||| g |||^2_{\Phi_\gamma} \leq  C_2 \left\|
g\right\|^2_{H^s_{s+\gamma/2}(\RR^3_v)}\,,
\end{align}
where $C_1, C_2>0$ are two generic constants. Recall that in the definition
of the non-isotropic norm, we obtain an equivalent norm if we replace $\mu$ by any positive power of $\mu$.

The coercivity of the linearized
operator $-Q(\mu,\, h)$ is given by the next result
\begin{prop}\label{prop-coercivity-1}
There exists $C>0$ such that
$$
-\Big(Q(\mu, h),\,h\Big)_{L^2(\RR^3_v)}\geq \frac 12|||h|||^2_{\Phi_{\gamma}}-C\|h\|^2_{L^2_{\gamma/2+s}}\,.
$$
\end{prop}
\begin{proof} Though the statement follows from \cite{amuxy4-2}, we give a proof for the convenience of the reader. By definition,
\begin{align*}
-\Big(Q(\mu, h),\,h\Big)_{L^2(\RR^3_v)}
=&-\int B(v-v_*, \sigma)\big(\mu'_*h'-\mu_* h\big)hd\sigma d v_* dv\\
=&-\int B(v-v_*, \sigma)\mu_* h \big(h'- h\big)d\sigma d v_* dv
\\
=&\frac 1 2\int B(v-v_*, \sigma)\mu_*\big(h'- h\big)^2d\sigma d v_* dv
\\
&+\frac 12\int B(v-v_*, \sigma)\mu_*\big(h^2- h'^2\big)d\sigma d v_* dv\\
=&\frac 12 J_1(h)+I=\frac 12 |||h|||^2_{\Phi_\gamma}-\frac 12 J_2(h)+I\,.
\end{align*}
Then we have from \cite{amuxy4-2}
$$
|J_2(h)|\leq C_1\|h\|^2_{L^2_{\gamma/2+s}}\,,
$$
and the cancellation Lemma \cite {al-1} implies
$$
|I|\leq C\|h\|^2_{L^2_{\gamma/2+s}},$$
 thus proving proposition
\ref{prop-coercivity-1}.
\end{proof}

Let us note than another proof is also possible by using instead the Appendix.

\begin{lemm}\label{differ-Q-cD}
Let $0<s<1$ and $\gamma > \max \{-3, -2s -3/2\}$.
If we put
 \begin{align*}
 \cD(\sqrt{\mu} \,f ,\,g)  = \int B \big (\sqrt{\mu}\, f\big)_ * \Big( g- g'\Big)^2 dv dv_* d\sigma\,,
\end{align*}
then there exists a $C >0$ such that
\begin{align}\label{df}
\left|\Big( Q(\sqrt{\mu} \,f, g), g\Big)_{L^2{(\RR^3)}}+ \frac12 \cD(\sqrt{\mu}\, f, g) \right|\le C
\left\{ \begin{array}{lc}
 \|f\|_{L^2} \| g \|^2_{ L^2_{\gamma/2}}
&\mbox{if $\gamma >-3/2$}\\
 \|f\|_{L^2} \| g \|^2_{ H^{s'}_{\gamma/2}}
 &\mbox{if $-3/2 \ge \gamma$}\\
 \|f\|_{H^{2s'} } \| g \|^2_{ L^2_{\gamma/2 + s'}}
 &\mbox{if $-3/2 \ge \gamma$}
\end{array} \right.
\end{align}
for any $s' \in ]0,s[ $ satisfying $\gamma +2s' >-3/2$ and $s' < 3/4$.
\end{lemm}
\begin{proof}
The left hand side of \eqref{df} equals
\[ \frac{1}{2} \int B \sqrt{\mu}_* \,f_{*}
 \Big( g^2 - \big (g \big)'^2 \Big) d\sigma dvdv_*\,,
\]
and by the cancellation lemma  of \cite{al-1}
\begin{align*}
&\lesssim \int  \sqrt{\mu}_* \, |f_{*} | \int |v-v_*|^\gamma  g^2 dv dv_* = J(f,g)\,.
\end{align*}
Divide the integral to  $\{|v-v_*| \leq 1\}$ and another region, if necessary. Then
it follows from Lemma \ref{estimates-facile} that we obtain the first two estimates.
The third estimate is a direct consequence of  Pitt's inequality,
\begin{align*}
J(f,g)
&\lesssim \int \Big(\int|v-v_*|^{2(\gamma+2s')} \mu_{*} dv_*\Big)^{1/2}\Big(\int \frac{f_*^2}
{|v-v_*|^{4s'}}  dv_*\Big)^{1/2}
|g|^2 dv  \\
&\lesssim \|f\|_{H^{2s'}} \| g \|^2_{ L^2_{\gamma/2+s'}}\,,
\end{align*}
where we choose $0<2s'<3/2$.
\end{proof}

\begin{lemm}\label{upper-triple-Dirichlet}
Let $0<s<1$ and $\gamma >\max\{-3,\,-3/2-2s\}$.  Then for any $N \in \NN$ we have
\begin{align}\label{triple-dirichlet}
\cD(\sqrt \mu\,| f|\, , g) \lesssim \|f\|_{L^2_{-N}} |||g|||^2_{\Phi_\gamma}\,.
\end{align}
\end{lemm}
\begin{proof}
Put $F= \sqrt \mu |f|$. Then \eqref{triple-dirichlet} in the case $\gamma \ge 0$ follows from
Lemma 3.2 of \cite{amuxy4-2} with $f^2 = F$.
Suppose that $\gamma<0$.  In view of $\Phi_\gamma = \Phi_c + \Phi_{\overline c}$ \,, we write
\[
\cD(F,g) = \int b \Big(\Phi_c + \Phi_{\overline c}\Big)F_* \Big (g-g'\Big)^2 dv dv_* d \sigma=\cD_c(F,g)
+ \cD_{\overline c}(F,g)\,.
\]
We have
\begin{align*}
\cD_c(F,g) &\lesssim \int b\Phi_c \big( \la v_*\ra^{|\gamma|} F_* \big) \Big ( \la v \ra^{\gamma/2} g-
\la v' \ra^{\gamma/2} g'\Big)^2 dv dv_* d \sigma\\
& \quad + \int b\Phi_c \big( \la v_*\ra^{|\gamma|} F_* \big)\Big ( \la v'\ra^{\gamma/2} -
\la v \ra^{\gamma/2} \Big)^2 g^2 dv dv_* d \sigma\\
&= \cD_c^{(1)}(F,g) + \cD_c^{(2)}(F,g)\,,
\end{align*}
because $\la v_*\ra \sim \la v+ \tau(v'-v) \ra$ for any  $\tau \in [0,1]$ on the support of $\phi(v-v_*)$,
  as stated in the proof of
Proposition \ref{IV-prop-2.5}. Since
\[
\Big| \la v' \ra^{\gamma/2} -
\la  v \ra^{\gamma/2} \Big|
\lesssim \int_0^1 \la v+ \tau(v'-v) \ra^{\gamma/2-1}d\tau |v-v_*|\sin \theta/2 \lesssim \la v\ra^{\gamma/2}|v-v_*|\sin \theta/2
\]
on the support of $\phi(v-v_*)$, it follows from the Cauchy-Schwarz inequality that
\[\cD_c^{(2)}(F,g) \lesssim \int\Big(\int_{|v-v_*|\le 1} |v-v_*|^{\gamma+2}\la v_*\ra^{|\gamma|}F_* dv_*\Big)
\Big(\la v\ra^{\gamma/2}g\Big)^2 dv\lesssim \|F\|_{L^2_{|\gamma|}}\|g\|^2_{L^2_{\gamma/2}}\,.
\]
By means of Lemma \ref{estimates-square} we have $\cD_c^{(1)}(F,g) \lesssim \|F\|_{L^2_{|\gamma|}}\|g\|^2_{H^s_{\gamma/2}}$.
On the other hand, noting that $\Phi_{\overline c} \lesssim \la v-v_*\ra^\gamma \le \la v_* \ra^{|\gamma|}
 \la v \ra^{\gamma}$ we have
\begin{align*}
\cD_{\overline c}(F,g) &\lesssim \int b \big( \la v_*\ra^{|\gamma|} F_* \big) \Big ( \la v \ra^{\gamma/2} g-
\la v' \ra^{\gamma/2} g'\Big)^2 dv dv_* d \sigma\\
& \quad + \int b \big( \la v_*\ra^{|\gamma|} F_* \big)\Big ( \la v' \ra^{\gamma/2} -
\la v \ra^{\gamma/2} \Big)^2 g^2 dv dv_* d \sigma\\
&= \cD_{\overline c}^{(1)}(F,g) + \cD_{\overline c}^{(2)}(F,g)\,.
\end{align*}
It follows from Lemma 3.2 of \cite{amuxy4-2} with $\gamma =0$ together with Proposition 2.4 of  \cite{amuxy4-2}  that
\[
\cD_{\overline c}^{(1)}(F,g) \lesssim \|\la v \ra^{|\gamma|}F\|_{L^1_{2s}} |||\la v \ra^{\gamma/2}g|||^2_{\Phi_0}
\lesssim \|F\|_{L^1_{|\gamma|+2s} }|||g|||^2_{\Phi_\gamma} \,.
\]
Since with $v_\tau = v+ \tau (v'-v) $ we have
\begin{align*}
\Big| \la v' \ra^{\gamma/2} -
\la  v \ra^{\gamma/2} \Big|
&\lesssim \int_0^1 \la v_\tau \ra^{\gamma/2-1}d\tau |v-v_*|\sin \theta/2\\
& \lesssim \la v_*\ra^{|\gamma|/2 +1} \int_0^1 \la v_\tau -v_* \ra^{\gamma/2-1}|v-v_*|\sin \theta/2\\
&\lesssim \la v_*\ra^{|\gamma|/2 +1}\la v -v_* \ra^{\gamma/2-1}|v-v_*|\sin \theta/2
\lesssim  \la v_*\ra^{|\gamma| +1}\la v \ra^{\gamma/2}\sin \theta/2
\end{align*}
we see that
\[
 \cD_{\overline c}^{(2)}(F,g) \lesssim \int \la v_*\ra^{3|\gamma| +2}F_* \Big( \la v\ra^{\gamma/2} g\Big)^2 dv dv_*
 \lesssim \|F\|_{L^1_{3|\gamma|+2}}\|g\|^2_{L^2_{\gamma/2}}.
 \]
Summing up above four estimates, in view of \eqref{IV-2.2} we also obtain the desired estimate
\eqref{triple-dirichlet} in the case $\gamma <0$.
\end{proof}
By means of Lemma \ref{differ-Q-cD} and Lemma \ref{upper-triple-Dirichlet}, in view of \eqref{IV-2.2}
we get
the following upper bounded estimate, which  is needed in order to prove the
non linear coercivity for small perturbative solution.

\begin{prop}\label{prop-IV-2.5}
Let $0<s<1$ and $\gamma >\max\{-3,\,-3/2-2s\}$.  Then we have
$$
\big | \big( Q(\sqrt{\mu}\,g,\, h) ,h \big)_{L^2(\RR^3_v)} \big|\lesssim  \|g\|_{L^2} \, |||h|||^2_{\Phi_\gamma}\,.
$$
\end{prop}
\begin{rema}
If we proceed as in the proof of Proposition 3.1 from
\cite{amuxy4-2}, we can prove
$$
\big | \big( Q(\sqrt{\mu}\,g,\, f) ,h \big)_{L^2(\RR^3_v)} \big|\lesssim\Big\{ |||g|||_{\Phi_\gamma} + \| g \|_{W^{1, \infty}}\Big \}|||f|||_{\Phi_\gamma}|||h|||_{\Phi_\gamma}\,
$$
for any $\gamma>-3$.
\end{rema}

{}From Proposition \ref{prop-coercivity-1}, Proposition
\ref{prop-IV-2.5} and \eqref{IV-2.2}, we can deduce the following
non linear coercivity for the small perturbation $g$.
\begin{coro}\label{coro-coercivity-1}
Let $0<s<1$ and $\gamma >\max\{-3,\,-3/2-2s\}$.
There exist $\eta_0>0, \epsilon_0>0$ and $C>0$ such that if
$ 
 \|g\|_{L^2(\RR^3_v)}\leq \epsilon_0,
$
then we have
$$
\Big(-Q(\mu+\sqrt{\mu}\,g,\, h ),\,h\Big)_{L^2(\RR^3_v)}\geq  \eta_0|||h|||^2_{\Phi_{\gamma}}-C\|h\|^2_{L^2_{\gamma/2+s}}
\gtrsim \eta_0||h||^2_{H^s_{\gamma/2}}-C\|h\|^2_{L^2_{\gamma/2+s}}\,.
$$
\end{coro}

\subsection{Estimate of commutators with pseudo-differential operators}

We study now the commutators with pseudo-differential operators:
again in the next Sections, these will be used as a rigorous
replacement of formal derivatives, and when the operator is a
smoothed one, as completely justified test functions.
\begin{prop} \label{comu-deriv-2}
Let  $M_\lambda(\xi)= \la \xi \ra^\lambda$ for  $\lambda \ge 0$. Assume that  $0< s<1$ and   $\gamma>\max\{-3, -\frac 32 -2s\} $.
Let $\alpha, \beta,\rho  \ge 0$ satisfy
\begin{align}
&3+ \gamma +\alpha +\beta +\rho > 3/2 \, ,\label{111}\\
&\alpha + \beta \geq 2s -1\,,\label{333}\\
&\beta \le 1\,, \label{444}\\
& \alpha+ \beta+ \rho \ge s \,. \label{555}
\end{align}
If $\alpha + \lambda <3/2$ then  we have
\begin{align*}
\Big
| \Big (M_\lambda(D) \, Q_c (f,  g) - Q_c (f,  M_\lambda(D)\, g)  , h \Big ) \Big |
\lesssim  \| f\|_{{ H^\rho}} || M_\lambda(D)g||_{H^{ \alpha}} \, || h||_{H^{\beta}}\,.
\end{align*}
If $\alpha+\lambda \ge 3/2$ then $\rho= (\lambda - \beta)^+$ satisfies \eqref{111} and we have
\begin{align*}
\Big
| \Big (M_\lambda(D) \, Q_c (f,  g) - Q_c (f,  M_\lambda(D)\, g)  , h \Big ) \Big |
\lesssim  \| f\|_{{ H^{ (\lambda - \beta)^+}}} || M_\lambda(D)g||_{H^{ \alpha}} \, || h||_{H^{\beta}}\,.
\end{align*}
\end{prop}

\begin{proof}
We recall \eqref{equivalence-relation}, that is,
\begin{equation*}
\left \{ \begin{array}{ll}
\la \xi \ra \lesssim \la \xi_* \ra \sim \la \xi-\xi_*\ra   &\mbox{on supp ${\bf 1}_{\la \xi_* \ra\geq \sqrt 2 |\xi|}$}\\
\la \xi \ra \sim \la \xi-\xi_*\ra   &\mbox{on supp ${\bf 1}_{\la \xi_*\ra \leq |\xi |/2 } $} \\
\la \xi \ra \sim \la \xi_* \ra \gtrsim   \la \xi-\xi_*\ra  &
\mbox{on supp ${\bf 1}_{\sqrt 2 |\xi| \geq \la \xi_*\ra \geq | \xi|/2 }$\,.}
\end{array}
\right.
\end{equation*}
Since $M_\lambda(\xi)$ is increasing function of $|\xi|$,  we have
\begin{align}\label{mean-lambda}
\Big| M_\lambda(\xi) - M_\lambda(\xi-\xi_*)\Big|
&\lesssim M_\lambda(\xi-\xi_*)
{\bf 1}_{  \la \xi_* \ra \geq | \xi |/2 }  + \frac{\la \xi_* \ra}{\la \xi \ra} M_\lambda( \xi-\xi_* ) {\bf 1}_{ \la \xi_*\ra < | \xi |/2 } \\
&\quad
+ M_\lambda(\xi-\xi_*)
\frac{M_\lambda(\xi_*)}{ \la \xi-\xi_*\ra^\lambda           }{\bf 1}_{\la \xi- \xi_* \ra  \leq \la \xi_*\ra}
\,,\notag
\end{align}
where we have used the mean value theorem to gain the second term of the right hand side. Since we have
\begin{align*}
( Q_c(f, g), h ) =& \iiint b \Big({\xi\over{ | \xi |}} \cdot \sigma \Big) [ \hat\Phi_c (\xi_* - \xi^- ) - \hat \Phi_c (\xi_* ) ] \hat f (\xi_* ) \hat g(\xi - \xi_* ) \overline{{\hat h} (\xi )} d\xi d\xi_*d\sigma,
\end{align*}
it follows that
\begin{align*}
 &\Big (M_\lambda(D) \, Q_c (f,  g) - Q_c (f,  M_\lambda(D)\, g)  , h \Big )  \\
& = \iiint b \Big({\xi\over{ | \xi |}} \cdot \sigma \Big) \Big( \hat\Phi_c (\xi_* - \xi^- ) - \hat \Phi_c (\xi_* )
\Big)\Big( M_\lambda(\xi) - M_\lambda(\xi-\xi_*)\Big) \hat f (\xi_* ) \hat g(\xi - \xi_* ) \overline{{\hat h} (\xi )} d\xi d\xi_*d\sigma .\\
 & = \iiint_{ | \xi^- | \leq {1\over 2} \la \xi_*\ra }  \cdots\,\, d\xi d\xi_*d\sigma
+ \iiint_{ | \xi^- | \geq {1\over 2} \la \xi_*\ra } \cdots\,\, d\xi d\xi_*d\sigma \,\\
&= B_1(f,g,h)  +  B_2(f,g,h) \,\,.
\end{align*}
The estimations for  $B_1(f,g,h)$ and $B_2(f,g,h)$ are almost similar as those
for $A_1(f,g,h)$ and $A_2(f,g,h)$ in the proof of Lemma \ref{weight-radja-version123},
by adding the extra factor $M_\lambda(\xi) - M_\lambda(\xi-\xi_*)$. Indeed, for $B_{1,j}(f,g,h)$ corresponding to $A_{1,j}(f,g,h)$, ($j=1,2$),
note that  it follows from \eqref{later-use1}, \eqref{later-use2} and \eqref{mean-lambda} that
\begin{align*}
&\Big|\big(|K(\xi,\xi_*)| +|\tilde K(\xi,\xi_*) |\big)\big( M_\lambda(\xi) - M_\lambda(\xi-\xi_*)\big) \Big| \notag \\
&\lesssim  \frac{\la \xi_*\ra^\rho   \la \xi- \xi_*\ra^\alpha \la \xi\ra^\beta  }{\la \xi_*\ra^{3+\gamma +\alpha +\beta +\rho}}
 M_\lambda(\xi-\xi_*)
\left ( {\bf 1}_{  \la \xi_* \ra \geq | \xi |/2 }
                   +   \left (\frac{\la \xi_* \ra}{\la \xi \ra}\right)^{1+\alpha +\beta-2s} {\bf 1}_{ \la \xi_*\ra < | \xi |/2}
\right) \\
& +  \frac{\la \xi_*\ra^\rho   \la \xi- \xi_*\ra^\alpha \la \xi\ra^\beta  }{\la \xi_*\ra^{3+\gamma +\beta +\rho}}
 M_\lambda(\xi-\xi_*)      \left(       \frac{M_\lambda(\xi_*) }{{\la \xi- \xi_*\ra^{\alpha+\lambda}}}
+  \frac{1 }{{\la \xi- \xi_*\ra^{\alpha}}}\right)
{\bf 1}_{\la \xi- \xi_* \ra  \leq \la \xi_*\ra} \,,
\end{align*}
where we have estimated the factor $\la \xi \ra/\la \xi_* \ra$ of $K$  on the supp ${\bf 1}_{\la \xi_* \ra\geq \sqrt 2 |\xi|}$
by $(\la \xi \ra/\la \xi_* \ra)^\beta$ because of \eqref{444}.
Noting \eqref{333} we have
\begin{align*}
&\Big|\big(|K(\xi,\xi_*)| +|\tilde K(\xi,\xi_*) |\big)\big( M_\lambda(\xi) - M_\lambda(\xi-\xi_*)\big) \Big|\\
 &\lesssim
 \frac{\la \xi_*\ra^\rho   \la \xi- \xi_*\ra^\alpha \la \xi\ra^\beta  }{\la \xi_*\ra^{3+\gamma +\alpha +\beta +\rho}}
 M_\lambda(\xi-\xi_*)         \left\{
\left ( \frac{\la \xi_* \ra}{\la \xi \ra} \right)^{1+\alpha +\beta-2s}+
\left(       \frac{M_\lambda(\xi_*) }{{\la \xi- \xi_*\ra^{\alpha+\lambda}}}
+  \frac{1 }{{\la \xi- \xi_*\ra^{\alpha}}}\right)
{\bf 1}_{\la \xi- \xi_* \ra  \leq \la \xi_*\ra}
\right\}\,.
\end{align*}
The Cauchy-Schwarz inequality shows
\begin{align*}
&|B_{1,1}|^2+
|B_{1,2}|^2 \\
&\lesssim
\Big( \iint_{\RR^6} \Big|\big(|K(\xi,\xi_*)| +|\tilde K(\xi,\xi_*) |\big)\big( M_\lambda(\xi) - M_\lambda(\xi-\xi_*)\big) \Big|
|\hat f (\xi_* ) ||\hat g(\xi - \xi_* )|| \bar{\hat h} (\xi ) |d\xi d\xi_*\Big)^2 \\
&\lesssim \int_{\RR^3}
\frac{|\la \xi_*\ra^\rho\hat f(\xi_*)| }{\la \xi_*\ra^{3+\gamma +\alpha +\beta +\rho}} d \xi _* \|M_\lambda(D)g \|^2_{H^\alpha}
\int_{\RR^3}
\frac{\la \xi_*\ra^\rho|\hat f(\xi_*)| }{\la \xi_*\ra^{3+\gamma +\alpha +\beta +\rho}}d \xi_* \| h \|^2_{H^\beta} \\
&+
 \int_{\RR^6}
\frac{|\la \xi_*\ra^\rho \hat f(\xi_*)|^2  }{\la \xi_*\ra^{6+2(\gamma +\beta + \rho)}}
\left( \frac{M_\lambda(\xi_*)^2{\bf 1}_{ \la \xi -\xi *\ra \lesssim \la \xi_{ *}  \ra}}{
\la \xi-\xi_*\ra^{2(\alpha +\lambda)} }   +
 \frac{{\bf 1}_{ \la \xi -\xi *\ra \lesssim \la \xi_{ *}  \ra}}{
\la \xi-\xi_*\ra^{2\alpha } } \right)d\xi d\xi_*  \|M_\lambda(D)g \|^2_{H^\alpha}\| h \|^2_{H^\beta} \,.
\end{align*}
Since it follows from \eqref{111} that $\la \xi_*\ra^{-(3+\gamma +\alpha +\beta +\rho)} \in L^2$,
the first term has the upper bound $\|f\|_{H^\rho}^2  \|M_\lambda(D)g \|^2_{H^\alpha}\| h \|^2_{H^\beta}$.
If $\alpha+ \lambda <3/2$ then
\[
\int\frac{{\bf 1}_{ \la \xi -\xi *\ra \lesssim \la \xi_{*}  \ra}}{
\la \xi-\xi_*\ra^{2(\alpha+\lambda)} } d\xi  \lesssim \frac{1}{\la \xi_*\ra^{-3+2\lambda +2\alpha}}\,,
\]
which gives the same upper bound for the second term.
If $\alpha+\lambda \ge 3/2$ then the condition $\gamma >-3$ implies
\[
 \int_{\RR^6}
\frac{|\la \xi_*\ra^\rho \hat f(\xi_*)|^2  }{\la \xi_*\ra^{6+2(\gamma +\beta + \rho)}}
\left( \frac{M_\lambda(\xi_*)^2{\bf 1}_{ \la \xi -\xi *\ra \lesssim \la \xi_{ *}  \ra}}{
\la \xi-\xi_*\ra^{2(\alpha +\lambda)} }   +
 \frac{{\bf 1}_{ \la \xi -\xi *\ra \lesssim \la \xi_{ *}  \ra}}{
\la \xi-\xi_*\ra^{2\alpha } } \right)d\xi d\xi_*
\lesssim
\|\la \xi_*\ra^{\lambda-\beta} \hat f(\xi_*)\|^2_{L^2}\,.\]
Thus $B_{1,j}$ $(j=1,2)$ have the desired upper bound.

The estimation for  $B_{2,2}(f,g,h)$ are almost the same as above, in view of \eqref{A-2-2}.
As for $B_{2,1}$,  it remains only to estimate
\begin{align*}
B_{2,1,p}&=  \iiint b\,\,{\bf 1}_{\la \xi_* \ra \leq \sqrt 2 |\xi|}
 {\bf 1}_{ | \xi^- | \ge {1\over 2} \la \xi_*\ra }{\bf 1}_{| \xi_* \,\cdot\,\xi^-| \ge {1\over 2} | \xi^-|^2}
\hat \Phi_c (\xi_* - \xi^-)\Big( M_\lambda(\xi) - M_\lambda(\xi-\xi_*)\Big) \\
& \qquad \qquad \times
\hat f (\xi_* ) \hat g(\xi - \xi_* ) \overline{{\hat h} (\xi )} d\sigma d\xi d\xi_* \,.
\end{align*}
By
\[
{\bf 1}= {\bf 1}_{\la \xi_* \ra \geq |\xi|/2} {\bf 1}_{\la\xi-\xi_* \ra \leq \la \xi_* - \xi^- \ra}
+  {\bf 1}_{\la \xi_* \ra \geq |\xi|/2} {\bf 1}_{\la\xi-\xi_* \ra > \la \xi_* - \xi^-\ra}
+  {\bf 1}_{\la \xi_* \ra < |\xi|/2}
\]
we decompose
\begin{align*}
B_{2,1,p}
=
B_{2,1,p}^{(1)} + B_{2,1,p}^{(2)} +B_{2,1,p}^{(3)} \,.
\end{align*}
On the sets for
above integrals, we have $\la \xi_* -\xi^- \ra \lesssim \,
\la \xi_* \ra$, because $| \xi^- | \lesssim | \xi_*|$
that follows from  $| \xi^-|^2 \le 2 | \xi_* \cdot\xi ^-| \lesssim |\xi^-|\, | \xi_*|$.
Furthermore, on the sets for $B_{2,1,p}^{(1)}$ and $B_{2,1,p}^{(2)}$  we have $\la \xi \ra \sim \la \xi_* \ra$,
so that $\sup \Big(b\,\, {\bf 1}_{ | \xi^- | \ge {1\over 2} \la \xi_*\ra } {\bf 1}_{\la \xi_* \ra \geq |\xi|/2}\Big) \lesssim
{\bf 1}_{|\xi^- |\leq |\xi|/\sqrt 2}$ and
$\la \xi_* -\xi^- \ra \lesssim \,
\la \xi \ra$.
Hence we have
\begin{align*}
&|B_{2,1,p}^{(1)} | ^2 \\
& \lesssim \iiint \frac{
|\hat \Phi_c (\xi_* - \xi^-) |^2 |\la \xi_* \ra^\rho \hat f (\xi_* )|^2 }
{\la \xi_* -\xi^- \ra^{2\beta+ 2\rho}}
\left(\frac{M_\lambda(\xi_*)^2 {\bf 1}_{\la\xi-\xi_* \ra \leq \la \xi_* - \xi^- \ra}}{\la\xi-\xi_* \ra^{2\alpha+2\lambda}}
+\frac{{\bf 1}_{\la\xi-\xi_* \ra \leq \la \xi_* - \xi^- \ra}}{\la\xi-\xi_* \ra^{2\alpha}}\right)
d\xi d\xi_* d \sigma\\
& \times \iiint  |\la\xi-\xi_* \ra^{\alpha} M_\lambda(\xi-\xi_*)\hat g(\xi - \xi_* )|^2 |\la\xi \ra^{\beta}{\hat h} (\xi ) |^2 d\sigma d\xi d\xi_*\,.
\end{align*}
Note that  $\la \xi_*\ra \sim \la \xi \ra \sim \la \xi^+ \ra \lesssim \la \xi^+ -u\ra + \la u\ra$ with
$u = \xi_* - \xi^-$.
Since  $3+ 2(\gamma + \alpha +\beta +\rho) >0$ then if $\alpha+\lambda <3/2$ we have
\begin{align*}
&|B_{2,1,p}^{(1)} | ^2
\lesssim  \int |\la \xi_* \ra^{\rho} \hat f(\xi_*)|^2 \\
&\times \left\{
\sup_{u}
{\la u \ra^{-( 6 +2\gamma+ 2\beta + 2\rho)}}    \int \frac{ {\bf 1}_{\la \xi^+ -u \ra \leq \la u \ra}(
\la \xi^+ -u\ra^{2\lambda } + \la u\ra^{2\lambda})}{\la  \xi^+ -u \ra^{2(\alpha+\lambda)}}d\xi^+\right\} d\xi_*\,\,
 \|M_\lambda(D)g\|^2_{H^{\alpha}} \|h\|_{H^{\beta}}^2\\
&\lesssim
\|f\|^2_{H^\rho}  \|M_\lambda(D)g\|^2_{H^{\alpha}} \|h\|_{H^{\beta}}^2 \,
\end{align*}
because $d\xi \sim d \xi^+$ on the support of ${\bf 1}_{|\xi^- |\leq |\xi|/\sqrt 2}$\,.
The case $\alpha + \lambda \ge 3/2$ can be considered by the same arguments as above.
As for $B_{2,1,p}^{(2)}$, we first note that
$\xi^+ = \xi-\xi_* +u$ implies
\[ (
\enskip M_\lambda(\xi_*) \sim \enskip ) \enskip M_\lambda(\xi^+) \lesssim \la \xi -\xi_* \ra^\lambda + \la u\ra^\lambda \lesssim \la \xi -\xi_* \ra^\lambda
\]
on the integral set, and hence we have by the Cauchy-Schwarz inequality
\begin{align*}
|B_{2,1,p}^{(2)} | ^2  \lesssim& \iiint \frac{
|\hat \Phi_c (\xi_* - \xi^-) | |\la \xi_*\ra^\rho \hat f (\xi_* )|  }
{\la \xi_* -\xi^- \ra^{\alpha +\beta +\rho }}|{\la\xi-\xi_* \ra^{\alpha}}M_\lambda(\xi-\xi_*) \hat g(\xi -\xi_*)|^2 d \sigma d\xi d\xi_* \\
& \times \iiint \frac{
|\hat \Phi_c (\xi_* - \xi^-) | |\la \xi\ra^\rho \hat f (\xi_* )|  }
{\la \xi_* -\xi^- \ra^{\alpha +\beta +\rho}} |{\la \xi \ra^\beta \hat h} (\xi ) |^2 d\sigma d\xi d\xi_*\\
\lesssim &
\|f\|^2_{H^\rho}  \|M_\lambda(D)g\|^2_{H^{\alpha}} \|h\|_{H^{\beta}}^2
\end{align*}
because  $\hat \Phi_c (u ) \la u \ra^{-( \alpha +\beta +\rho)} \in L^2$.

On the integral  set of  $B_{2,1,p}^{(3)}$ we have $\la \xi \ra \sim \la \xi - \xi_*\ra$ and
\[
|M_\lambda(\xi) -M_\lambda(\xi-\xi_*) |\lesssim \frac{\la \xi_*\ra}{\la \xi \ra} M_\lambda(\xi -\xi_*)\,,
\]
so that
\begin{align*}
|B_{2,1,p}^{(3)} | ^2  \lesssim& \iiint b\,\, {\bf 1}_{ | \xi^- | \ge {1\over 2} \la \xi_*\ra }\frac{
|\hat \Phi_c (\xi_* - \xi^-) | |\la \xi_*\ra^\rho \hat f (\xi_* )|  }
{\la \xi\ra^{\alpha + \beta+1} \la \xi_* \ra^{\rho-1}} |{\la\xi-\xi_* \ra^{\alpha}}M_\lambda(\xi-\xi_*)  \hat g(\xi -\xi_*)|^2 d \sigma d\xi d\xi_* \\
& \times \iiint b\,\, {\bf 1}_{ | \xi^- | \ge {1\over 2} \la \xi_*\ra }\frac{
|\hat \Phi_c (\xi_* - \xi^-) | |\la \xi_* \ra^\rho \hat f (\xi_* )|  }
{\la \xi\ra^{\alpha + \beta+1} \la \xi_* \ra^{\rho-1}} |\la\xi \ra^{\beta}{\hat h} (\xi ) |^2 d\sigma d\xi d\xi_*\,.
\end{align*}
We use the change of variables in $\xi_*$,  \, $u= \xi_* -\xi^-$.
Note that $| \xi ^-| \ge {1\over 2} \la u +\xi^-\ra $ implies  $|\xi^-| \geq \la u\ra/\sqrt {10}$.
Since $\la \xi_*-\xi^- \ra + \la \xi_*\ra \lesssim \la \xi \ra$, in view of  \eqref{333} and \eqref{555} we have
\begin{align*}
&
\iint b\,\, {\bf 1}_{ | \xi^- | \ge {1\over 2} \la \xi_*\ra }\frac{
|\hat \Phi_c (\xi_* - \xi^-) | |\la \xi_* \ra^\rho \hat f (\xi_* )|  }
{\la \xi\ra^{\alpha + \beta+1} \la \xi_* \ra^{\rho-1}}d\sigma d\xi_*\\
&=
\iint b\,\, {\bf 1}_{ | \xi^- | \ge {1\over 2} \la \xi_*\ra }\frac{
|\la \xi_* - \xi^- \ra^s \hat \Phi_c (\xi_* - \xi^-) |}{\la \xi \ra^{s} \la \xi_* -\xi^-\ra^{\alpha + \beta +\rho}}
\frac{\la \xi_*\ra^s |\la \xi_*\ra^\rho\hat f (\xi_* )|  }
{\la \xi \ra^{s}} \\
&\qquad \qquad \qquad \times \left(\frac{\la \xi_*-\xi^-\ra^{\alpha+ \beta+\rho-s}  \la \xi_*\ra^{1-s-\rho}}{
\la \xi \ra^{\alpha+ \beta+1-2s}           }\right)
 d \sigma d\xi_*
\\
&\lesssim \left( \iint b\, {\bf 1}_{ | \xi^- |  \gtrsim \la u \ra }\frac{\la u\ra^{2s}
|\hat \Phi_c (u) |^2}{\la \xi \ra^{2s} \la u \ra^{2(\alpha +\beta+\rho)}} d\sigma d u\right)^{1/2}
\left(\iint b\,   {\bf 1}_{ | \xi^- | \ge {1\over 2} \la \xi_*\ra }\frac{\la \xi_*\ra^{2s} |\la \xi_*\ra^\rho \hat f (\xi_* )|^2 }
{\la \xi \ra^{2s}} d \sigma d\xi_* \right)^{1/2}\\
&\lesssim
\| f \|_{H^\rho} \,,
\end{align*}
from which we also can obtain the desired bound for $B_{2,1,p}^{(3)}$. The proof of the proposition is complete.
\end{proof}

We give an application of Proposition \ref{comu-deriv-2}.
Let $S \in C_0^\infty(\RR)$ satisfy $0 \leq S\leq 1$ and
$$
S(\tau) = 1,\quad  |\tau|\leq 1; \qquad S(\tau) =0, \quad |\tau| \geq 2.
$$
Set $\Lambda_N^\lambda(D_v) = M_\lambda(D_v)\,S_N(D_v)= \langle D_v\rangle^\lambda\,S_N(D_v)$.
\begin{coro}\label{IV-coro-2.15}
Assume that  $0< s<1$ and   $\gamma>\max\{-3, -\frac 32 -2s\} $.
If $0 \le \lambda <3/2$ then  for any $\max\{2s-1 , s/2\} \le s' <s$ satisfying $\gamma +2s' >-3/2$ we have
\begin{align*}
\Big
| \Big (\Lambda_N^\lambda\, Q (f,  g) - Q (f,  \Lambda_N^\lambda\, g)  , h \Big ) \Big |
\lesssim  \| f\|_{H^{s'}_{3/2+(2s-1)^++\gamma^++\epsilon}} || g||_{H^{\lambda}_{(2s-1)^++\gamma^+}}  \, || h||_{H^{s'}}\,
\end{align*}
and moreover if $\lambda \ge 3/2$ then we have the same estimate with
$\|f\|_{H^{s'}}$ replaced by $\|f\|_{H^{\lambda -s'}}$.
\end{coro}
\begin{proof} In the proof of Proposition \ref{comu-deriv-2}, instead of \eqref{mean-lambda} we use
\begin{align*}
\Big| \Lambda_N^\lambda(\xi) - \Lambda_N^\lambda(\xi-\xi_*)\Big|
&\lesssim M_\lambda(\xi-\xi_*)
{\bf 1}_{  \la \xi_* \ra \geq | \xi |/2 }  + \frac{\la \xi_* \ra}{\la \xi \ra} M_\lambda( \xi-\xi_* ) {\bf 1}_{ \la \xi_*\ra < | \xi |/2 } \\
&\quad
+ M_\lambda(\xi-\xi_*)
\frac{M_\lambda(\xi_*)}{ \la \xi-\xi_*\ra^\lambda           }{\bf 1}_{\la \xi- \xi_* \ra  \leq \la \xi_*\ra}
\,,\notag
\end{align*}
and choose
$$
(\alpha, \beta, \rho) = (0,s',s')  \enskip or \enskip  (0,s',\lambda- s')\,,
$$
we can get ,
\begin{align*}
\Big
| \Big (\Lambda_N^\lambda\, Q_c (f,  g) - Q_c (f,  \Lambda_N^\lambda\, g)  , h \Big ) \Big |
\lesssim  \| f\|_{H^{s'}} || M_\lambda (D_v)g||_{L^2} \, || h||_{H^{s'}}\, .
\end{align*}
On the other hand, using Proposition 2.9 of \cite{amuxy-nonlinear-3},
\begin{align*}
\Big| \Big (\Lambda_N^\lambda\, Q_{\bar c} (f,  g) &- Q_{\bar c} (f,  \Lambda_N^\lambda\, g)  , h \Big ) \Big |
= \Big| \Big (\Lambda^{-s'}\big(\Lambda_N^{\lambda}\, Q_{\bar c} (f,  g) - Q_{\bar c} (f,  \Lambda_N^\lambda\, g)\big),\, \Lambda^{s'} h \Big ) \Big |\\
&\leq \Big| \Big (\Lambda_N^{\lambda-s'}\, Q_{\bar c} (f,  g) - Q_{\bar c} (f,  \Lambda_N^{\lambda-s'}\, g),\, \Lambda^{s'} h \Big ) \Big |\\
&+\Big| \Big ( Q_{\bar c} (f,\, \Lambda_N^{\lambda-s'}\,g) - \Lambda_N^{-s'}\,Q_{\bar c} (f,  \Lambda_N^\lambda\, g)\big),\, \Lambda^{s'} h \Big ) \Big |\\
&\lesssim \| f\|_{L^2_{3/2+(2s-1)^++\gamma^++\epsilon}} ||g||_{H^{\lambda-s'+(2s-1)^+}_{(2s-1)^++\gamma^+}} \, || h||_{H^{s'}}\, ,
\end{align*}
which completes the proof of Corollary.
\end{proof}

\section{Full regularity of solutions}\label{sect-IV-3}
\smallskip \setcounter{equation}{0}

Let $f\,\in L^\infty([0, T[\,;\,\,{H}^{5}_\ell(\Omega\times\RR^3))$, for any
$\ell\in\NN$ be a solution of Cauchy problem \eqref{IV-1.1}. The
regularity of $f$ will now be considered. First of all, note that
$f\in C^1([0, T[;\, H^1_\ell(\Omega\times\RR^3))$ by using the equation.

For $\alpha\in\NN^6$, we recall the Leibniz formula
$$
\partial^\alpha\,Q(g, f)=\sum_{\alpha_1+\alpha_2=\alpha}
C^{\alpha_1}_{\alpha_2}Q(\partial^{\alpha_1}g,\,
\partial^{\alpha_2}f).
$$

Here and below,  $\phi$ denotes a cutoff  function satisfying
$\phi\in C^\infty_0$ and $0\leq \phi\leq 1$. Notation
$\phi_1\subset\subset\phi_2$ stands for two cutoff functions such
that $\phi_2=1$ on the support of $\phi_1$.

Take some smooth cutoff functions $\varphi,\, \varphi_2,
\varphi_3\in C^\infty_0(]T_1, T_2[)$ and $ \psi,\, \psi_2, \psi_3\in
C^\infty_0(K)$  such that $\varphi\subset\subset
\varphi_2\subset\subset \varphi_3$ and
$\psi\subset\subset\psi_2\subset\subset\psi_3$. Set
$f_1=\varphi(t)\psi(x) f$, $f_2=\varphi_2(t) \psi_2(x) f$ and
$f_3=\varphi_3(t) \psi_3(x) f$, so here we can suppose that $0<T<\infty$. For $\alpha\in\NN^6, |\alpha|\leq 5$, denote
$$
F=\partial^\alpha_{x, v} (\varphi(t)\psi(x) f)\in L^\infty(]T_1,
T_2[;\, L^2_\ell (\RR^6)).
$$
Then the Leibniz formula  yields the following equation :
\begin{equation}\label{5.3.1}
 F_t + v\,\cdot\,\partial_x {F }- Q(f,\,\, F)= G,
\enskip (t, x, v) \in \RR^7,
\end{equation}
where
\begin{equation}\label{5.3.2}\begin{array}{rcl}
G&=&\sum_{\alpha_1+\alpha_2=\alpha,\,\, 1\leq
|\alpha_1|}C^{\alpha_1}_{\alpha_2} Q\Big(\partial^{\alpha_1}{f_2},\,\,
\partial^{\alpha_2} f_1\Big)\Big\}+
\partial^\alpha\Big( \varphi_t \psi(x) f+ v\,\cdot\,\psi_x(x)
\varphi(t) f\Big)\\
&+&[\partial^\alpha,\,\,\,
v\,\cdot\,\partial_x](\varphi(t)\psi(x) f)\equiv  (A)+(B)+(C).\\
\end{array}
\end{equation}

Note carefully that a priori $F$ is not regular enough, and
therefore at that point, taking it as test function in the equation
of \eqref{5.3.1} is not allowed. This is one of the main
difficulties alluded to in the Introduction. Therefore, as in
\cite{amuxy-nonlinear-3}, we need to mollify $F$. This mollification process of
course complicates the analysis below, but is necessary if we want
to avoid formal proofs. The previous set of tools related to
commutators estimations will then be used. For this purpose, let
$S\in C^\infty_0(\RR)$ satisfy $0\leq S\leq 1$ and
$$
S(\tau)=1,\,\,\,\,|\tau|\leq 1;\,\,\,\,\,\,\,
S(\tau)=0,\,\,\,\,|\tau|\geq 2.
$$
Then
$$
S_N(D_x)S_N(D_v)=S(2^{-2N}|D_x|^2)S(2^{-2N}|D_v|^2)\, :\,\,
H^{-\infty}_\ell (\RR^6) \,\, \rightarrow\,\, H^{+\infty}_\ell  (\RR^6),
$$
is a regularization operator such that
$$
\|\big(S_N(D_x)S_N(D_v)f\big) -f\|_{L^2_\ell  (\RR^6)}\rightarrow 0,
\,\,\,\,\,\, \mbox{ as } N\rightarrow \infty.
$$
Set
$$
P_{N,\, \ell}=\psi_2(x) W_{\ell  }\,S_N(D_x)\, S_N(D_v).
$$
Then
$$
P_{N,\, \ell}\,\,F \in C^1_0(]T_1, T_2[; C^\infty_0(K; H^{+\infty}(\RR^3))),
$$
and we can take
$$
h=P^\star_{N,\, \ell}\,\,( P_{N,\, \ell}\,\,F) \in C^1(\RR; H^{+\infty}(\RR^6))
$$
as a test function for equation (\ref{5.3.1}).

It follows by integration by parts on $\RR^7=\RR^1_t\times\RR^3_x\times\RR^3_v$ that
\begin{align*}
&\Big([S_N(D_v),\,\, v]\,\cdot\,\nabla_x\, S_N(D_x) F,\,\, W_{\ell} P_{N, \ell}\,\,F\Big)_{L^2(\RR^7)}\\
&-\Big(P_{N,
\ell}\,Q({f}_2,\,  F),\, P_{N, \ell}\,\, F\Big)_{L^2(\RR^7)}
= \Big(G,\, h\Big)_{L^2(\RR^7)},
\end{align*}
where we used the fact
$$
\Big((\partial_t+v\,\cdot\,\nabla_x)\,  P_{N, \ell} F,\,\,  P_{N, \ell}\,\,F\Big)_{L^2(\RR^7)}=0.
$$
We get then
\begin{eqnarray*}
&&-\Big(Q({f},\, P_{N,\, \ell}\,F),\, P_{N,\, \ell}\,
F\Big)_{L^2(\RR^7)}= -\Big([S_N(D_v),\,\, v]\,\cdot\,\nabla_x\,
S_N(D_x) F,\,\, W_{\ell} P_{N, \ell }\,F\Big)_{L^2(\RR^7)}
\\
&&\,\,\,\,\,+\Big(P_{N,\,\ell}\,Q(f_2,\,  F)-Q({f}_2,\, P_{N,\,
\ell}\, F),\, P_{N,\, \ell}\,\,F\Big)_{L^2(\RR^7)}+ \Big(G,\, h\Big)_{L^2(\RR^7)}.
\end{eqnarray*}

Next, we follow the main steps in our previous works \cite{amuxy-nonlinear-3},
but need to be careful due to the singular behavior of the relative
velocity part of the kernel.
\subsection{Gain of regularity in $v$}
In this subsection, we will prove
a  partial smoothing effect  on the weak solution $F$ in
the velocity variable $v$ .

\begin{prop}\label{propIV-6.1}
Assume that $0<s<1, \gamma>\max\{-3, -2s-3/2\}$. Let $f\in L^\infty([0, T]; H^5_\ell (\Omega_x\times\RR^3_v))$, for any $\ell\in\NN$  be a solution of the equation \eqref{IV-1.1} satisfying the coercivity condition \eqref{coercivity-nonlinear-a}. Then one has
\begin{equation}\label{3.3}
\Lambda^s_v \big(\varphi(t)\psi(x)f\big)\in L^2(\RR_t; H^5_\ell(\RR^6)),
\end{equation}
for any big $\ell>0$ and any cut off function $\varphi\in C^\infty_0(]T_1, T_2[), \psi\in C^\infty_0(K)$.
\end{prop}

Similarly to \cite{amuxy-nonlinear-3},
$$
\Big|\Big([S_N(D_v),\,\, v]\,\cdot\, \nabla_x\,\, S_N(D_x) \,F,\,\, W_\ell P_{N, \ell  }\,\,F\Big)_{L^2(\RR^7)}\Big|\leq C \|f_1\|^{2}_{L^2([0, T]; H^5_{
\ell}(\RR^6))}.
$$
Then the coercivity assumption  \eqref{coercivity-nonlinear-a} implies
\begin{align}\label{coercivity-2}
&\eta_0||\Lambda^s_v W_{\gamma/2} P_{N,\, \ell}\,F||^2_{L^2(\RR^7)} \leq
C\|f_1\|^2_{L^2(\RR; H^5_{\ell+(\gamma/2+s)^+}(\RR^6))}
+ \Big|\Big(G,\, h\Big)_{L^2(\RR^7)}\Big|\\
&+\Big|\Big(P_{N,\,\ell}\,Q(f_2,\,  F)-Q(f_2,\, P_{N,\,
\ell}\, F),\, P_{N,\, \ell}\,\,F\Big)_{L^2(\RR^7)}\Big|.\notag
\end{align}

The proof of Proposition \ref{propIV-6.1} will be completed by
estimating the last two terms in \eqref{coercivity-2} through the
following three Lemmas.

\begin{lemm}\label{lemmIV-6.1}
Let $f_1\in L^\infty([0, T]; H^5_\ell (\RR^6)),\, \ell \geq 0$. Then, we have, for any $\varepsilon>0$,
$$
\left|\Big(G,\, h\Big)_{L^2(\RR^7)}\right|\leq
C_\varepsilon\|f_2\|^{4}_{L^\infty([0, T]; H^5_{\ell+2+(\gamma+2s)^+} (\RR^6))}+\varepsilon
\| P_{N,\, \ell}\, F\|^2_{L^2(\RR^4_{t, x}, H^s_{\gamma/2}(\RR^3_{v}))}.
$$
\end{lemm}

\begin{proof} By using the decomposition in \eqref{5.3.2}, it is obvious that
$$
(B)=\partial^\alpha\Big( \varphi_t \psi(x) f+
v\,\cdot\,\psi_x(x) \varphi(t) f\Big)\in L^2_\ell (\RR^7),
$$
and
$$
\|(B)\|_{L^2_\ell(\RR^7)}\leq C \|f_2\|_{L^\infty(\RR, H^5_{\ell+1}(\RR^6))}.
$$
Note that $[\partial^\alpha,\,\,\, v\,\cdot\,\partial_x]$ is a
differential operator of order $|\alpha|$ so that  we have
$$
\|(C)\|_{L^2_\ell(\RR^7)}\leq C \|f_2\|_{L^\infty(\RR, H^5_{\ell}(\RR^6))}.
$$
For the term $(A)$, recall that $\alpha_1+\alpha_2=\alpha$,
$|\alpha|\leq 5$ and $|\alpha_2|\leq 4$. Here we use the following upper
bounded estimate from Proposition \ref{prop-IV-2.5b}
\begin{equation}\label{IV-2.5b}
\Big|\Big(Q(f,\, g),\, h\Big)_{L^2(\RR^3_v)}\Big|\leq
C\|f\|_{L^2_{\ell-\gamma/2+2+(\gamma+2s)^+}(\RR^3_v)}
\|g\|_{H^s_{\ell-\gamma/2+(\gamma+2s)^+}(\RR^3_v)}
\|h\|_{H^s_{-\ell+\gamma/2}(\RR^3_v)}.
\end{equation}

Then, by separating the cases $|\alpha_1|\leq 3$ and
$|\alpha_1|> 3$, we get
\begin{align*}
&\Big|\Big(Q\big(\partial^{\alpha_1}
{f_2},\,\partial^{\alpha_2} f_1\big),\,P^\star_{N, \ell}
\, P_{N, \ell} F\Big)_{L^2(\RR^7)}\Big|\\
&\leq C\int \|\partial^{\alpha_1}
{f_2}\|_{L^2_{\ell-\gamma/2+2+(\gamma+2s)^+}(\RR^3_v)}
\|\partial^{\alpha_2} f_1\|_{H^s_{\ell-\gamma/2+(\gamma+2s)^+}(\RR^3_v)}
\|P^\star_{N, \ell}
\, P_{N, \ell} F\|_{H^s_{-\ell+\gamma/2}(\RR^3_v)} dx dt\\
&\leq C\|
{f_2}\|_{L^\infty(\RR, H^5_{\ell+2+(\gamma+2s)^+}(\RR^6))}
\|f_1\|_{L^2(\RR, H^{4+s}_{\ell+2+(\gamma+2s)^+}(\RR^6))}
\| P_{N, \ell} F\|_{L^2(\RR^4_{t, x}, H^s_{\gamma/2}(\RR^3_v))}\,.
\end{align*}
Here we used the fact that $W_{-\ell}\,P^\star_{N, \ell}$ is a
uniformly (with respect to $N, \ell$) bounded operator.
This ends the proof of Lemma \ref{lemmIV-6.1} by Cauchy Schwarz
inequality.
\end{proof}

We turn now to estimating the commutators of the regularization
operator with the collision operator which are given in the
following two Lemmas.

The next lemma is about the commutator of the collision operator
with a mollification w.r.t. $x$ variable.

\begin{lemm}\label{lemm3.3}
Let $0< s <1,\, \gamma>\max\{-3, -2s-3/2\}$. For any suitable functions $f$ and $h$
with the following norms well defined, one has
\begin{eqnarray}\label{3.10+1}
&&\Big|\Big(S_N(D_x) Q(f,\,\, g )-Q(f,\,\, S_N(D_x)\, g),\,\,h\Big)_{L^2(\RR^7)}\Big| \\
&&\,\,\,\,\,\leq C 2^{-N} \|\nabla_x f\|_{L^\infty(\RR^4_{t,
x},\,\, L^2_{\ell+2-\gamma/2+(2s+\gamma)^+}(\RR^3_v))}
\|g\|_{L^2(\RR^4_{t, x},\,\,
H^s_{\ell-\gamma/2+(2s+\gamma)^+}(\RR^3_v))}
\|h\|_{L^2(\RR^4_{t, x},\,\,
H^s_{-\ell+\gamma/2}(\RR^3_v))}.\nonumber
\end{eqnarray}
for a constant $C$ independent of $N$.
\end{lemm}

\begin{proof} Let us introduce $\tilde{K}_N(z)=
2^{3N}\hat{S}(2^N z)2^N z$. Note that $\tilde{K}_N \in L^1(\RR^3)$
uniformly with respect to $N$. Then for any smooth function $\tilde{h}$, one has
\begin{eqnarray*}
&&\Big(\big(S_N(D_x) \, Q(f,\,\,g)-Q(f,\,\,S_N(D_x)
g)\big),\,\, h\Big)_{L^2(\RR^7)}
= \int_0^1\Big\{ \int_{\RR_t} \int_{\RR_x^3 \times \RR_y^{3}}
\tilde{K}_N(x-y)\\
&&\,\,\,\,\,\,\,\,\,\,\,\times \Big(Q\big(\nabla_x f(t,+\tau(x-y),\, \cdot\,),\,\,
2^{-N}g(t, y, \, \cdot\,)\big),\,\, h(t,x, \,
\cdot\,)\Big)_{L^2(\RR_v^3)}dtdxdy\Big\} d\tau .
\end{eqnarray*}
By applying \eqref{IV-2.5b}, the right hand side of
this equality can be estimated from above by
\begin{align*}
&C\Big\{\sup_{t,x} ||\nabla_x f(t, x,
\,\cdot\,)||_{L^2_{\ell+2-\gamma/2+(2s+\gamma)^+}(\RR^3_v)}
\Big \}\times \\
&\qquad \int_{\RR_t}\int_{
\RR_x^3}\big(|\tilde{K}_N|*||2^{-N} g (t,\,\cdot\,)
||_{H^s_{\ell-\gamma/2+(2s+\gamma)^+}(\RR^3_v)}
\big)(x)||h(t, x, \cdot)||_{H^s_{-\ell+\gamma/2}(\RR^3_v)}dxdt\\
&\qquad
\leq C 2^{-N}\|\nabla_x f\|_{L^\infty(\RR^4_{t, x}; L^2_{\ell+2-\gamma/2+(2s+\gamma)^+}(\RR^3_v))}
\|g\|_{L^2(\RR_{t,x}^4;\, H^s_{\ell-\gamma/2+(2s+\gamma)^+}(\RR^3_v))}
||h||_{L^2(\RR_{t,x}^4;
\, H^s_{-\ell+\gamma/2}(\RR^3_v))},
\end{align*}
which completes the proof of the lemma.
\end{proof}

We now apply \eqref{3.10+1} with $g\,\sim\,S_N(D_v)g$, and use the
fact that a regularization operator $S_N(D_v)$ w.r.t. $v$ variable
has the property that, for any $p$
$$
\|2^{-N}S_N(D_v)g(t, x, \,\cdot\,) \|_{H^s_{p}(\RR^3_v)}\leq \|2^{-N}S_N(D_v)g(t, x, \,\cdot\,) \|_{H^1_{p}(\RR^3_v)}\leq C
\|g(t, x, \,\cdot\,)\|_{L^2_{p}(\RR_v^3)},
$$
where $C$ is a constant independent on $N$. It follows that
\begin{eqnarray}\label{3.10+1b}
&&\Big|\Big(S_N(D_x) Q(f,\,\, S_N(D_v)g )-Q(f,\,\, S_N(D_x)\, S_N(D_v)g),\,\,h\Big)_{L^2(\RR^7)}\Big| \\
&&\,\,\,\,\,\leq C  \|\nabla_x f\|_{L^\infty(\RR^4_{t,
x},\,\, L^2_{\ell+2-\gamma/2+(2s+\gamma)^+}(\RR^3_v))}
\|g\|_{L^2(\RR^4_{t, x},\,\,
L^2_{\ell-\gamma/2+(2s+\gamma)^+}(\RR^3_v))}
\|h\|_{L^2(\RR^4_{t, x},\,\,
H^s_{-\ell+\gamma/2}(\RR^3_v))}.\nonumber
\end{eqnarray}

\smallbreak\noindent
{\bf Completion of proof of Proposition \ref{propIV-6.1}.}

As regards the commutator terms in \eqref{coercivity-2}, we have
\begin{eqnarray*}
&&\Big(P_{N, \ell }\, Q({f}, F)-Q({f},\, P_{N, \ell }\, F), \,\, P_{N,
\ell}\,F\Big)_{L^2(\RR^7)}\\
&=&\Big(S_{N}(D_v)\,Q({f}, F)-Q({f},\, S_{N}(D_v)\, F), \,\,
S^\star_N(D_x)\psi_2(x) W_\ell P_{N,
\ell }\,F\Big)_{L^2(\RR^7)}\\
&+&\Big(S_{N} (D_x)\, Q({f}, S_{N}(D_v)\,F)-Q({f},\, S_{N}
(D_x)S_{N}(D_v)\, F), \,\,\psi_2(x) W_\ell  P_{N, \ell
}\,F\Big)_{L^2(\RR^7)}\nonumber\\
&+&\Big( W_\ell \, Q({f},  S_{N}
(D_x)S_{N}(D_v)\,F)-Q({f},\, P_{N, \ell }\, F), \,\, \psi_2(x) P_{N, \ell
}\,F\Big)_{L^2(\RR^7)}\,.
\\
&=&(1)+(2)+(3).
\end{eqnarray*}
Note that $[W_\ell,\,\,S_N(D_v)]$ is also a uniformly
bounded operator from $L^2$ to $L^2_{\ell-1}$ with respect to the parameter $N$.

Using Corollary \ref{IV-coro-2.15} with $\lambda=0$, we have, for
$0<s'<s, \gamma+2s'>-3/2$,
\begin{align*}
|(1)|&\leq C \|f_2\|_{L^\infty(\RR^4_{t, x}\, ,\,\, H^{s'}_{3/2+\gamma^+
+(2s-1)^+}(\RR^3_v))}\|F\|_{L^2_{(\gamma^+ +2s-1)^+} (\RR^7)} \|W_{\ell}P_{N,\,
\ell}\,F\|_{L^2(\RR^4_{t, x}\, ,\,\,
H^{s'}(\RR^3_v))}
\\
&\leq \varepsilon\|\Lambda^{s}_v W_{\gamma/2} P_{N, \ell }\,
g\|^2_{L^2(\RR^7)}+ C_\varepsilon\|f_3\|^{4}_{H^5_{2\ell}(\RR^7)}.
\end{align*}
We can use \eqref{3.10+1b} to show that
\begin{align*}
|(2)|&\leq \leq C  \|\nabla_x f_2\|_{L^\infty(\RR^4_{t,
x},\,\, L^2_{\ell+2-\gamma/2+(2s+\gamma)^+}(\RR^3_v))}
\|F\|_{L^2(\RR^4_{t, x},\,\,
L^2_{\ell-\gamma/2+(2s+\gamma)^+}(\RR^3_v))}
\|W_{\ell}P_{N,\,
\ell}\,F\|_{L^2(\RR^4_{t, x},\,\,
H^s_{-\ell+\gamma/2}(\RR^3_v))}\\
&\leq \varepsilon \|\Lambda^s_v W_{\gamma/2} P_{N,\,
\ell}\, g\|^2_{L^2(\RR^7_{t, x, v})}+ C_\varepsilon\|f_3\|^{4}_{H^5_{2 \ell}(\RR^7)}\,.
\end{align*}
Finally, \eqref{IV-4.3}) implies that
\begin{align*}
|(3)| &\leq C \|f_2\|_{L^\infty(\RR^4_{t, x}\, ,\,\,
L^2_{\ell+2+(2s-1)^++\gamma^+}(\RR^3_v))}\|S_{N}(D_x)\,
S_{N}(D_v)\,F\|_{L^2(\RR_{t,x}^4\, ,\,
H^{(2s-1+\delta)^+}_{\ell+(2s-1)^++\gamma^+}(\RR_v^3))} \|P_{N,\, \ell }\,F\|_{L^2(\RR^7)}\\
&\leq \varepsilon \|\Lambda^s_v W_{\gamma/2} P_{N,\,
\ell}\, F\|^2_{L^2(\RR^7_{t, x, v})}+ C_\varepsilon\|f_3\|^{4}_{H^5_{2 \ell}(\RR^7)}\,.
\end{align*}
In summary, we have obtained the following estimate for the second
term on the right hand side of \eqref{coercivity-2}
\begin{align*}
&\Big|\Big(P_{N, \ell  }\, Q({f}_2, F)-Q({f}_2,\, P_{N, \ell  }\,
F), \,\, P_{N, \ell  }\,g\Big)_{L^2(\RR^7)}\Big|\\
&\leq \varepsilon \|\Lambda^s_v W_{\gamma/2} P_{N,\,
\ell}\, F\|^2_{L^2(\RR^7_{t, x, v})}+ C_\varepsilon\|f_3\|^{4}_{H^5_{2 \ell}(\RR^7)}\,.
\end{align*}

Finally, it holds that from \eqref{coercivity-2} and
\eqref{IV-2.2} that
$$
\|\Lambda^{s}_v W_{\gamma/2} P_{N, \ell}\, F\|^2_{L^2(\RR^7)}\leq C\Big(1+
\|f_3\|^{4}_{H^5_{2\ell}(\RR^7)}\Big),
$$
where the constant $C$ is independent of  $N$. Therefore,
Proposition \ref{propIV-6.1} is proved by taking the limit
$N\rightarrow\infty$.

\subsection{Gain of regularity in $(t, x)$}

In \cite{amuxy-nonlinear-b}, by using a generalized uncertainty
principle, we proved a hypo-elliptic estimate, as regards a
transport equation in the form of
\begin{equation}\label{2.1}
 f_t + v\cdot\nabla_x f = g \in D '({\RR}^{2n+1}) ,
\end{equation}
where $(t,x,v) \in \RR^{1 + n +n} = \RR^{2n +1}$.

\begin{lemm}\label{lemm2.1}
Assume that $ g \in H^{-s'} (\RR^{2n+1})$, for some $0\leq s' <1$.
Let $f\in L^2 (\RR^{2n+1}) $ be a weak solution of the transport
equation (\ref{2.1}), such that $\Lambda^s_v\, f \in L^2
(\RR^{2n+1})$ for some $0<s\leq 1$. Then it follows that
\[
\Lambda_x^{s (1-s')/(s +1)}f\in L^2_{-\frac{s s'}{s +1}}(\RR^{2n+1})
, \enskip \enskip \Lambda_t^{s (1-s')/(s +1)}f\in L^2_{-\frac{s}{s
+1}}(\RR^{2n+1}),
\]
where $\Lambda_\bullet=(1+|D_\bullet|^2)^{1/2}$.
\end{lemm}

As mentioned earlier, this hypo-elliptic estimate together with Proposition \ref{propIV-6.1} are used to obtain
the partial regularity in the variable $(t,x)$.
With this partial regularity in $(t,x)$, by
applying the Leibniz type estimate on the fractional differentiation
on the solution, we will show some improved regularity in
all variables, $v$ and $(x,t)$. Then the hypo-elliptic
estimate can be used again to get higher regularity in
the variable $(x,t)$. This procedure can be continued to
obtain at least one order higher differentiation regularity
in $(t,x)$ variable.

To proceed, recall (see for example \cite{amuxy-nonlinear-3} a Leibniz type
formula for fractional derivatives with respect to variable $(t,
x)$. Let $0<\lambda<1$. Then there exists a positive constant
$C_\lambda\neq 0$ such that for any $f\in \cS(\RR^n)$, one has
\begin{equation}\label{4.6}
|D_{y}|^{\lambda} f(y)=\cF^{-1}
\big(|\xi|^{\lambda}\hat{f}(\xi)\big)
=C_\lambda\int_{\RR^n}\frac{f(y)-f(y+h)}{|h|^{n+\lambda}} d h.
\end{equation}

\smallbreak
First of all, we have the following  proposition on the gain of regularity
in the variable $(t,x)$ through uncertainty principle as in \cite{amuxy-nonlinear-3}.

\begin{prop}\label{prop4.1}
Under the hypothesis of Theorem \ref{theo-IV-1.1}, one has
\begin{equation}\label{4.1}
\Lambda^{s_0}_{t, x}\, f_1\in L^2([0, T]; H^5_\ell(\RR^6)),
\end{equation}
for any $\ell\in\NN$ and $0<s_0=\frac{s(1-s)}{(s+1)}$.
\end{prop}

Therefore, under the hypothesis $f\in L^\infty([0, T]; H^5_\ell(\RR^6))$ for all $\ell\in\NN$, it follows
that for any $ \ell\in\NN$ we have
\begin{equation}\label{4.2}
\Lambda^{s}_{v}(\varphi(t)\psi(x) f) \, \in L^2([0, T]; H^5_\ell(\RR^6)),\,\, \hskip
0.5cm \Lambda^{s_0}_{t, x}(\varphi(t)\psi(x) f) \, \in L^2([0, T]; H^5_\ell(\RR^6))\, .
\end{equation}

\smallbreak This partial regularity in $(t, x)$ variable will now be
improved.

\begin{prop}\label{IV-prop4.2}
Let $0<\lambda< 1$. Suppose that $f\in L^\infty([0, T]; H^5_\ell(\Omega\times\RR^3))$ for all $\ell\in\NN$ is a solution of the equation (\ref{IV-1.1}), and for any cutoff functions
$\varphi, \psi$, we have
\begin{equation}\label{4.1+0}
\Lambda^{s}_{v}(\varphi(t)\psi(x) f) \, \in L^2([0, T]; H^5_\ell(\RR^6)),\quad \Lambda^{\lambda}_{t, x}(\varphi(t)\psi(x) f)\in L^2([0, T]; H^5_\ell(\RR^6)).
\end{equation}
Then, one has
$$
\Lambda^s_v\Lambda^{\lambda}_{t, x}
(\varphi(t)\psi(x) f)\in L^2([0, T]; H^5_\ell(\RR^6)),
$$
for any $\ell\in\NN$ and any cutoff functions $\varphi, \psi$.
\end{prop}

Set
$$
F_{N, \ell}=P_{N, \ell}\, F= \psi_2(x) S_N(D_x)\, W_{\ell}\,\,
S_N(D_v)\partial^\alpha(\varphi(t)\psi(x) f),
$$
where $ \alpha\in\NN^6, |\alpha|\leq 6$ and $ \ell \in\NN$. Then \eqref{4.1+0} yields
$$
 \|\Lambda^{s}_{v} F_{N, \ell}\|_{L^2(\RR^7)}\leq C
\|\Lambda^{s}_{v}\partial^\alpha (\varphi(t)\psi(x) f)
\|_{L^2_\ell(\RR^7)},
$$
and
$$
 \|\Lambda^{\lambda}_{t, x} F_{N, \ell}\|_{L^2(\RR^7)}\leq C
\|\Lambda^{\lambda}_{t, x}\partial^\alpha (\varphi(t)\psi(x) f)
\|_{L^2_\ell(\RR^7)},
$$
where the constant $C$ is independent on $N$.

\smallbreak
It follows that $F_{N, \ell}$ satisfies the following equation
\begin{equation}\label{4.3}
\partial_t(F_{N, \ell}) + v\,\cdot\,\partial_x \,(F_{N, \ell}) = Q({f},\,\, F_{N, \ell})
+ G_{N, \ell},
\end{equation}
where $G_{N, l}$ is given by
\begin{align*}
 G_{N, \ell}=&  W_{ \ell}\, \Big[S_N(D_v),\,\, v\Big]\,\cdot\,\nabla_x S_N(D_x) F
+\Big(P_{N, \ell}\,  Q\big({f}_2,\,\, F\big)-Q\big({f}_2,\,\,
P_{N, \ell}\,  F\big)\Big)+ P_{N, l}\, G,
\end{align*}
with $G$ defined in \eqref{5.3.2}.

We now choose $ |D_{t, x}|^{\lambda}
|D_{t, x}|^{\lambda} F_{N, \ell} $ as a test function for equation
(\ref{4.3}). It follows that
\begin{align*}
&\|\Lambda^{s}_v \Lambda^{\lambda}_{t, x} P_{N, \ell }\,F \|_{L^2(\RR^7)}\notag\\
&\leq C \Big|\Big( |D_{t, x}|^{\lambda}Q({f}_2,\,\, F_{N,
\ell})-Q({f}_2,\,\, |D_{t, x}|^{\lambda}F_{N,
\ell}),\,\,  |D_{t, x}|^{\lambda} F_{N,
l}\Big)_{L^2(\RR^7)}\Big|\\
&\quad+\Big|\Big(|D_{t, x}|^{\lambda}G_{N, \ell},\,\,  |D_{t, x}|^{\lambda} F_{N,
l}\Big)_{L^2(\RR^7)}\Big|\, .
\end{align*}
Using the formula \eqref{4.6}, the proof of the Proposition
\ref{IV-prop4.2} is similar to the corresponding result in \cite{amuxy-nonlinear-3}, here we omit the cut-off function, it is easy to trait as before.

We can then get the following regularity result on the solution with
respect to the $(t, x)$ variable.
\begin{prop}\label{IV-prop4.3}
Under the hypothesis of Theorem \ref{theo-IV-1.1}, one has
\begin{equation}\label{4.1+19}
\Lambda^{1+\varepsilon}_{t, x}\,
(\varphi(t)\psi(x) f)\in L^2([0, T];\, H^5_\ell(\RR^6)),
\end{equation}
for any $\ell  \in\NN$ and some $\varepsilon>0$.
\end{prop}

\begin{proof} By fixing $s_0=\frac{s(1-s)}{(s+1)}$, then
\eqref{4.2} and Proposition \ref{IV-prop4.2} with $\lambda=s_0$ imply
$$
\Lambda^{s}_{v}\Lambda^{s_0}_{t, x} F\in L^{2}_\ell(\RR^7).
$$
It follows that,
\begin{equation*}\label{4.1+17}
 (\Lambda^{s_0}_{t, x} F)_t + v\,\cdot\,\partial_x {(\Lambda^{s_0}_{t, x} F)} +L_1 (\Lambda^{s_0}_{t, x} F)
  =\Lambda^{s_0}_{t, x} Q(f_2,\,\, F)+ \Lambda^{s_0}_{t, x}  G\in
H^{-s}_\ell(\RR^7).
\end{equation*}
By applying Lemma \ref{lemm2.1} with $s'=s$, we can deduce that
\begin{equation*}\label{4.1+18}
\Lambda^{s_0+s_0}_{t, x}
(F)\in L^2_\ell  (\RR^7),
\end{equation*}
for any $\ell  \in\NN$. If $2s_0<1$, by using again Proposition \ref{IV-prop4.2}
with $\lambda=2s_0$ and Lemma \ref{lemm2.1} with $s'=s$, we have
$$
\Lambda^{s}_{v} (\varphi(t)\psi(x) f),\,\,\Lambda^{2s_0}_{t, x}
(F)\in L^2_\ell (\RR^7)\,\,\Rightarrow\,
\Lambda^{3s_0}_{t, x} (F)\in L^2_\ell (\RR^7).
$$
Choose $k_0\in\NN$ such that
$$
k_0 s_0<1,\,\,\,\,\,\,\,\,\, (k_0+1) s_0=1+\varepsilon>1.
$$
Finally, (\ref{4.1+19}) follows from (\ref{4.1}) and Proposition
\ref{IV-prop4.2} with $\lambda=k_0s_0$ by induction. And this completes
the proof of the proposition \ref{IV-prop4.3}.
\end{proof}

\subsection{Full regularity of solution}

The above preparations will be used for the proof of the full
regularity of solution in Theorem \ref{theo-IV-1.1}, by using an
induction argument.

{}From Proposition \ref{IV-prop4.2} and Proposition
\ref{IV-prop4.3}, it follows that for any $\alpha\in\NN, |\alpha|\leq 5$ and any $\ell\in\NN$,
$$
\Lambda^{s}_{v}\, \partial^\alpha(\varphi(t)\psi(x) f),\,\, \Lambda_{t, x}
\, \partial^\alpha(\varphi(t)\psi(x) f) \in L^2_\ell(\RR^7).
$$

These will be used to get the high order regularity
with respect to the variable $v$.
\begin{prop}\label{IV-prop5.1}
Let $s\leq\lambda< 1$. Suppose that, for any cutoff functions $\varphi\in
C^\infty_0(]0, T[), \psi\in C^\infty_0(\RR^3)$, any $\alpha\in\NN, |\alpha|\leq 5$ and all $\ell\in\NN$,
\begin{equation}\label{5.2+0}
\Lambda^{\lambda}_{v}\, \partial^\alpha(\varphi(t)\psi(x) f),\,\,\,\Lambda_{t, x}
\, \partial^\alpha(\varphi(t)\psi(x) f) \in L^2_\ell(\RR^7).
\end{equation}
Then, for all cutoff function and all $\alpha\in\NN, |\alpha|\leq 5, \ell\in\NN$,
\begin{equation}\label{5.2}
\Lambda^{\lambda+s}_v\, \partial^\alpha(\varphi(t)\psi(x)
f)\in L^2_\ell(\RR^7).
\end{equation}
\end{prop}

\begin{proof} Recall that, for $|\alpha|\leq 5$, $F=\partial^\alpha
(\varphi(t)\psi(x) f)$ is the weak solution of the equation :
$$
 \frac{\partial F}{\partial t} + v\,\cdot\,\partial_x {F }- Q(f,\,\, F)= G,
\enskip (t, x, v) \in \RR^7,
$$
where $G$ is given in \eqref{5.3.2}. Set
$$
P_{N, \ell, \lambda}=\psi_2(x) W_{\ell}\,S_N(D_x)\, S_N(D_v)\,
\Lambda^{\lambda}_v\, ,
$$
we take now
$$
P_{N, \ell, \lambda}^\star\,P_{N, \ell, \lambda} F
=P_{N, \ell, \lambda}^\star\,F_{N, \ell, \lambda}\in
C^1_0([T_1, T_2]; H^{+\infty}_p(\RR^6))\,
$$
as test function. Then, one has
\begin{align*}
&\Big( \big[P_{N, \ell, \lambda}, \,\, v\big]\,\cdot\,\partial_x \,
F ,\,\,  P_{N, \ell, \lambda} F\Big)_{L^2(\RR^7)}
-\Big(Q({f},\, F_{N, \ell, \lambda}),\,\, F_{N, \ell, \lambda}\Big)_{L^2(\RR^7)}\\
&=\Big( P_{N, \ell, \lambda}Q({f},\,\, F)-Q({f},\,\, P_{N, \ell, \lambda}F),\,\, P_{N, \ell, \lambda} F\Big)_{L^2(\RR^7)}\\
&\qquad+\Big( P_{N, \ell, \lambda}G,\,\, P_{N, \ell, \lambda} F\Big)_{L^2(\RR^7)}.
\end{align*}
Since
$$
\big[\Lambda^{\lambda}_v, \,\, v\big]\,\cdot\,\partial_x =\lambda
\Lambda^{\lambda-2}_v \, \partial_v\,\cdot\,\partial_x,
$$
and $\Lambda^{\lambda-2}_v \, \partial_v$ are bounded operators in
$L^2$, for any $0<\lambda< 1$, we have, by using the hypothesis (\ref{5.2+0}) that
$$
\Big|\Big( \big[P_{N, \ell, \lambda}, \,\, v\big]\,\cdot\,\partial_x \,
F ,\,\,  P_{N, \ell, \lambda} F\Big)_{L^2(\RR^7)}\Big| \leq C\|\Lambda^{\lambda}_{v}\,F\|_{L^2_l(\RR^7)}\|\Lambda_x
\,F\|_{L^2_l(\RR^7)} .
$$
Using the coercivity \eqref{coercivity-nonlinear-a}, we get as
\eqref{coercivity-2},
\begin{align}\label{coercivity-2b}
&\eta_0||\Lambda^s_v W_{\gamma/2}F_{N,\, \ell, \lambda}||^2_{L^2(\RR^7)} \leq
C\|\Lambda^{\lambda}_{v}\,F\|_{L^2_l(\RR^7)}\|\Lambda_x
\,F\|_{L^2_l(\RR^7)}
+ \Big|\Big( P_{N, \ell, \lambda}G,\,\, P_{N, \ell, \lambda} F\Big)_{L^2(\RR^7)}\Big|\\
&+\Big|\Big( P_{N, \ell, \lambda}Q({f}_2,\,\, F)-Q({f}_2,\,\, P_{N, \ell, \lambda}F),\,\, P_{N, \ell, \lambda} F\Big)_{L^2(\RR^7)}\Big|.\notag
\end{align}
We conclude the proof of Proposition \ref{IV-prop5.1} by using the
following Lemma.
\end{proof}

\begin{lemm}\label{lemmIV-6.1b}
Let $f\in L^\infty([0, T]; H^5_\ell (\Omega\times\RR^3)),\, \ell \geq \ell_0$ (large). Then, we have, for any $\varepsilon>0$,
\begin{align*}
&\left|\Big( P_{N, \ell, \lambda}G,\,\, P_{N, \ell, \lambda} F\Big)_{L^2(\RR^7)}\right|\leq \varepsilon
||\Lambda^s_v W_{\gamma/2}F_{N,\, \ell, \lambda}||^2_{L^2(\RR^7)}\\
&+
C_\varepsilon\Big(\|\Lambda^\lambda_v f_3\|^{2}_{L^2([0, T]; H^5_{2\ell} (\RR^6))}\|\Lambda^1_t  f_3\|^{2}_{L^2([0, T]; H^5_{2\ell} (\RR^6))}
+\|\Lambda^1_x  f_3\|^{2}_{L^2([0, T]; H^5_{2\ell} (\RR^6))}\Big).
\end{align*}
\end{lemm}

\begin{proof} By using the decomposition in \eqref{5.3.2}, it is obvious that for the linear terms
$$
\left|\Big( P_{N, \ell, \lambda}((B)+(C)),\,\, P_{N, \ell, \lambda} F\Big)_{L^2(\RR^7)}\right|\leq
C\|\Lambda^\lambda_v f_2\|^{2}_{L^2([0, T]; H^5_{\ell+1} (\RR^6))}.
$$
For the term $(A)$, recall that $\alpha_1+\alpha_2=\alpha$,
$|\alpha|\leq 5$ and $|\alpha_2|<5$. Then, by separating the cases
$|\alpha_1|\leq 3$ and $|\alpha_1|> 3$, we get, with
$\Lambda^\lambda_N(D_v)=\Lambda^\lambda_v S_N(D_v)$,
\begin{align*}
&\Big|\Big(P_{N, \ell,\lambda}Q\big(\partial^{\alpha_1}
{f_2},\,\partial^{\alpha_2} f_1\big),\,
\, P_{N, \ell, \lambda} F\Big)_{L^2(\RR^7)}\Big|\\
&=\Big|\Big(\Lambda^\lambda_N(D_v)\big(Q\big(
\partial^{\alpha_1}
{f_2},\,\partial^{\alpha_2} f_1\big)\big),\,
\, W_\ell\, S^\star_N(D_x)\psi_2 F_{N, \ell, \lambda} \Big)_{L^2(\RR^7)}\Big|\\
&\leq\Big|\Big(\Lambda^\lambda_N(D_v)\big(Q\big(
\partial^{\alpha_1}
{f_2},\,\partial^{\alpha_2} f_1\big)\\
&\quad-\big(Q\big(
\partial^{\alpha_1}
{f_2},\,\Lambda^\lambda_N(D_v)\partial^{\alpha_2} f_1\big)\big),\,
\, W_\ell S^\star_N\psi_2(D_x)F_{N, \ell, \lambda} \Big)_{L^2(\RR^7)}\Big|\\
&\quad+\Big|\Big(\big(Q\big(
\partial^{\alpha_1}
{f_2},\,\Lambda^\lambda_N(D_v)\partial^{\alpha_2} f_1\big)\big),\,
\, W_\ell S^\star_N(D_x)\psi_2 F_{N, \ell, \lambda} \Big)_{L^2(\RR^7)}\Big|\,.
\end{align*}
Using Corollary \ref{IV-coro-2.15}, we have
\begin{align*}
&\Big|\Big(\Lambda^\lambda_N(D_v)\big(Q\big(
\partial^{\alpha_1}
{f_2},\,\partial^{\alpha_2} f_1\big)-\big(Q\big(
\partial^{\alpha_1}
{f_2},\,\Lambda^\lambda_N(D_v)\partial^{\alpha_2} f_1\big)\big),\,
\, W_\ell S^\star_N(D_x)\psi_2 F_{N, \ell, \lambda} \Big)_{L^2(\RR^7)}\Big|\\
&\leq C\int \|\partial^{\alpha_1} f_2\|_{H^{s'}_{2+\gamma^++(2s-1)^+}(\RR^3_v)}\|\partial^{\alpha_2} f_1\|_{H^{\lambda}_{\gamma^++(2s-1)^+}(\RR^3_v)}\|W_\ell \,S^\star_N(D_x)\psi_2 F_{N, \ell, \lambda}\|_{H^{s'}(\RR^3_v)} dx dt\\
&\leq C\left\{
\begin{array}{ll}
\|\partial^{\alpha_1} f_2\|_{L^\infty(\RR^4_{t, x}; H^{s'}_{2+\gamma^++(2s-1)^+}(\RR^3_v))}\|\Lambda^5_{x, v}\Lambda^\lambda_{v} f_1\|_{L^2_{\gamma^++(2s-1)^+}(\RR^7)}\|\Lambda^{s'}_{v} F_{N, \ell, \lambda}\|_{L^{2}_\ell(\RR^7)},& |\alpha_1|\leq 3\\
\|\partial^{\alpha_1} f_2\|_{L^2(\RR^4_{t, x}; H^{s'}_{2+\gamma^++(2s-1)^+}(\RR^3_v))}\|\Lambda^{1+3/2+\delta}_{x, v}\Lambda^{\lambda}_{v} \Lambda^{1/2+\delta}_{t} f_1\|_{L^2_{\gamma^++(2s-1)^+}(\RR^7)}\|\Lambda^{s'}_{v} F_{N, \ell, \lambda}\|_{L^{2}_\ell(\RR^7)},& |\alpha_1|> 3
\end{array}
\right.\\
&\leq C\|\Lambda^5_{x, v}\Lambda_{t} f_2\|_{L^2_{2+\gamma^++(2s-1)^+}(\RR^7)}\|\Lambda^5_{x, v}\Lambda^\lambda_{v} f_2\|_{L^2_{\gamma^++(2s-1)^+}(\RR^7)}\|\Lambda^{s'}_{v} F_{N, \ell, \lambda}\|_{L^{2}_\ell(\RR^7)}
\\
&\leq \epsilon\|\Lambda^{s}_{v}W_{\gamma/2} F_{N, \ell, \lambda}\|^2_{L^{2}(\RR^7)}+ C_\epsilon\|\Lambda^5_{x, v}\Lambda_{t} f_2\|^2_{L^2_{2\ell}(\RR^7)}\|\Lambda^5_{x, v}\Lambda^\lambda_{v} f_2\|^2_{L^2_{2\ell}(\RR^7)}\,,
\end{align*}
Proposition \ref{prop-IV-2.5b} with $m=0$ and Sobolev embedding for $x\in\RR^3$ and $t\in\RR$  give
\begin{align*}
&\Big|\Big(\big(Q\big(
\partial^{\alpha_1}
{f_2},\,\Lambda^\lambda_N(D_v)\partial^{\alpha_2} f_1\big)\big),\,
\, W_\ell S^\star_N(D_x)\psi_2 F_{N, \ell, \lambda} \Big)_{L^2(\RR^7)}\Big|\\
&\leq C\int \|\partial^{\alpha_1} f_2\|_{L^2_{2+(\gamma+2s)^++\ell-\gamma/2}(\RR^3_v)}\|\partial^{\alpha_2} f_1\|_{H^{\lambda+s}_{\ell-\gamma/2+(\gamma+2s)^+}(\RR^3_v)}\|W_\ell \,S^\star_N(D_x)\psi_2 F_{N, \ell, \lambda}\|_{H^{s}_{-\ell+\gamma/2}(\RR^3_v)} dx dt
\\
&\leq \epsilon\|\Lambda^{s}_{v}W_{\gamma/2} F_{N, \ell, \lambda}\|^2_{L^{2}(\RR^7)}+ C_\epsilon\|\Lambda^5_{x, v}\Lambda_{t} f_2\|^2_{L^2_{2\ell}(\RR^7)}\|\Lambda^5_{x, v}\Lambda^\lambda_{v} f_2\|^2_{L^2_{2\ell}(\RR^7)}\,.
\end{align*}
This ends the proof of Lemma \ref{lemmIV-6.1b}.
\end{proof}

Recall $\Lambda^\lambda_N(D_v) = \Lambda_v^\lambda S_N(D_v)$, we
have for the last term of \eqref{coercivity-2b}, (the term involving
$\mu$ is omitted since is easier than $f_2$),
\begin{align*}
&\Big( P_{N, \ell, \lambda}Q({f}_2,\,\, F)-Q({f}_2,\,\, P_{N, \ell, \lambda}F),\,\, P_{N, \ell, \lambda} F\Big)_{L^2(\RR^7)}\\
&=\Big( \Lambda^\lambda_N(D_v)Q({f}_2,\,\, F)-Q({f}_2,\,\, \Lambda^\lambda_N(D_v)F),\,\, W_\ell\,S^\star_N(D_x)\,P_{N, \ell, \lambda} F\Big)_{L^2(\RR^7)}\\
&+\Big(S_N(D_x)\, Q({f}_2,\,\, \Lambda^\lambda_N(D_v)F)-Q({f}_2,\,\, S_N(D_x)\,\Lambda^\lambda_N(D_v)F),\,\, W_\ell\,P_{N, \ell, \lambda} F\Big)_{L^2(\RR^7)}\\
&+\Big( W_\ell\,Q({f}_2,\,\, S_N(D_x)\,\Lambda^\lambda_N(D_v)F)-Q({f}_2,\,\, P_{N, \ell, \lambda}F),\,\, P_{N, \ell, \lambda} F\Big)_{L^2(\RR^7)}\\
&=(I)+(II)+(III)\, .
\end{align*}
Using again Corollary \ref{IV-coro-2.15}, we have by interpolation,
\begin{align*}
|(I)|&\leq C\int \|f_2\|_{H^{s'}_{2+(2s-1)^++\gamma^+}}
\|F\|_{H^\lambda_{(2s-1)^++\gamma^+}}
\|W_\ell\,S^\star_N(D_x)\,P_{N, \ell, \lambda} F\|_{H^{s'}} dx dt\\
&\leq C \|\Lambda_t \Lambda^2_x\Lambda^{s'}_v f_2\|_{L^2_{2+(2s-1)^++\gamma^+}(\RR^7)}
\|\Lambda^{\lambda}_v F\|_{L^2_{(2s-1)^++\gamma^+}}\|\Lambda^{s'}_v W_\ell\,S^\star_N(D_x)\,P_{N, \ell, \lambda} F\|_{L^{2}(\RR^7)}\\
&\leq C \|\Lambda_t \Lambda^4_{x, v} f_2\|_{L^2_{2+(2s-1)^++\gamma^+}(\RR^7)}
\|\Lambda^{\lambda}_v F\|_{L^2_{(2s-1)^++\gamma^+}}\Big(\epsilon\|\Lambda^{s}_v W_{\gamma/2}\,P_{N, \ell, \lambda} F\|_{L^{2}(\RR^7)}+
C_\epsilon \|\Lambda^{\lambda}_v F\|_{L^{2}_{3\ell}(\RR^7)}\Big)
\\
&\leq \epsilon\|\Lambda^{s}_v W_{\gamma/2}\,P_{N, \ell, \lambda} F\|^2_{L^{2}(\RR^7)}+
C_\epsilon \|\Lambda_t \Lambda^4_{x, v} f_2\|^2_{L^2_{2+(2s-1)^++\gamma^+}(\RR^7)}
\|\Lambda^{\lambda}_v \Lambda^5_{x, v} f_1\|^2_{L^2_{2\ell}(\RR^7)}\,.
\end{align*}
Using now \eqref{3.10+1}, similarly as for \eqref{3.10+1b}, we have
\begin{align*}
|(II)|&\leq C 2^{-N} \|\nabla_x f_2\|_{L^\infty(\RR^4_{t,
x},\,\, L^2_{\ell+2-\gamma/2+(2s+\gamma)^+}(\RR^3_v))}
\|M_\lambda(D_v)F\|_{L^2(\RR^4_{t, x},\,\,
H^s_{\ell-\gamma/2+(2s+\gamma)^+}(\RR^3_v))}
\|W_\ell\,P_{N, \ell, \lambda} F\|_{L^2(\RR^4_{t, x},\,\,
H^s_{-\ell+\gamma/2}(\RR^3_v))}\\
&\leq C \|\Lambda_t\Lambda^3_x f_2\|_{L^2_{\ell+2-\gamma/2+(2s+\gamma)^+}(\RR^7)}
\|\Lambda^\lambda_v F\|_{L^2_{\ell-\gamma/2+(2s+\gamma)^+}(\RR7)}
\|\Lambda^s_v W_{\gamma/2}\,P_{N, \ell, \lambda} F\|_{L^2(\RR^7)}\,.
\end{align*}
For the term $(III)$, we use \eqref{IV-4.3}
\begin{align*}
|(III)|&=\Big|\Big( W_\ell\,Q({f}_2,\,\, S_N(D_x)\,\Lambda^\lambda_N(D_v)F)-Q({f}_2,\,\, P_{N, \ell, \lambda}F),\,\, P_{N, \ell, \lambda} F\Big)_{L^2(\RR^7)}\Big|\\
&\leq C \|f_2\|_{L^\infty(\RR^4_{t, x}; L^2_{\ell+2+(2s-1)^++\gamma^+}(\RR^3_v))}
\|S_N(D_x)\,\Lambda^\lambda_N(D_v)F\|_{L^2(\RR^4_{t, x}; H^{(2s-1+\epsilon)^+}_{\ell+(2s-1)^++\gamma^+}(\RR^3_v))}
\|P_{N, \ell, \lambda} F\|_{L^2(\RR^7)}\\
&\leq \epsilon \|\Lambda^s_v W_{\gamma/2}\,P_{N, \ell, \lambda} F\|^2_{L^2(\RR^7)}+C_\epsilon\Big(\|\Lambda^\lambda_v F\|^2_{L^2_{2\ell}(\RR7)}+ \|\Lambda_t\Lambda^2_x f_2\|^2_{L^2_{\ell+2-\gamma/2+(2s+\gamma)^+}(\RR^7)}
\|\Lambda^\lambda_v F\|^2_{L^2_{\ell}(\RR7)}\Big)\,.
\end{align*}
Finally, from \eqref{coercivity-2b}, choose $\epsilon>0$ small enough, we get for big $\ell$,
\begin{align*}
\frac {\eta_0}{2}||\Lambda^s_v W_{\gamma/2}F_{N,\, \ell, \lambda}||^2_{L^2(\RR^7)} \lesssim&
\|\Lambda^{\lambda}_{v}\,\Lambda^5_{x, v} f_2\|^2_{L^2_{2\ell}(\RR^7)}+\|\Lambda_x
\,\Lambda^5_{x, v} f_2\|^2_{L^2_{2\ell}(\RR^7)}\\
&+ \|\Lambda_{t}\Lambda^5_{x, v} f_2\|^2_{L^2_{2\ell}(\RR^7)}\|\Lambda^\lambda_{v}
\Lambda^5_{x, v} f_2\|^2_{L^2_{2\ell}(\RR^7)}.
\end{align*}
Taking the limit $N\,\rightarrow\,\infty$, we have proved
\eqref{5.2}, and ended the proof of Proposition \ref{IV-prop5.1}.

We can now conclude that the following regularity result with
respect to the variable $v$ holds true.
\begin{prop}\label{IV-prop5.2}
Under the hypothesis of Theorem \ref{theo-IV-1.1}, one has
\begin{equation}\label{5.8}
\Lambda^{1+\varepsilon}_{v}
(\varphi(t)\psi(x) f)\in L^2([0, T];  H^5_\ell(\RR^6)),
\end{equation}
for any $\ell \in\NN$ and some $\varepsilon>0$.
\end{prop}
Again, this result is indeed obtained by noticing that there
exists $k_0\in\NN$ such that
$$
k_0 s<1,\,\,\,\,\,\,\,\,\, (k_0+1) s=1+\varepsilon>1.
$$
Then we get (\ref{5.8}) from (\ref{3.3}), Proposition
\ref{IV-prop5.1} with $\lambda=k_0s$  by induction.

\smallbreak\noindent {\bf High order regularity by iterations}

{}From Proposition \ref{IV-prop4.3} and Proposition \ref{IV-prop5.2}, we can now deduce
that, for any $\ell\in\NN$, and any cutoff functions $\varphi(t)$
and $\psi(x)$,
$$
\Lambda_v(\varphi(t)\psi(x) f), \Lambda_{t, x}(\varphi(t)\psi(x) f)\in L^2([0, T]; H^5_\ell(\RR^6))\,,
$$
which is
$$
\varphi(t)\psi(x) f\in L^2([0, T]; H^6_\ell(\RR^6))\cap H^1([0, T]; H^5_\ell(\RR^6))
$$
The proof of full regularity is then completed by induction for $(x, v)$ variable
$$
\varphi(t)\psi(x) f\in L^2([0, T]; H^7_\ell(\RR^6))\cap H^1([0, T]; H^6_\ell(\RR^6))
$$
and using the equation to prove the regularity for $t$ variable.

\section{Uniqueness  of solutions}\label{sect-IV-7}
\smallskip \setcounter{equation}{0}

In this section, we prove precise versions for uniqueness results which will cover more general cases than those presented in Theorem \ref{theo-IV-1.2}.

We need the coercive estimate in a global version: For suitable function $f$, we say that $f$ satisfies the global coercive estimate, if there exist constants $c_0>$ and $C>0$ independent of
$ t\in ]0, T[$ such that

\begin{align}\label{cor-sol}
 -(Q(f(t),h),h)_{L^2(\RR^6)}  \ge c_0 \|h\|^2_{L^2(\RR^3_x; H^s_{\gamma/2}(\RR^3_v))} -C\|h\|^2 _{L^2(\RR^3_x; L^2_{ (\gamma/2+s)^+}(\RR^3_v))}
\end{align}
for any $h \in L^2(\RR^3_x; \cS(\RR^3_v))$.
Using the notations introduced in Section \ref{sect-IV-1}, we prove the following precise version of Theorem \ref{theo-IV-1.2}, where we do not assume that solution is a perturbation around a normalized Maxwellian.

\begin{theo}\label{uniqueness}  Assume that $0<s<1$ and $\max \big \{-3, -3/2-2s \big \} <\gamma < 2-2s $.
Let $ f_0 \in
\tilde \cE^{0}_0(\RR^6)$, $ 0<T < +\infty$ and
suppose that $f \in
\tilde { \mathcal E}^{m}([0, T]\times{\mathbb R}^6_{x, v}), m\geq s$
is a weak solution to the Cauchy problem
\eqref{IV-1.1}. If $f$ is non-negative, then solution $f$ is unique in the function space $
\tilde  { \mathcal E}^{2s}([0,T]\times{\mathbb R}^6_{x, v})$.

Moreover, if $f$ is non-negative and satisfies  the global coercive estimate
\eqref{cor-sol}, 
then the solution $f$ is unique in the function space $
\tilde  { \mathcal E}^s([0,T]\times{\mathbb R}^6_{x, v})$.
The same conclusion holds
without the non-negativity of $f $ if the term $\|h\|^2_{L^2(\RR^3_x; H^s_{\gamma/2}(\RR^3_v))}$
in the condition  \eqref{cor-sol} is replaced by $\int |||h|||^2_{\Phi_\gamma}dx $\,.

\end{theo}

\begin{rema}
In the case where $\gamma >-3/2$ and  $f\in
\tilde  { \mathcal E}^{s}([0,T]\times{\mathbb R}^6_{x, v})$ is non-negative,
it follows  that $f$ coincides with any another solution $f_2 \in
\tilde  { \mathcal E}^{2s}([0,T]\times{\mathbb R}^6_{x, v})$ without the coercivity condition \eqref{cor-sol}.
\end{rema}

The next result proves the uniqueness of perturbative solutions around a normalized Maxwellian obtained in \cite{amuxy4-2,amuxy4-3} where we do not assume the non-negativity of solutions.
\begin{theo}\label{uniqueness-small-pur}
Assume that $0<s<1$, $\max \{ -3, -3/2-2s \}  <\gamma $.
Let $\ell_1 > 3/2 + \max \{(\gamma +2s)^+, |\gamma|/2 \}$.
Then there exists an $\varepsilon_0 >0$ such that
if $f_1(t), f_2(t) \in \tilde  { \mathcal E}^s([0,T]\times{\mathbb R}^6_{x, v})$ are  two solutions of the Cauchy problem
\eqref{IV-1.1} with the properties
\[
\mu^{-1/2}\big ( f_j(t) - \mu \big) \in L^\infty([0,T]\times \RR_x^3; H^s_{\ell_1}) \, , \enskip j=1,2 \,,
\]
and the smallness condition for $f_1$
\begin{equation}\label{small}
||\mu^{-1/2}\big ( f_1(t) - \mu \big) ||_{L^\infty([0,T]\times \RR_x^3; { L^2(\RR^3))}} \le \varepsilon_0,
\end{equation}
then $f_1(t) \equiv f_2(t)$ for all $t \in [0,T]$.
\end{theo}

To study the uniqueness of solutions constructed in Theorem 1.4 of \cite{amuxy4-2}, we define
another function space
with exponential decay in the velocity variable as follows:
For $m\in\RR$
and for $T>0$,  set
\begin{eqnarray*}
\tilde {\mathcal B}^m([0,T]\times{\mathbb R}^6_{x, v})&=&\Big\{f\in
C^0([0,T];{\mathcal D}'({\mathbb R}^6_{x, v}));\, \exists \,\rho>0
\\
&&\hskip 0.5cm s. t. \,\, e^{\rho \langle v \rangle^2} f\in L^\infty([0,
T]\times \RR^3_x ;\,\, L^2({\mathbb R}^3_{v})) \mathop{\cap} L^2([0,T]; L^\infty(\RR_x^3; H^m(\RR^3_v))) \Big\}.
\end{eqnarray*}
We get the following refinement of the last part of Theorem \ref{uniqueness}, in the case $\gamma+2s \le 0$.

\begin{theo}\label{unique-x-global}
Assume that $0<s<1$ and $\max \big \{-3, -3/2-2s \big \} <\gamma \le -2s $.
 Let $ 0<T < +\infty$ and
suppose that $f_1(t) \in
\tilde { \mathcal B}^{s}([0,T]\times{\mathbb R}^6_{x, v})$
is a  solution to the Cauchy problem
\eqref{IV-1.1} satisfying the global coercivity estimate
\eqref{cor-sol} with
the term $\|h\|^2_{L^2(\RR^3_x; H^s_{\gamma/2} (\RR^3_v))}$
 replaced by $\int |||h|||^2_{\Phi_\gamma}dx\, $.
Then $f_1(t)$ coincides with any another solution $f_2(t) \in
\tilde  { \mathcal B}^s([0,T]\times{\mathbb R}^6_{x, v})$.
\end{theo}

If the Cauchy problem
\eqref{IV-1.1} admits two solutions $f_1(t), f_2(t)  \in \tilde{\cE}^{s}([0,T] \times \RR^6_{x, v})$, then
 there exist
$\rho_0, \rho_1,  \rho_2 >0$  such that
$$
e^{\rho_0\la v\ra ^2} f_0\in L^\infty(\RR_x^3;  L^2(\RR^3_{v}) ),\,\,
e^{\rho_1\la v\ra ^2} f_1,\,\,e^{\rho_2\la v\ra ^2} f_2\in L^\infty([0,
T]\times \RR^3_x ;\,\,  H^s({\mathbb R}^3_{v}))
\,.
$$
Take  $0<\rho<\min\{\rho_0,\, \rho_1,\, \rho_2\}$ and $\kappa>0$
sufficiently small such that $\frac{\rho}{2\kappa}>T$.
Then we have
$$
g_0=e^{\rho\la v\ra ^2} f_0\in L^\infty(\RR_x^3; L^2_l(\RR^3_{v})),
\,\,
g_1=e^{(\rho-\kappa t)\la v\ra ^2} f_1,\,\,\, g_2= e^{(\rho-\kappa t)\la v\ra ^2}
f_2\in L^\infty([0,
T]\times \RR^3_x ;\,\, H^s_l({\mathbb R}^3_{v}))
$$
for any $l\in\NN$, and  $g_1, g_2$ are two solutions of the following Cauchy problem
\begin{equation}\label{E-Cauchy-B}
\left\{\begin{array}{l}
\tilde{g}_t+v\cdot\nabla_x \tilde{g}\  +\kappa (1+ |v|^2) \tilde{g}=\Gamma^t(\tilde{g}, \tilde{g}),\\
\tilde{g}|_{t=0}=g_0\,,
\end{array}\right.
\end{equation}
where
$$
\Gamma^t(g, h)=\mu_\kappa(t)^{-1}Q(\mu_\kappa(t) g,\,
\mu_\kappa(t) h)\,\,, \,\,
\mu_\kappa(t)=\mu(t,v)=e^{-(\rho-\kappa t)(1+ |v|^2)}.
$$
Set $g= g_1-g_2$. Then we have
\begin{equation}\label{E-Cauchy-B^diff}
\left\{\begin{array}{l}
g_t+v\cdot\nabla_x g\  +\kappa (1+ |v|^2) g=\Gamma^t(g_1, g) +\Gamma^t(g, g_2)   ,\\
g|_{t=0}= 0\,.
\end{array}\right.
\end{equation}

\subsection{Estimates for modified collisional operator}\label{subsct4.2}

We now prepare several  lemmas concerning
the estimates for $\Big(\Gamma^t(f,g),h\Big)_{L^2}$, where $L^2 = L^2(\RR^3_v)$. In this subsection, variables $t$ and $x$ are regarded as parameters.
For the brevity we often write $\Gamma$ and $\mu$ instead of $\Gamma^t$ and $\mu_\kappa(t,v)$,
respectively. All constants in estimates are uniform with respect to $t \in [0,T]$ and moreover
they hold with $\mu_k(t)$ replaced by $\mu^{1/2}$.

\begin{lemm}\label{differ-Gam-Q}
Let $0<s<1$ and $\gamma > \max \{-3, -2s -3/2\}$. Then for any $\beta \in \RR$ we have
\begin{align}\label{diff-G-Q}
&\Big|\Big(\Gamma^t( f, g)\,  , \,h \Big)_{L^2} -\Big( Q(\mu_\kappa f, g), h\Big)_{L^2}\Big|\\
&\lesssim \Big(\cD(\mu_\kappa { \,|f|} ,  \la v \ra^\beta g) \Big)^{1/2}
\|f\|_{L^2} ^{1/2}
\|h\|_{L^2_{s+\gamma/2-\beta}}
+ \|f\|_{L^2}
\|g\|_{L^2_{s +\gamma/2+ \beta}}
\|h \|_{L^2_{s+\gamma/2-\beta }} \,.\notag
\end{align}
\end{lemm}

\begin{proof}
We write
\begin{align*}
\Big(\Gamma( f, g)\,  , \,h \Big)_{L^2} -\Big( Q(\mu f, g), h\Big)_{L^2}&= \iiint B \Big( \mu'_{*} -\mu_{*} \Big) \big( f_{*} \big)
g h ' d\sigma dv_*dv \\
& = 2
\iiint B \Big( (\mu'_{*})^{1/2} -\mu_{*}^{1/2} \Big) \big( \mu_{*}^{1/2}
f_{*} \big)
g'h' d\sigma dv_*dv \\
&+
\iiint B \Big( (\mu'_{*})^{1/2} -\mu_{*}^{1/2} \Big) ^2 f_{*}
g  h'
d\sigma dv_*dv \\
&+ 2
\iiint B \Big( (\mu'_{*})^{1/2} -\mu_{*}^{1/2} \Big)
h' f_*   \mu_{*} ^{1/2} \Big (
g       - g '\Big) d\sigma dv_*dv \\
&= D_1 + D_2 + D_3\,.
\end{align*}
By the Cauchy-Schwarz inequality we have for any $\beta \in \RR$
\begin{align*}
|D_3| &\lesssim  \Big(
\iiint B \Big( (\mu'_{*})^{1/2} -\mu_{*}^{1/2} \Big)^2 { | f_{*}|}
{\big(\la v \ra^{-\beta} h\big)'}^2 d\sigma dv_*dv \Big)^{1/2}\\
&\quad \times
 \Big( \iiint B   \mu_{*} \, { |f_*| } \la v' \ra^{2\beta}
\Big (  g- g '
\Big)^2 d\sigma dv_*dv \Big)^{1/2}\\
& = \Big( \widetilde D_{3}(f,\la v \ra^{-\beta}h) \Big)^{1/2}\Big(\cD_\beta(\mu f, g)\Big)^{1/2}\,.
\end{align*}
We have
\begin{align*}
\cD_\beta(\mu f, g)
&\le 2 \Big( \cD(\mu \, { | f|},  \la v \ra^\beta g) + \iiint B \mu_* \, { |f_*| } \Big(\la v \ra^\beta -
\la v' \ra^\beta\Big)^2g^2 dv dv_*d\sigma \Big)\\
&\lesssim \cD(\mu\,{ | f|},  \la v \ra^\beta g) + \|f \|_{L^2} \|g\|^2_{L^2_{\beta+\gamma/2}}\,,
\end{align*}
because it follows from the same arguments in the proof of Lemma \ref{upper-triple-Dirichlet} that
\begin{equation}\label{wet}
\Big|\la v \ra^\beta -
\la v' \ra^\beta\Big| \lesssim \sin \frac{\theta}{2} \Big(\la v \ra^\beta \la v_*\ra^{ 2|\beta|+1}{\bf 1}_{|v-v_*| > 1}+
\la v \ra^{\beta-1} |v-v_*| {\bf 1}_{|v-v_*| \leq 1}\Big)\,.
\end{equation}
Note that
\begin{align*}
\Big((\mu'_{*})^{1/2} -\mu_{\kappa,*}^{1/2} \Big)^2&\lesssim
 (\mu'_{*})^{1/4} -\mu_{*}^{1/4} )^2\Big(  (\mu'_{*})^{1/2} +\mu_{*}^{1/2} \Big)\\
&\lesssim        \min(|v'-v_*|\theta,1)  \min(|v'-v'_*|\theta,1)  (\mu'_{*})^{1/2}
+  \Big(\min(|v'-v_*|\theta,1) \Big)^2\mu_{*}^{1/2} \,.
\end{align*}
By this decomposition we estimate
\[ \widetilde D_3(f,\la v \ra^{-\beta}h) \lesssim \widetilde D_3^{(1)}+ \widetilde D_3^{(2)}.
\]
It follows from the Cauchy-Schwarz inequality that
\begin{align*}
\Big|\widetilde D_3^{(1)}\Big|^2 &\lesssim
\Big(\iiint B |v'-v'_*|^{\gamma+2s}  \Big( \min(|v'-v'_*|\theta,1)\Big)^2  \mu'_{*}
{\Big(W_{-\beta}h \Big)'}^2 \la v'\ra^{-(\gamma+2s)} d\sigma dv_*dv dx \Big)\\
& \times
\Big(\iiint B |v'-v_*|^{-(\gamma+2s)}  \Big(\min(|v'-v_*|\theta,1) \Big)^2 f_{*}^2
{\Big(W_{-\beta} h\Big)'}^2 \la v'\ra^{\gamma+2s} d\sigma dv_*dv  \Big)\\
&\lesssim \|h\|^2_{L^2_{s+\gamma/2-\beta }} \|f\|^2_{L^2}\,
\|h\|^2_{L^2_{s+\gamma/2-\beta}}\,.
\end{align*}
Here, we have taken the change of variables $(v,v_*) \rightarrow (v',v'_*)$ and $v \rightarrow v'$ in the first and second
factors, respectively, and moreover,
  in view of $2(\gamma + 2s) >-3$, we have used the fact that
\[
\int\Big( \int B |v-v_*|^{\gamma+2s}  \min(|v-v_*|^2\theta^2 ,1) d\sigma\Big)  \mu_{*}
dv_*\lesssim   \int |v-v_*|^{2(\gamma+2s)}  \mu_{*}dv_* \lesssim \la v\ra^{2(\gamma+2s)} \,.
\]
Since the estimation of $\widetilde D_3^{(2)}$ is quite similar as $\widetilde D_3^{(1)}$ we obtain
\begin{equation}\label{wide-D-3}
\widetilde D_3(f,\la v \ra^{-\beta}h) \lesssim
\|f\|_{L^2}
\|h\|^2_{L^2_{s+\gamma/2-\beta}}
\end{equation}
and hence
\[
|D_3| \lesssim  \Big(\cD(\mu { | f|},  \la v \ra^\beta g) \Big)^{1/2}
\|f\|_{L^2} ^{1/2}
\|h\|^2_{L^2_{s+\gamma/2-\beta}}
+ \|f\|_{L^2}
\|g\|_{L^2_{\beta +\gamma/2}}
\|h \|_{L^2_{s+\gamma/2-\beta } }
\]
The Cauchy-Schwarz inequality shows
\[
|D_2|^2 \le  \widetilde D_3 (f, \la v\ra^{-\beta}h)
\left(\iiint B \Big( (\mu'_{*})^{1/2} -\mu_{*}^{1/2} \Big) ^2 { |f_* |}\big(  \la v' \ra^\beta g\big)^2
d\sigma dv_*dv \right),
\]
so that it is easy to see, in view of  \eqref{wet},
\[
|D_2| \lesssim
\|f\|_{L^2}
\|g\|_{L^2_{s+\gamma/2+ \beta}}
\|h \|_{L^2_{s+\gamma/2-\beta } }\,.
\]
Take the change of variables $(v',v'_*) \rightarrow (v,v_*)$ and $(v,v_*) \rightarrow (v_*,v)$ for $D_1$.
Then we  consider
\begin{align*}
D_1 &=2
\iiint B \Big( \mu^{1/2} - \big(\mu')^{1/2} \Big)  \big(\mu^{1/2}
f )'
{(gh)}_*  d\sigma dv_*dv \\
&= 2 \iiint B \big(\nabla \mu^{1/2}\big) (v') \cdot (v-v')  \big(\mu^{1/2}
f \big)'
{(gh)}_*  d\sigma dv_*dv \\
&+\int_0^1\iiint B  \big( \nabla^2 \mu^{1/2}\big)(v'+ \tau(v-v') ) (v-v')^2
\big(\mu^{1/2} f \big)'
{(gh)}_* d\sigma dv_*dv d\tau\\
&= D_{1,1}+ D_{1,2}\,,
\end{align*}
by using the Taylor formula
\begin{align*}
{\mu}^{1/2} -{\mu'}^{1/2}
&= \big(\nabla \mu^{1/2}\big) (v')\cdot (v-v') + \frac{1}{2}
\int_0^1 \big( \nabla^2 \mu^{1/2}\big)(v'+ \tau(v-v') ) (v-v')^2  d\tau\,.
\end{align*}
Note
\[
|\big( \nabla^2 \mu_{\kappa}^{1/2}\big)(v'+ \tau(v-v') ) (v-v')^2| \lesssim  |v' - v_*|^2
(1-\cos \theta).
\]
and devide
\[
D_{1,1} = 2\Big( \iiint_{\{v'-v_*| ^2 (1-\cos \theta)  \le 1\}} +  \iiint_{\{v'-v_*| ^2 (1-\cos \theta)  > 1\} }\Big)\,.
\]
Then
it follows from the spherical symmetry that the first term of the decomposition $D_{1,1}$ vanishes, so that
we can estimate by the change of variables $v \rightarrow v'$
\begin{align*}
|D_1| & \lesssim
\iint |v'-v_*|^\gamma
\big(\mu^{1/2} \,{ |f|}\big)' \Big(\int_{\SS^2}  b(\cos \theta) \min\Big ( \theta^2 | v'-v_*|^2, 1\Big) d\sigma\Big)
(gh )_*
dv'dv_*\\
& \lesssim
\left(\iint|v'-v_*|^{2(\gamma+2s)} {\mu'} g^2_*\la v_*\ra^{-(\gamma+2s)+ \beta} dv' dv_* \right)^{1/2}\\
&\qquad \times
\left(\iint {f'}^2  h^2_*\la v_*\ra^{\gamma+2s-\beta} dv' dv_* \right)^{1/2}\\
&\lesssim
\|f\|_{L^2}
\|g\|_{L^2_{s+\gamma/2+ \beta}}
\|h \|_{L^2_{s+\gamma/2-\beta } }
\end{align*}
Summing up above estimates we obtain
the desired estimate.
\end{proof}

Since Lemma \ref{differ-Q-cD} holds with $\sqrt \mu$ replaced by $\mu_\kappa(t,v)$, the combination of Lemma \ref{differ-Q-cD} and Lemma \ref{differ-Gam-Q} with $\beta =0$ implies

\begin{lemm}\label{coer-gamma}
 Let $0<s<1$, $\gamma > \max\{-3, -2s-3/2\}$. If $f \ge 0$ then we have
\begin{align}\label{coe-Ga}
&\Big(\Gamma^t( f, \, g),\, g \Big ) _{L^2}
\le -\frac 1 4 \cD(\mu_\kappa f, g) \\
&\qquad + C \min\Big\{
\|f\|_{L^2}\big( \|g\|^2_{H^{s'}_{\gamma/2}} + \|g\|^2_{L^2_{s+\gamma/2}}\big)\,\,,\,\,
\|f\|_{H^{2s'}} \|g\|^2_{L^2_{s+\gamma/2}} \Big\} \notag
\end{align}
for any $s' \in ]0,s[ $ satisfying $\gamma +2s' >-3/2$ and
$s' <3/4$.

Furthermore, if $\gamma >-3/2$ then the second term on the right hand
side can be replaced by $C
\|f\|_{L^2}\|g\|^2_{L^2_{s+\gamma/2}}$.
\end{lemm}



\begin{lemm}\label{upp-l-2} Let $0<s<1$, $\gamma > \max\{-3, -2s-3/2\}$.  For any $\ell \in \RR$ and $m \in
[0,s]$ we have
\begin{align}\label{different-g-h}
\Big | \Big(\Gamma^t( f, \, g),\, h \Big ) _{L^2} \Big|
\lesssim
\|f\|_{L^2}
\|g\|_{H^{s+m}_{\ell + (\gamma+2s)^+}}\|h\|_{H^{s-m}_{-\ell}}\,.
\end{align}
Furthermore
\begin{align}\label{same-g}
\Big | \Big(\Gamma^t( f, \, g),\, g \Big ) _{L^2} \Big|
\lesssim \|f\|_{L^2} |||g|||^2_{\Phi_{\gamma}}\,.
\end{align}
\end{lemm}
\begin{proof}
Since Lemma \ref{upper-triple-Dirichlet} holds with $\sqrt \mu$ replaced by $\mu_\kappa(t,v)$, in view of \eqref{IV-2.2}
we have
\[
\Big|\cD(\mu_\kappa f, g)\Big| \lesssim \|f\|_{L^2}|||g|||^2_{\Phi_\gamma}
\lesssim \|f\|_{L^2} \|g\|^2_{H^s_{(s+\gamma/2)^+}}\,.
\]
Applying this to the right hand side of \eqref{diff-G-Q}, by Proposition \ref{prop-IV-2.5b} we obtain
the desired estimate \eqref{different-g-h}. The second estimate can be obtained by using
Proposition \ref{prop-IV-2.5} instead of   Proposition \ref{prop-IV-2.5b}.
\end{proof}

For $\alpha >3/2$, set $\varphi(v,x) = (1+|v|^2 +|x|^2)^{\alpha/2}$ and
\[
W_{\varphi,l} = \frac{\la v \ra^l}{\varphi(v,x)}= \frac{(1+|v|^2)^{\ell/2}}
{(1+|v|^2 +|x|^2)^{\alpha/2}} .
\]

\begin{lemm}\label{weight}
If  $\varphi(v,x) = (1+|v|^2 +|x|^2)^{\alpha/2}$ for $\alpha  >3/2$ and
\[
W_{\varphi,l} = \frac{\la v \ra^l}{\varphi(v,x)}= \frac{(1+|v|^2)^{\ell/2}}
{(1+|v|^2 +|x|^2)^{\alpha/2}} .
\]
for $\ell \in \RR$ then we have
\begin{align}\label{dif-estime}
|\partial_x^{\beta_1} \partial_v^{\beta_2} W_{\varphi,\ell}(v)|
\lesssim
 \la v \ra^{-|\beta_1| -|\beta_2|}\, W_{\varphi,\ell}(v)\,,
\end{align}
\begin{align}\label{wei-estime}
&\Big|W_{\varphi,\ell}(v')  - W_{\varphi,\ell} (v)\Big|\lesssim
\sin\Big(\frac{\theta}{2}\Big) \frac{|v-v_*'| \la v_*' \ra^{\alpha+ |\ell-1|}}{\la v\ra}
W_{\varphi,\ell}(v)\,,\\
&\Big|\nabla^2 W_{\varphi,\ell}(v+\tau(v'-v))\Big|\lesssim
 \frac{\la v_* \ra^{\alpha+ |\ell-2|}}{\la v\ra^2}W_{\varphi,\ell}(v)\,\enskip,
 \enskip  \tau \in[0,1]\,.\label{wei-est-2}
\end{align}

\end{lemm}
\begin{proof}
The first inequality follows from the direct calculation.
Since
\[
W_{\varphi,\ell}(v')  - W_{\varphi,\ell} (v) = \int_0^1
\nabla_v W_{\varphi,\ell}(v+\tau(v'-v)) d\tau \cdot(v'-v)\,
\]
and since $|v-v'|=\sin (\theta/2)|v-v_*|, \, |v-v_*| \sim |v-v_*'|$,
for the proof of \eqref{wei-estime} it suffices to show
\begin{equation}\label{shyom}
\Big|\nabla W_{\varphi,\ell}(v_\tau) \Big| \lesssim \frac{W_{\varphi,\ell}(v_\tau)}
{\la v_\tau \ra}
\lesssim \frac{\la v_*' \ra^{\alpha+ (\ell-1)^+}}{\la v\ra}
W_{\varphi,\ell}(v)\,\enskip, \enskip  v_\tau =v+\tau(v'-v)\,.
\end{equation}
For $a\in \RR$ we have
\begin{align*}
1+a^2+ |v|^2 &\lesssim 1+ a^2+ |v-v_*'|^2 + |v_*'|^2 \lesssim 1+ a^2+ |v_\tau-v_*'|^2 + |v_*'|^2\\
&\lesssim 1+ a^2+ |v_\tau|^2 + |v_*'|^2 \lesssim (1+ a^2+ |v_\tau|^2 )\la v_*'\ra^2\,,
\end{align*}
from which we get $\varphi(v_\tau,x)^{-1} \lesssim \varphi(v,x)^{-1}\la v_*'\ra^\alpha$
by setting $a=|x|$.  Putting $a=0$ in the above inequality we have
$\la v \ra\lesssim \la v_\tau \ra \la v_*' \ra$. Since
$\la v_\tau \ra\lesssim \la v \ra \la v_*' \ra$ holds similarly
we have $\la v_\tau \ra^{\ell-1} \lesssim \la v \ra^{\ell-1} \la v_*' \ra^{|\ell-1|}$,
which concludes the second inequality of \eqref{shyom}. \eqref{wei-est-2}
also follows from the similar observation.
\end{proof}
\begin{lemm}\label{commu-Gam} Let $0<s<1$ and  $\gamma>\max\{-3, -2s-3/2\}$. Then for any $\ell \ge 0$ we have
\begin{align*}
&\Big|\Big( W_{\varphi,\ell} \Gamma^t(f,g) -\Gamma^t(f, W_{\varphi, \ell}g), h \Big)_{L^2}\Big|\\
&\lesssim \Big(\cD(\mu_\kappa \,{ |f|},  h) \Big)^{1/2}
\|f\|_{L^2} ^{1/2}
\|  W_{\varphi,\ell} g \|_{L^2_{\gamma/2}}
+ \|f\|_{H^{(2s'-1)^+}}
\|  W_{\varphi,\ell}  g\|_{L^2_{s' +\gamma/2}}
\|h \|_{L^2_{s+\gamma/2 } } \,,\notag
\end{align*}
for any $s' \in ](2s-1)^+, s[$ satisfying
$\gamma+2s' > -3/2$.
\end{lemm}
\begin{proof}
Note that
\begin{align*}
&\Big( W_{\varphi,\ell} \Gamma(f,g) -\Gamma(f, W_{\varphi, \ell}g), h \Big)_{L^2}=  \iiint B\,  f_{*} \, \, \Big(\mu_{*} h\Big)' g
\big ( W'_{\varphi,l} - W_{\varphi,l}  \big) dvdv_* d\sigma\\
&= \iiint B\,  f_{*} \, \, \Big\{ \Big({\mu_{*}}^{1/2} h  \Big)'
-  \Big({\mu_{*}}^{1/2} h\Big) \Big\} \,   {\mu'_{*}}^{1/2} \,
g
\big ( W'_{\varphi,l} - W_{\varphi,l}  \big) dvdv_* d\sigma \\
&\qquad +
\iiint B\,  f_{*} \, \, \Big({\mu_{*}}^{1/2} h \Big) \Big(
{\mu'_{*}}^{1/2} - {\mu_{*}}^{1/2}\Big)\,
g \big ( W'_{\varphi,l} - W_{\varphi,l} \big) dvdv_* d\sigma \\
&\qquad +
\iiint B\,  f_{*} \, \, \Big({\mu_{*}}^{1/2} h \Big)
{\mu_{*}}^{1/2} \, g
\big ( W'_{\varphi,l} - W_{\varphi,l}  \big) dvdv_* d\sigma \\
&=\cA_{1} + \cA_{2} + \cA_{3}\,.
\end{align*}
It follows from the Cauchy-Schwarz inequality that
\begin{align*}
\cA_{1}^2 \leq & \iiint
B\,   \, { |f_{*}|}\, \Big\{ \Big( {\mu_{*}}^{1/2} h\Big)'
-  \Big({\mu_{*}}^{1/2} h \Big) \Big\}^2dvdv_* d\sigma \\
&\quad \times
\iiint B\, { |f_{*}|} \,  \Big( {\mu'_{*}}^{1/2}  g
\big ( W'_{\varphi,l} - W_{\varphi,l}  \big) \Big)^2 dvdv_* d\sigma \\
=& \cA_{1,1}\times \cA_{1,2}\,.
\end{align*}
Writing
\[
 \Big({\mu_{*}}^{1/2} h \Big)'
-  \Big({\mu_{*}}^{1/2} h \Big)
= {\mu_{*}}^{1/2} \Big( h' - h\Big)
+  h'  \Big(  {\mu'_{*}}^{1/2}  - {\mu_{*}}^{1/2}   \Big)
\]
we obtain easily
\begin{align*}
\cA_{1,1}& \leq 2\Big \{\iiint B  (\mu \,{ |f|})_{*} \,
\Big( h'
-  h \Big)^2dvdv_* d\sigma \\
&\quad +
\iiint B \,{  |f_{*}| }\,
 \Big(  {\mu'_{*}}^{1/2}  - {\mu_{*}}^{1/2}   \Big)^2
{h '}^2dvdv_* d\sigma  \Big \}\\
&= 2 (\cD(\mu \,{ | f|}, h) +  \widetilde D_3(f,h))\\
&\lesssim \cD(\mu \,{| f|}, h) +\|f\|_{L^2}
\|h\|^2_{L^2_{s+\gamma/2}}
\,,
\end{align*}
where we have used \eqref{wide-D-3}.
By \eqref{wei-estime} and the Cauchy-Schwarz inequality  we have
\begin{align*}
\Big|\cA_{1,2}\Big|^2
&\lesssim \Big(\iiint
b\sin^2 (\theta/2)\,\frac{|v-v'_*|^{\gamma+2}}{\la v \ra^2}{\mu'_{*}}^{1/2}
{ |f_{*}|}\,\big(W_{\varphi,l}g\big)^2 dv dv_*d\sigma
\Big)^2
\\
&\lesssim
\Big(\iiint b\sin^2 (\theta/2)\,
\frac{|v-v'_*|^{2(\gamma+2)}}{\la v \ra^{\gamma+4}}{\mu'_{*}}\big(W_{\varphi,l}g\big)^2
d\sigma dv'_*dv  \Big)\\
& \times
\Big(\iiint    b\sin^2 (\theta/2)\,   f_{*}^2\la v \ra^{\gamma}
{\Big(W_{\varphi,l} g\Big)}^2 d\sigma dv_*dv  \Big)\\
&\lesssim \|W_{\varphi,l} g\|^2_{L^2_{\gamma/2}} \Big(\|f\|^2_{L^2}
\|W_{\varphi,l} g\|^2_{L^2_{\gamma/2}}\Big)\,,
\end{align*}
where we have used the change of variables $v_* \rightarrow v_*'$.
Hence we have
\begin{align*}
|\cA_{1}| &\lesssim  \Big(\cD(\mu {| f|}, h) + \|f\|_{L^2}^{1/2} \|h\|_{L^2_{s+\gamma/2}}\Big)
\|f\|^{1/2}_{L^2}
\|W_{\varphi,l} g\|_{L^2_{\gamma/2}}\,.
\end{align*}
By using the similar formula as \eqref{wei-estime} with $v_*'$ replaced by $v_*$  we have %
\[
\Big|\mu_{*}^{1/4} \Big({\mu'_{*}}^{1/2} - {\mu_{*}}^{1/2}\Big)\,
\big ( W'_{\varphi,l} - W_{\varphi,l} \big) \Big|
\lesssim \min \Big(|v-v_*|^2 \theta^2, |v-v_*| \theta\Big) W_{\varphi,l}\la v\ra^{-1}\,,
\]
so that for any $\delta>0$
\begin{align*}
|\cA_{2}|&\lesssim
\iiint |v-v_*|^\gamma\Big(|v-v_*|^{2s} + {\bf 1}_{|v-v_*|\gtrsim 1}|v-v_*|^{(2s-1+\delta)^++1}\Big)\mu_{*}^{1/4}\left|
f_{*} \frac{ \big(W_{\varphi,l} g \big) h}{\la v\ra} \right|dvdv_*\\
&\lesssim
\|f\|_{L^2}
 \| W_{\varphi,l} g \|_{L^2_{(2s-1+\delta)^+/2+\gamma/2}}\|h\|_{L^2_{(2s-1+\delta)^+/2+\gamma/2}}.
\end{align*}
In order to estimate $\cA_{3}$ we use the Taylor expansion for
$ W_{\varphi,l} - W_{\varphi,l} '$ of second order. Then
 we have
\begin{align*}
\cA_{3}= &\iiint B\, f_{*}\, \, {\mu_{*}} g h
\big ( \nabla_v W_{\varphi,l}\big) (x,v)  \cdot ( v' -v) dvdv_* d\sigma \\
&+ \frac{1}{2} \int_0^1 d \tau \iiint B\,  f_{*} \, \, {\mu_{*}}
 g h
\,\big(\nabla^2_v W_{\varphi,l} \big)(x,v+\tau(v'-v))(v'-v)^2 dvdv_* d\sigma \\
= & \cA_{3,1} + \cA_{3,2}\,\,.
\end{align*}
 Setting $ {\bf k} = \frac{v-v_*}{|v-v_*|}$ and
writing
\[v' -v = \frac{1}{2}|v-v_*|\Big( \sigma - (\sigma \cdot {\bf k})  {\bf k} \Big)
+ \frac{1}{2} ((\sigma \cdot {\bf k})-1)( v-v_*),
\]
 we have
$$
\cA_{3,1} =   \frac{1}{2}
\iiint B\,  f_{*} \, \mu_{*} g h
\big ( \nabla_v W_{\varphi,l}\big) (x,v)  \cdot (v-v_*)
\, \big(\cos \theta -1 \big)  dvdv_* d\sigma
$$
because it follows from  the symmetry that $\int_{\SS^2} b(\sigma
\cdot {\bf k})\Big( \sigma - (\sigma \cdot {\bf k})  {\bf k} \Big) d
\sigma =0$. Therefore, for any $0 \le s' < s$ satisfying
$\gamma+2s' < -3/2$ we have, in view of \eqref{dif-estime},
\begin{align*}
|\cA_{3,1}| &\lesssim \int \Big(\int|v-v_*|^{2(\gamma+2s')} \mu_{*}^2 dv_*\Big)^{1/2}
\Big( \int\frac{f_{*}^2}{|v-v_*|^{2(2s'-1)^+}}dv_*\Big)^{1/2}  |W_{\varphi,l-1}g| |h| dv \\
&\lesssim \|f\|_{H^{(2s'-1)^+}} \| W_{\varphi,l} g \|_{L^2_{s'+\gamma/2}}\|h\|_{L^2_{s'+\gamma/2}} \,,
 \end{align*}
 by means of $|\nabla_v
(W_{\varphi,l} ) (x, v)|\lesssim  W_{\varphi,l-1}$.
 The better  bound holds for $|\cA_{3,2}|$ since it follows from
 \eqref{wei-est-2} that
 \[
 |\big(\nabla^2_v W_{\varphi,l} \big)(v+\tau(v'-v))(v'-v)^2|\lesssim
\la v_*\ra^{\alpha+|\ell-2|} \, W_{\varphi,l-2}{\theta}^{2}|v-v_*|^2\,\,.
 \]

Therefore we have
\[
|\cA_{3}| \lesssim \|f\|_{H^{(2s'-1)^+}} \| W_{\varphi,l} g \|_{L^2_{s'+\gamma/2}}
\|h\|_{L^2_{s+\gamma/2}}\,.
\]
Summing up above estimates we obtain the conclusion.
\end{proof}

\subsection{Proofs of the uniqueness Theorems}

Using the notations introduced in subsection
\ref{subsect4.1},
the proof of Theorem \ref{uniqueness} is reduced to

\begin{prop}\label{prop-unique}  Assume that $0<s<1$ and $\max \big (-3, -3/2-2s \big ) <\gamma < 2-2s $. Let $\ell_0 > 3/2 + \max\{1, (\gamma+2s)^+\}$ and
$g_0\in L^\infty(\RR_x^3;\, L^2_{\ell_0}(\RR^3_v))$. Suppose that the Cauchy problem
\eqref{E-Cauchy-B} admits two solutions
$$
g_1,\, g_2\in L^\infty([0, T]\times \RR_x^3 \,;\, H^{m}_{\ell_0}({\mathbb R}^3_{v}))\,.
$$
Then $g_1 \equiv g_2 $ in $[0, T]$,

\noindent
{\em 1)}\,\,  if $m=2s$ and $g_1 \ge 0$. When $\gamma >-3/2\,$, we can suppose  $g_1 \in  L^\infty([0, T]\times \RR_x^3 \,;\, H^{s}_{\ell_0}({\mathbb R}^3_{v}))$.

\noindent
{\em 2)}\,
if $m=s$ and the coercivity inequality \eqref{cor-sol} is satisfied for $f_1 = \mu_\kappa(t) g_1\ge 0$.

\noindent
{\em 3)}\,
if $m=s$ and  $f = \mu_\kappa(t) g_1$ satisfies the following strong coercivity estimate
\begin{align}\label{cor-sol-strong}
 -(Q(f(t),h),h)_{L^2(\RR^6)}  \ge c_0 \int |||h|||^2_{\Phi_\gamma}dx  -C\|h\|^2 _{L^2(\RR^3_x; L^2_{ (\gamma/2+s)^+}(\RR^3_v))}\,.
\end{align}

\end{prop}

\begin{proof}
Let $S(\tau) \in C_0^\infty(\RR)$ satisfy $0 \leq S \leq 1$ and
$$S(\tau)=1, \enskip |\tau| \leq 1 \enskip; \enskip S(\tau)=0, \enskip |\tau| \geq 2. $$
Set $S_N(D_x) = S(2^{-2N}|D_x|^2)$ and multiply $W_{\varphi, l} S_N(D_x)^2 W_{\varphi, l} g$
to \eqref{E-Cauchy-B^diff},
where we choose $\ell, \alpha$ such that
\[
\ell_0 - \max\{1, \, (\gamma+2s)^+\} > \ell  > \alpha >3/2.\]
Integrating and letting $N \rightarrow \infty$, we have
\begin{align*}
\frac 12 \frac{d}{d t} \|  W_{\varphi, l} g(t)\|^2_{L^2(\RR^6)} + \kappa \|
W_{\varphi, l+1} g(t)\|^2_{L^2(\RR^6)}&=\Big(W_{\varphi,l}\,
\Gamma^t (g_1,\, g)+W_{\varphi, l}\, \Gamma^t(g,\, g_2)\, , W_{\varphi,l} g\Big)_{L^2(\RR^6)}\\
&- ( v\cdot \nabla_x (\varphi^{-1}) W_l g, W_{\varphi,l} g)_{L^2(\RR^6)}, 
\end{align*}
because
$( v\cdot \nabla_x S_N(D_x)W_{\varphi, l }g, S_N(D_x)W_{\varphi,l} g)_{L^2(\RR^6)} =0$.
The second  term on the right hand side is estimated by $\|W_{\varphi,l} g\|^2_{L^2(\RR^6)}$
because of \eqref{dif-estime}.
Write the first term on the right hand side as
\begin{align*}
&
\Big(
\Gamma^t (g_1,\, W_{\varphi,l}\,g)\, , W_{\varphi,l} g\Big)_{L^2(\RR^6)}
+
\Big(W_{\varphi,l}\,
\Gamma^t (g_1,\, g)-  \Gamma^t(g_1 ,\, W_{\varphi, l}\,g)\, , W_{\varphi,l} g\Big)_{L^2(\RR^6)}\\
&+ \Big(\,
\Gamma^t(g,\, g_2)\, , W_{\varphi,l}\, W_{\varphi,l} g\Big)_{L^2(\RR^6)} = \cB_1 + \cB_2 +\cB_3\,.
\end{align*}
If $g_1 \ge 0$ then it follows from Lemma \ref{coer-gamma} that
\begin{align}\label{B-1}
&\cB_1 \le  -  \frac 1 4\int \cD(\mu_\kappa g_1 , W_{\varphi,l} g) dx  \\
& + C \min\Big(
\|g_1\|_{L^\infty_{t,x}(L^2)} \big( \|W_{\varphi,l} g\|^2_{L^2_x(H^{s'}_{\gamma/2})}
+ \|W_{\varphi,l}g\|^2_{L^2_x(L^2_{s+\gamma/2})}\big)\,\,,\,\,
\|g_1\|_{L^\infty_{t,x}(H^{2s'})} \|W_{\varphi,l}g\|^2_{L^2_x(L^2_{s+\gamma/2})} \Big) \,.\notag
\end{align}
We notice that the last term can be replaced by $\|g_1\|_{L^\infty_{t,x}(L^2)} \|W_{\varphi,l}g\|^2_{L^2_x(L^2_{s+\gamma/2})}$
if $\gamma >-3/2$.
By means of Lemma \ref{commu-Gam} we obtain for any $\delta >0$
\begin{align}\label{B-2}
\Big| \cB_2\Big| \le \delta  \int \cD(\mu_\kappa  |g_1|\, , W_{\varphi,l} g) dx
+ C_\delta \|g_1\|_{L^\infty_{t,x}(H^{(2s-1)^+})} \|W_{\varphi,l}g\|_{L^2_x(L^2_{s'+\gamma/2})} \|W_{\varphi,l}g\|_{L^2_x(L^2_{s+\gamma/2})}\,.
\end{align}
Lemma \ref{upp-l-2} with $\ell = \ell -\gamma/2$ implies that for $m= s, 0 $
\begin{align}\label{B-3}
\Big |\cB_3 \Big| &\lesssim
\int \|\la x \ra^{-\alpha} g\|_{L^2(\RR^3)} \|g_2\|_{H^{s+m}_{{ \ell}-\gamma/2 + (2s+\gamma)^+}(\RR^3)}
\|\frac{\la x \ra^\alpha}{\la v \ra^{\ell}}W_{\varphi,l} W_{\varphi,l} g\|_{H^{s-m}_{\gamma/2}} dx \\
&\lesssim
\|g_2\|_{L^\infty_{t,x}(H^{s+m}_{{ \ell}-\gamma/2 + (2s+\gamma)^+})}
\|W_{\varphi,\alpha}g\|_{L^2_x(L^2_{s+\gamma/2})} \|W_{\varphi,\ell}g\|_{L^2_x(H^{s-m}_{\gamma/2})} \,,\notag
\end{align}
because $\la x\ra^{-\alpha} \leq W_{\varphi, \alpha}$ and ${\la x \ra^\alpha}{\la v \ra^{-\ell}}W_{\varphi,l} $ is
a bounded operator on $H^{s-m}$.
Note that for any $\delta >0 $
\[
\|W_{\varphi, \ell+(s+\gamma/2)^+} g(t) \|^2_{L^2(\RR^6)}  \leq             \delta  \|
W_{\varphi, \ell+1} g(t)\|^2_{L^2(\RR^6)} + C_\delta \|W_{\varphi, \ell} g(t)\|^2_{L^2(\RR^6)}\,.
\]
If $g_1, g_2 \in \tilde{\cE}^{2s}([0,T] \times \RR^6)$ and $g_1 \ge 0$ then by summing up \eqref{B-1}, \eqref{B-2} and \eqref{B-3} with
$m=s$ we have
$$
\frac{d}{d t} \|W_{\varphi,l} g(t)\|^2_{L^2(\RR^6)}
\leq C_\kappa \left(
\|g_1\|_{L^{\infty}_{t,x}(H^{2s})} +
\|g_2\|_{L^\infty_{t,x}( H^{2s}_{\ell-\gamma/2 + (2s+\gamma)^+})}
\right)\|W_{\varphi,l}g(t)\|^2_{L^2(\RR^6)},
$$
where $\|g_1\|_{L^{\infty}_{t,x}(H^{2s})}$ can be replaced by $\|g_1\|_{L^{\infty}_{t,x}(L^2)}$
if $\gamma >-3/2$. Here it should be noted that the term $\cD(\mu_\kappa g_1 , W_{\varphi,l} g)  \ge 0$ follows from
the non-negativity of $g_1$.
Therefore, $\| W_{\varphi,l}g(0)\|_{L^2(\RR^6)}=0$ implies $\|
W_{\varphi,l}g(t)\|_{L^2(\RR^6)}=0$ for all $t\in [0, T]$.
And this gives $g_1=g_2$, and concludes the  part 1)  of Proposition  \ref{prop-unique}.

For the part 2) of Proposition \ref{prop-unique},  using $\|h\|^2_{H^{s'}_{\gamma/2} }\le \delta \|h\|^2_{H^{s}_{\gamma/2}} + C_\delta\|h\|^2_{L^2_{\gamma/2}}$
and summing up \eqref{B-1}, \eqref{B-2} and \eqref{B-3} with
$m=0$, we have
\begin{align*}
&\frac{d}{d t} \|W_{\varphi,l} g(t)\|^2_{L^2(\RR^6)} \leq -\frac{1}{16} \Big(\int \cD(f_1(t), W_{\varphi,l} g(t)) dx
+ \kappa \|
W_{\varphi, \ell+1} g(t)\|^2_{L^2(\RR^6)} \Big)\\
& + \delta
 \left(
\|g_1\|_{L^{\infty}_{t,x}(H^{(2s-1)^+})} +
\|g_2\|_{L^\infty_{t,x}( H^{s}_{\ell-\gamma/2 + (2s+\gamma)^+})}
\right)\|W_{\varphi,l}g(t)\|^2_{L^2_x (H^s_{\gamma/2})} \\
&+ C_{\kappa,\delta} \left(
\|g_1\|_{L^{\infty}_{t,x}(H^{(2s-1)^+})} +
\|g_2\|_{L^\infty_{t,x}( H^{s}_{\ell-\gamma/2 + (2s+\gamma)^+})}
\right)\|W_{\varphi,l}g(t)\|^2_{L^2(\RR^6)},
\end{align*}
then the coercivity condition \eqref{cor-sol}(with  $(\gamma/2+s)^+<1$) together with \eqref{df} leads us to
\begin{align}\label{d-e}
&\frac{d}{d t} \|W_{\varphi,l} g(t)\|^2_{L^2(\RR^6)}
\lesssim \left(
\|g_1\|_{L^{\infty}_{t,x}(H^{(2s-1)^+})} +
\|g_2\|_{L^\infty_{t,x}( H^{s}_{\ell-\gamma/2 + (2s+\gamma)^+})}
\right)\|W_{\varphi,l}g(t)\|^2_{L^2(\RR^6)}\,,
\end{align}
where it should be noted that \eqref{df} holds with $\mu^{1/2}$ replaced by $\mu_\kappa(t,v)$.
Thus, the part 2) of Proposition  \ref{prop-unique} is proved.

When $g_1$ is not necessarily non-negative, by using Lemma \ref{differ-Gam-Q} we obtain
 \begin{align*}
\cB_1 &\le  \Big( Q(\mu_\kappa g_1 , W_{\varphi,l} g)\, , \,W_{\varphi,l} g\Big)_{L^2(\RR^6)} +
\delta \int \cD(\mu_\kappa \,|g_1|\,  , W_{\varphi,l} g) dx
+ C_\delta
 \|W_{\varphi,l}g\|^2_{L^2_x(L^2_{s+\gamma/2})} 
\end{align*}
instead of \eqref{B-1}.  Since Lemma \ref{upper-triple-Dirichlet} holds with $\sqrt \mu$ replaced by
$\mu_\kappa$\,, by means of \eqref{cor-sol-strong} we get
 \begin{align}\label{B-1-dautre}
\cB_1 &\le- (c_0 -\delta)  \int||| W_{\varphi,l} g)|||^2_{\Phi_\gamma}dx
+ C_\delta
 \|W_{\varphi,l}g\|^2_{L^2_x(L^2_{(s+\gamma/2)^+})} \,.
\end{align}
This estimate and \eqref{B-3}, together with \eqref{B-2} applied by Lemma \ref{upper-triple-Dirichlet}, imply
\eqref{d-e}. Hence the part 3) of Proposition \ref{prop-unique} is also proved.
\end{proof}

\noindent
{\bf Proof of Theorem \ref{uniqueness-small-pur} :}

If we set $g_j(t) = \mu^{-1/2}f_j(t)$ ($j=1,2$) and $g=g_1-g_2$, then we have
$$
\left\{\begin{array}{l}
g_t+v\cdot\nabla_x g =\Gamma (g_1, g)+\Gamma (g,g_2),\\
g|_{t=0}=0\,,
\end{array}\right.
$$
where
$
\Gamma (g, h)=\mu^{-1/2}Q(\mu^{1/2}g, \mu^{1/2}h)$. Take the inner product with $W_{\varphi, \ell} S_N(D_x)^2 W_{\varphi, \ell} g$ {where we choose
$\ell, \alpha$ such that $\ell_1 -(\gamma+2s) > \ell> \alpha>3/2$.  Then we
obtain
\begin{align*}
\frac 12 \frac{d}{d t} \|  W_{\varphi, \ell} g(t)\|^2_{L^2(\RR^6)} &=\Big(W_{\varphi,\ell}\,
\Gamma (g_1,\, g)+W_{\varphi, \ell}\, \Gamma(g,\, g_2)\, , W_{\varphi,\ell} g\Big)_{L^2(\RR^6)}\\
&- ( v\cdot \nabla_x (\varphi^{-1}) W_\ell g, W_{\varphi,\ell} g)_{L^2(\RR^6)}, 
\end{align*}
where the second  term on the right hand side is estimated by $\|W_{\varphi,\ell} g\|^2_{L^2(\RR^6)}$ because of  \eqref{dif-estime}.
We write the first term on the right hand side as
\begin{align*}
&
\Big(
\Gamma (g_1,\, W_{\varphi,\ell}\,g)\, , W_{\varphi,\ell} g\Big)_{L^2(\RR^6)}
+ \widetilde \cB_2 + \widetilde \cB_3\,,
\end{align*}
where $\widetilde \cB_2, \tilde \cB_3$ are defined by the same way as the above $\cB_2, \cB_3$ with $\Gamma^t$ replaced by $\Gamma$ and satisfy
the similar  estimates as \eqref{B-2} and \eqref{B-3}, respectively, that is,
\begin{align*}
\Big| \widetilde \cB_2\Big|& \le \delta  \int \cD(\mu^{1/2} {|g_1|} , W_{\varphi,\ell} g) dx
+ C_\delta \|g_1\|_{L^\infty_{t,x}(H^{(2s-1)^+})} \|W_{\varphi,\ell}g\|_{L^2_x(L^2_{s'+\gamma/2})} \|W_{\varphi,\ell}g\|_{L^2_x(L^2_{s+\gamma/2})}\,,
\\
\Big |\widetilde \cB_3 \Big|
&\lesssim
\|g_2\|_{L^\infty_{t,x}(H^{s}_{{\ell }-\gamma/2 + (2s+\gamma)^+})}
\|W_{\varphi,\alpha}g\|_{L^2_x(L^2_{\gamma/2})} \|W_{\varphi,\ell}g\|_{L^2_x(H^{s}_{\gamma/2})} \\
&\lesssim
\|g_2\|_{L^\infty_{t,x}(H^{s}_{{\ell} -\gamma/2 + (2s+\gamma)^+}}\Big( \delta \int |||W_{\varphi,\ell}g|||^2_{\Phi_\gamma}dx +
C_\delta \|W_{\varphi,\ell}g\|_{L^2(\RR^6)}\Big)
 \,,\notag
\end{align*}
where the non-isotropic norm $|||\,\cdot\,|||_{\Phi_\gamma}$ is recalled in  \eqref{trip-norm} .
By means of Lemma \ref{upper-triple-Dirichlet},  we have
\begin{align*}
\Big| \widetilde \cB_2\Big|&
\le \delta \|g_1\|_{L^\infty_{t,x}(H^{(2s-1)^+})} \, \int||| W_{\varphi,\ell}g|||^2_{\Phi_\gamma} dx
+ C'_\delta \|g_1\|_{L^\infty_{t,x}(H^{(2s-1)^+})} \|W_{\varphi,\ell}g\|^2_{L^2(\RR^6)}\,.
\end{align*}
On the other hand,  it follows from Proposition 2.1 of \cite{amuxy4-2} and \eqref{same-g}  that for suitable $C_1, C_2 >0$
we have
\begin{align}\label{gamma-hyouka}
 &
\Big(
\Gamma (g_1,\, W_{\varphi,\ell}\,g)\, , W_{\varphi,\ell} g\Big)_{L^2(\RR^6)} =
-\Big (\cL_1 \, \big(W_{\varphi,\ell} g\big)\, ,W_{\varphi,\ell} g\Big)_{L^2(\RR^6)}  +\Big( \Gamma (\tilde g_1,\, W_{\varphi,\ell} g)\, ,W_{\varphi,\ell} g \Big)_{L^2(\RR^6)} \\
& \le - C_1\int||| W_{\varphi,\ell}g|||^2_{\Phi_\gamma} dx\notag \\
& \quad + C_2\Big( \big(\sup_{[0,T]\times \RR_x^3} \|\tilde
g _1 \|_{L^2(\RR^3_v)} \big)\int||| W_{\varphi,\ell}g|||^2_{\Phi_\gamma} dx+ ||W_{\varphi,\ell}g||^2_{L^2(\RR^6}\Big)\,,\notag
\end{align}
where $g_1 =\sqrt\mu +\tilde g_1$.
Therefore, \eqref{IV-2.2} and the smallness condition \eqref{small} imply
\begin{align*}
&\frac{d}{d t} \|W_{\varphi,\ell} g(t)\|^2_{L^2(\RR^6)}\\
& \le -\Big( C_1 - C_2 \varepsilon_0
-\delta\big( \|g_1\|_{L^\infty_{t,x}(H^{(2s-1)^+})} + \|g_2\|_{L^\infty_{t,x}( H^s_{ \ell+2s+\gamma/2})}\big ) \Big )
\int||| W_{\varphi,\ell}g|||^2_{\Phi_\gamma} dx\\
&+  C_\delta \left(
\|g_1\|_{L^{\infty}_{t,x}(H^{(2s-1)^+})} +
\|g_2\|_{L^\infty_{t,x}( H^s_{ \ell+ 2s+\gamma/2})}
\right)\|W_{\varphi,\ell}g(t)\|^2_{L^2(\RR^6)}.
\end{align*}
which  shows  $g(t) =0$ for all $t \in[0,T]$  if $\varepsilon_0 < C_1/C_2$. Thus we have proved  Theorem \ref{uniqueness-small-pur}.

\smallbreak
\noindent
{\bf Proof of Theorem \ref{unique-x-global}}
\smallbreak

Let now $f_j(t) \in \tilde {\mathcal B}^s([0,T]\times{\mathbb R}^6_{x, v})$, ($j=1,2$)
and set
$g_j(t) = \mu_\kappa(t)^{-1} f_j(t)$ for a suitable
$\mu_\kappa(t) =e^{-(\rho-\kappa t)(1+ |v|^2)}
$. Then we have for any $\ell \in \NN$
\[
g_j(t)  \in L^\infty([0,
T]\times \RR^3_x ;\,\, L^2_\ell({\mathbb R}^3_{v})) \mathop{\cap} L^2([0,T]; L^\infty(\RR_x^3; H^m_{\ell}(\RR^3_v))).
\]
The proof of Theorem \ref{unique-x-global} is reduced to

\begin{prop}\label{unique-prop-th1-2}
Assume that $0<s<1$ and $\max \{ -3, -3/2-2s \} <\gamma \le -2s $.
 Let $ 0<T < +\infty$ and $\ell_2 \ge 3 $.
Suppose that the Cauchy problem
\eqref{E-Cauchy-B} admits two solutions
$$
g_1,\, g_2\in  L^\infty([0,
T]\times \RR^3_x ;\,\, L^2_{\ell_2}({\mathbb R}^3_{v})) \mathop{\cap} L^2([0,T]; L^\infty(\RR_x^3; H^s_{\ell_2}(\RR^3_v)))\,.
$$
If
\eqref{cor-sol-strong}
is satisfied for $f = \mu_\kappa(t) g_1$
then $g_1(t) \equiv g_2(t)$ for all $t \in [0,T]$\,.
\end{prop}

\begin{proof}
Noting $\gamma+2s \le 0$, we estimate more carefully $\cB_2, \cB_3$ in the proof of Proposition \ref{prop-unique}.
It follows from Lemma \ref{commu-Gam} and Lemma \ref{upper-triple-Dirichlet} that
\begin{align*}
\Big| \cB_2\Big| \le \delta  \int |||W_{\varphi,l} g|||_{\Phi_\gamma}dx
+ C_\delta \|g_1(t)\|_{L^\infty_{x}(H^s)} \|W_{\varphi,l}g(t)\|^2_{L^2(\RR^6)}\,.
\end{align*}
Lemma \ref{upp-l-2} with $\ell = \ell -\gamma/2$ and $m=  0 $ yields
\begin{align*}
\Big |\cB_3 \Big|
&\lesssim
\|g_2(t)\|_{L^\infty_{x}(H^{s}_{{ \ell}+|\gamma|/2 })}
\|W_{\varphi,\alpha}g(t)\|_{L^2(\RR^6)} \|W_{\varphi,\ell}g(t)\|_{L^2_x(H^{s}_{\gamma/2})}  \\
&\le \delta \|W_{\varphi,\ell}g(t)\|^2_{L^2_x(H^{s}_{\gamma/2})}  + C_\delta
\|g_2(t)\|^2_{L^\infty_{x}(H^{s}_{{ \ell}+|\gamma|/2 })}
\|W_{\varphi,\alpha}g(t)\|^2_{L^2(\RR^6)} \,.
\end{align*}
Above estimates for $\cB_j$ ($j=2,3$) and \eqref{B-1-dautre}  imply that 
$$
\max_{t \in [0,T_1]} \|W_{\varphi,l}g(t)\|^2_{L^2(\RR^6)} \le \|W_{\varphi,l}g(0)\|^2_{L^2(\RR^6)} +
 \varepsilon (T_1) \max_{t \in [0,T_1]} \|W_{\varphi,l}g(t)\|^2_{L^2(\RR^6)} \,,
$$
where
\[
 \varepsilon (T_1)  \lesssim  T_1 +T_1 \|g_1\|_{L^2([0,T_1];L_x^\infty(H^s))} + \|g_2\|^2_{L^2([0,T_1];L_x^\infty(H^s_{\ell+|\gamma|/2}))}.
\]
By assumption $\varepsilon (T_1)\rightarrow  0$ as $T_1 \rightarrow 0\,$.
Therefore there exists a $T_* >0$ such that $g(t) \equiv 0$ for $t \in [0,T_*]$. Replacing the initial time $0$ by $T_*$ if needed, we
finally obtain $g(t) \equiv 0$ for $t\in [0,T]$.
\end{proof}

\subsection{Uniqueness of known solutions}
 Firstly we consider the uniqueness of global solutions given in \cite{amuxy4-2, amuxy4-3}.
Theorem \ref{uniqueness-small-pur}  is applicable to show the uniqueness
of global solutions in Theorem 1.5 of \cite{amuxy4-2}, and also solutions in Theorem 1.1 of \cite{amuxy4-3}
because the global solutions given there are of the form $\mu + \sqrt \mu \tilde g$ with
\[
 \enskip \tilde g(t,x,v) \in L^\infty([0,\infty[; {\cH}_\ell^m(\RR^6))
\]
for $m \ge 6$ and a suitable $\ell$.  It should be noted that the uniqueness holds under the smallness condition
\eqref{small} of
the perturbation $\tilde g$, without the non-negativity of solution $\mu + \mu^{1/2}\tilde g $.

It follows from  Corollary
\ref{coro-coercivity-1} that the smallness condition \eqref{small} implies
\eqref{cor-sol-strong} for the global solution given in Theorem 1.4 of \cite{amuxy4-2}, because $\tilde g$ there
satisfies
 $\|\tilde g \|_{L^\infty([0,\infty) ; H^3(\RR^3_x ; L^2(\RR^3_v))} < \varepsilon_0$ and for any
$0<T<+\infty$
\[
\int_0^T \Big(\sum_{| \alpha|\le 3} \int |||\partial_x^\alpha \tilde g(t,x)|||^2_{\Phi_\gamma}dx \Big) dt < + \infty\,.
\]
Therefore Theorem \ref{unique-x-global} shows the uniqueness of the solution given in Theorem 1.4 of
\cite{amuxy4-2} by means of the Sobolev embedding.

In \cite{amuxy6},
bounded solutions of the Boltzmann equation 
in the whole space have been constructed  without specifying
any limit behaviors at the spatial infinity and without assuming the smallness condition on initial data.
More precisely, it has been shown  that if the initial data
is non-negative and belongs to a uniformly local Sobolev space
with the Maxwellian decay property in the velocity variable, then the Cauchy problem of the Boltzmann
equation possesses
a non-negative local solution in the same function
space, both for the cutoff and non-cutoff collision cross section  with
mild singularity. Since  solutions there are non-negative and belong to  ${ \mathcal E}^{2s}([0,T]\times{\mathbb R}^6_{x, v})$,
Theorem \ref{uniqueness} yields their uniqueness.

\section{Non-negativity of solutions}\label{sect-IV-8}
\smallskip \setcounter{equation}{0}
The purpose of this section is to show the non-negativity of solutions constructed in
\cite{amuxy4-2, amuxy4-3}, where the solution $f = \mu + \sqrt \mu g$ is a perturbation around a normalized Maxwellian distribution $\mu(v)$, that means $g$ is solution of following Cauchy problem :
\begin{equation}\label{IV-5.1}
\left\{\begin{array}{l} \partial_t g + v\,\cdot\,\nabla_x
g +\cL(g)=
\Gamma (g, g), \\ g|_{t=0} = g_0\, ,
\end{array} \right.
\end{equation}
where
$$
\cL(g)=-\Gamma(\sqrt{\mu}, g)-\Gamma(g, \sqrt{\mu})=\cL_1(g)+\cL_2(g).
$$
It is the limit of a sequence constructed successively by the following linear Cauchy problem,
\begin{equation}\label{4.4.3}
\left\{\begin{array}{l} \partial_t f^{n+1} + v\,\cdot\,\nabla_x
f^{n+1} =Q (f^n, f^{n+1}), \\ f^{n+1}|_{t=0} = f_0 =\mu + \mu^{1/2}
 g_0\geq 0\, ,
\end{array} \right.
\end{equation}
if one returns to the original Boltzmann equation.
Hence the non-negativity of solution comes  from the following induction argument: Let $f^0=f_0 = \mu + \mu^{1/2} g_0 \geq 0$, suppose  that
\begin{equation}\label{4.4.1+100}
f^n = \mu + \mu^{1/2}  \tilde g^n \geq 0\,,
\end{equation}
for some $n\in\NN$. Then  (\ref{4.4.1+100}) is true
for $n+1$.


\begin{prop}\label{induct-1}
Assume that $\max\{-3, -\frac32 -2s\}<\gamma<2-2s $.  Let $\{f^n\}$ is a sequence of solutions of Cauchy problem \eqref{4.4.3} with
\[
\exists \rho >0 \, \,; \,\, e^{\rho \la v\ra^2}  f^n(t,x,v)  \in L^\infty([0,T] \times \RR^3_x
;  H^N (\RR^3_v) )\enskip
\mbox{for} \enskip \forall n = 1,2,3,\cdots\,,
\]
for some $N \ge 4$. Then for any $n\in\NN$,  $f^n \geq 0$ on $ [0,T]$ implies $f^{n+1} \geq 0$  on the same interval.
\end{prop}
\begin{proof}
Taking a $\kappa >0$ such that $\frac{\rho}{2\kappa} > T$, we set
 $g^n(t,x,v) =\mu_\kappa(v)^{-1} f^n(t,x,v)$  with $\mu_\kappa(t) = e^{-(\rho-\kappa t)\la v\ra^2}$ then it follows from \eqref{4.4.3} that
\begin{equation}\label{maytta}
\partial_t g^{n+1} + v\,\cdot\,\nabla_x
g^{n+1} + \kappa \la v \ra^2  g^{n+1 }= \Gamma^t (g^n, g^{n+1})\,.
\end{equation}
We notice that for any $\ell \in \NN$
\[
g^n \in L^\infty([0,T] \times \RR^3_x
;  H^N_\ell (\RR^3_v) )\,
\]
so that $\sup_{t,x}  |||g^n |||_{\Phi_\gamma} < \infty$.
If $g$ satisfies $|||g||| < \infty$ and if $g_{\pm}= \pm \max(\pm g, 0)$, then we have
\[
|||g_{+} |||^2_{\Phi_\gamma} + |||g_-|||^2_{\Phi_\gamma} \leq |||g|||^2_{\Phi_\gamma},
\]
because
\begin{align*}
|||g|||^2_{\Phi_\gamma}=& \iiint B \mu_*\, \Big( (g'_{+}  + g'_{-} )-( g_+ + g_{-})\,\Big)^2\, \notag \\
&+
\iiint B  (g_{*,+}  +  g_{*,-}  ) ^2 \big(\sqrt{\mu'}\,\, - \sqrt{\mu}\,\,
\big)^2\, \, \\
=& |||g_{+} |||^2_{\Phi_\gamma} + |||g_-|||^2_{\Phi_\gamma}  - 2 \iiint b(\cos\theta) \Phi(|v-v_*|)\mu_*\, \big(g'_{+} g_{-}+  g_+ g'_{-}\big)
\end{align*}
and the third term is non-negative. Therefore $g^n_- \in L^\infty_{t,x} (H^s_\ell(\RR^3))$.
Take the convex function $\beta (s) = \frac 1 2 (s^- )^2= \frac 1 2
s\,(s^- )  $ with $s^-=\min\{s, 0\}$.  Let $\varphi(v,x) = (1+|v|^2 +|x|^2)^{\alpha/2}$ with $\alpha >3/2$, and notice that
$$
\beta_s (g^{n+1}) \varphi(v,x)^{-2}=\left(\frac{d}{ds}\,\,\beta
\right)(g^{n+1}) \varphi(v,x)^{-2}=g^{n+1}_{-} \varphi(v,x)^{-2}\in
L^\infty([0,T];L^1(\RR_x^3; L^2(\RR^3_v)),
$$
because $\tilde g^{n+1} \in L^\infty_t(H^N(\RR^6))$ with $N \geq 4$ implies $g^{n+1} \in L_{t,x}^\infty(L^2_v)$.
 Multiplying \eqref{maytta} by $\beta_s (g^{n+1})\varphi(v,x)^{-2}$ $
= g_-^{n+1}\varphi(v,x)^{-2}$ we have
\begin{align*}
&\frac{d}{dt} \int_{\RR^6} \beta ( g^{n+1}) \varphi(v,x)^{-2}dxdv
+\kappa \int_{\RR^6}   \la v \ra^2\beta ( g^{n+1}) \varphi(v,x)^{-2}dxdv\\
&\qquad =\int_{\RR^6}
\Gamma^t(g^n,\, g^{n+1})\,\, \beta_s(g^{n+1}) \varphi(v,x)^{-2} \,\,dxdv \\
&\qquad- \int_{\RR^6} { v\,\cdot\, \nabla_x \enskip  (  \beta
(g^{n+1}) \varphi(v,x)^{-2}) }dxdv - \int_{\RR^6} {\big (\varphi(v,x)^{2}\,\,v\, \cdot\,
\nabla_x \,\varphi(v,x)^{-2} \big) }\enskip  \beta (g^{n+1})\varphi(v,x)^{-2}dxdv,
\end{align*}
where the first term on the right hand side is well defined because $g^{n+1} , g^{n+1}_- \in
L^\infty_{t,x} (H^s_\ell)$.
Since the second term vanishes
and $|v\, \cdot\, \nabla_x\, \varphi(v,x)^{-2} | \leq C \varphi(v,x)^{-2}$,
we obtain
\begin{align}\label{gronwall}
&\frac{d}{dt} \int_{\RR^6}  \beta ( g^{n+1}) \varphi(v,x)^{-2}dxdv + \kappa \int_{\RR^6}   \la v \ra^2\beta (g^{n+1})
\varphi(v,x)^{-2} dxdv \\
&\quad \leq
\int_{\RR^6}  \Gamma^t(g^n, g^{n+1} ) \beta_s(g^{n+1}) \varphi(v,x)^{-2}dxdv + C
\int_{\RR^6}   \beta ( g^{n+1}) \varphi(v,x)^{-2}dxdv. \notag
\end{align}
The first term on the right hand side is equal to
\begin{align*}
&\int_{\RR^6}  \Gamma^t(g^n, g^{n+1}_{-} ) g^{n+1}_{-} \varphi(v,x)^{-2} dxdv
 + \iiiint   B \, \mu_{\kappa,*}( g^n_*)' ( g^{n+1}_{+})'  g^{n+1}_{-} \varphi(v,x)^{-2}dvdv_* d\sigma dx \\
 &=A_1 + A_2\,.
\end{align*}
{}From the induction hypothesis, the second term $A_2$ is non-positive.

On the other hand, 
we have
\begin{align*}
A_1 &= \int( \Gamma^t  (g^n, g_-^{n+1}),  \varphi(v,x)^{-2}g_-^{n+1})_{L^2(\RR_v^3)} dx
\\
&= \int( \Gamma^t  (g^n, \varphi(v,x)^{-1} g_-^{n+1}),  \varphi(v,x)^{-1}g_-^{n+1})_{L^2(\RR_v^3)} dx\\
&+  \int( \varphi(v,x)^{-1} \Gamma^t  (g^n, g_-^{n+1}) -\Gamma^t  (g^n, \varphi(v,x)^{-1} g_-^{n+1}),   \varphi(v,x)^{-1}g_-^{n+1})_{L^2(\RR_v^3)} dx. \\
&=A_{1,1} + A_{1,2}.
\end{align*}
It follows from Lemma \ref{coer-gamma}  that
\[
A_{1,1} \leq  -\frac 1 4 \int \cD(\mu_\kappa g^n, \varphi(v,x)^{-1}g^{n+1}_-) dx +  C
\|g^n\|_{L^\infty([0,T] \times \RR_x^3; H^{2s})} \int \|     \varphi(v,x)^{-1}  g^{n+1}_-\|^2_{L^2_{s+\gamma/2}} dx.
\]
By means of
 Lemma \ref{commu-Gam}  we have
 \begin{align*}
 |A_{1,2} | & \le \delta  \int\Big( \cD(\mu_\kappa g^n, \varphi(v,x)^{-1}g^{n+1}_-)dx + \|g^n\|_{L^\infty([0,T] \times \RR_x^3; H^{s})} \int \|     \varphi(v,x)^{-1}  g^{n+1}_-\|^2_{L^2_{s+\gamma/2}} dx \Big)\\
 &\qquad \qquad  + C_\delta
 \|g^n\|_{L^\infty([0,T] \times \RR_x^3; H^{s})} \int \|     \varphi(v,x)^{-1}  g^{n+1}_-\|^2_{L^2} dx.
\end{align*}
Therefore
\begin{align*}
A_1
\lesssim \|g^n\|_{L^\infty([0,T] \times \RR_x^3; H^{2s})} \int \|     \varphi(v,x)^{-1}  g^{n+1}_-\|^2_{L^2_{s+\gamma/2}} dx\,,
\end{align*}
where we have used the fact that $\cD(\mu_\kappa g^n, \varphi(v,x)^{-1}g^{n+1}_-) \ge 0$ because of
the induction hypothesis $\mu_\kappa g^n=f^n \ge 0$. If $\gamma +2s <2$ then we have
\[
\int \|     \varphi(v,x)^{-1}  g^{n+1}_-\|^2_{L^2_{s+\gamma/2}} dx \leq
\delta \int_{\RR^6}   \la v \ra^2\beta (g^{n+1})
\varphi(v,x)^{-2} dxdv + C_\delta \int_{\RR^6}   \beta (g^{n+1})
\varphi(v,x)^{-2} dxdv.
\]
Therefore, from  \eqref{gronwall} we have
\begin{align}\label{final-gronwall-g}
&\frac{d}{dt} \int_{\RR^6}   \beta ( g^{n+1})  \varphi(v,x)^{-2} dxdv
\lesssim  \Big(1+  \|g^n\|_{L^\infty([0,T] \times \RR_x^3; H^{2s})}\Big) \int_{\RR^6} \beta(  g^{n+1})  \varphi(v,x)^{-2} dx dv\,.
\end{align}
Since $\beta(g^{n+1})|_{t=0} =0$ we obtain $\int_{\RR^6} \beta(  g^{n+1})  \varphi(v,x)^{-2} dx dv=0$ for all $t \in
[0,T]$,
which  implies that $g^{n+1}(t,x,v) \ge 0$ for $(t,x,v) \in [0,T]\times \RR^6$. This implies that $f^{n+1}\ge 0$ and then it completes the proof of the proposition.\end{proof}


\begin{prop}\label{induct-2}
Assume that $\gamma\ge 2-2s $.  Let $\{f^n\}$ with $f^n=\mu+\mu^{\frac 12}{\tilde g}^n$ be
 sequence of solutions of Cauchy problem \eqref{4.4.3} with
 $\sup_{[0,T]\times \RR_x^3} |||\tilde g^n|||_{\Phi_\gamma} $ being
 sufficiently small uniform in $n$. If
\[
e^{\frac 12 \la v\ra^2}  f^n(t,x,v)  \in L^\infty([0,T] \times \RR^3_x
;  H^N (\RR^3_v) )\enskip
\mbox{for} \enskip \forall n = 1,2,3,\cdots\,,
\]
for some $N \ge 4$, then for any $n\in\NN$,  $f^n \geq 0$ on $ [0,T]$ implies $f^{n+1} \geq 0$  on the same interval.
\end{prop}
\begin{proof}
This case can be treated by the same way as the proof of Theorem \ref{uniqueness-small-pur}. In fact,
if we put $g^n = \mu^{-1/2} f^n$, then we have
\begin{align*}
&\frac{d}{dt} \int_{\RR^6}  \beta ( g^{n+1}) \varphi(v,x)^{-2}dxdv 
\\
&\quad \leq
\int_{\RR^6}  \Gamma(g^n, g^{n+1} ) \beta_s(g^{n+1}) \varphi(v,x)^{-2}dxdv + C
\int_{\RR^6}   \beta ( g^{n+1}) \varphi(v,x)^{-2}dxdv \notag
\end{align*}
instead of  \eqref{gronwall}.  We need to estimate $\widetilde A_{1,1}$ and $\widetilde A_{1,2}$
defined by replacing $\Gamma^t$ by $\Gamma$ in above $A_{1,1}$ and $A_{1,2}$.
By the same way as in \eqref{gamma-hyouka} we have for suitable $C_1,C_2>0$
\begin{align*}
\widetilde A_{1,1}
 &\le - C_1\int|||\varphi(v,x)^{-1}  g^{n+1}_-|||^2_{\Phi_\gamma} dx\\
&+ C_2\Big(\big(\sup_{[0,T]\times \RR_x^3} |||\tilde g^n|||_{\Phi_\gamma} \big)
\int||| \varphi(v,x)^{-1}  g^{n+1}_-|||^2_{\Phi_\gamma} dx
+\int ||\varphi(v,x)^{-1}  g^{n+1}_-||^2_{L^2}dx\Big)\,.\notag
\end{align*}
 It follows from Lemma \ref{commu-Gam} and
 Lemma \ref{upper-triple-Dirichlet} that
\begin{align*}
|\widetilde A_{1,2}| &\le \delta \big(1+ \sup_{[0,T]\times \RR_x^3} ||\tilde g^n||_{L^2} \big)
\int|||\varphi(v,x)^{-1}  g^{n+1}_-|||^2_{\Phi_\gamma} dx\\
&\qquad +C_\delta
 \|g^n\|_{L^\infty([0,T] \times \RR_x^3; H^{s})} \int \|     \varphi(v,x)^{-1}  g^{n+1}_-\|^2_{L^2} dx.
\end{align*}
If $|||\tilde g^n|||_{\Phi_\gamma}$ is sufficiently small, then both estimates lead us to \eqref{final-gronwall-g}. Hence we have
$f^{n+1}\ge 0$ and then it completes the proof of the proposition.
\end{proof}

\noindent
{\bf Completion of the proof of Theorem \ref{theo-IV-1.3}.}

We recall now the existence and convergence of the sequence $\{\tilde{g}^n\}$ constructed in \cite{amuxy4-2, amuxy4-3} for different cases of index :

{\bf The hard potential case $\gamma+2s>0$.} (Theorem 1.1
of \cite{amuxy4-3})  Let $g_0 \in H^{k}_{\ell_0}(\RR^6)$ for some $k\geq 6,\, \ell_0> 3/2+2s+\gamma$. There exists $\varepsilon_0 >0$, such that if ~
$\|g_0\|_{H^{k}_{\ell_0}(\RR^6)}\leq \varepsilon_0$, then the sequence $\{\tilde{g}^n\}$ converges in $L^\infty([0, +\infty[\,;\,\,H^{k}_{\ell_0}(\RR^6))$
to a global solution $\tilde{g}$ with  $\|\tilde{g}\|_{L^\infty([0, +\infty[; H^{k}_{\ell_0}(\RR^6))}\leq C\varepsilon_0$.

{\bf The soft potential case $\gamma+2s\leq 0$.}\,\,  (Theorem 1.5  of \cite{amuxy4-2}) Assume $\gamma >
\max\{-3, -\frac32 -2s\}$. Let $g_0 \in {\tilde \cH}^{k}_k(\RR^6)$ for some $k\geq 6$. There exits
$\varepsilon_0>0$ such that  if
$ \|g_0\|_{{\tilde \cH}^{k}_k(\RR^6)}\leq \varepsilon_0,$
then the sequence $\{\tilde{g}^n\}$ converges in $ L^\infty([0, +\infty[\,;\,\,\tilde{\cH}^{k}_k(\RR^6))$ to a global solution $\tilde{g}$.
Remark that the approximate sequence $\{\tilde{g}^n\}$ is convergent
in $L^\infty([0, T]; H^k (\RR^6))$.

So in both cases, the sequence $f^n=\mu+\mu^{1/2} \tilde{g}^n$ satisfies the conditions of Propositions \ref{induct-1} and \ref{induct-2} with $N=k-2$, which implies that the limit $f=\mu+\mu^{1/2} \tilde{g}\ge0$. We have proved Theorem \ref{theo-IV-1.3}.

\section{Convergence to the equilibrium state}\label{sect-IV-9}
\smallskip \setcounter{equation}{0}

In this section, the convergence rates of the solutions to the
equilibrium will be discussed for both the soft and hard potentials.
Precisely, for the hard potential, the
optimal convergence rates in the Sobolev space can be obtained by
combining the energy estimates proven previously and the $L^p-L^q$
estimate on the solution operator of the linearized equation. Such
$L^p-L^q$ estimate can be obtained  either by spectrum analysis \cite{ukai-2} or
by using the compensating functions introduced by Kawashima \cite{kawashima}. On the other hand, for soft potential,
the convergence rate presented here is solely based on the energy
estimate and is not optimal.

\subsection{Hard potential.}

In this subsection, we will combine the compensating function and
the energy estimate to obtain the optimal convergence rate for the
hard potential case $\gamma+2s>0$, that is, the first part of
Theorem \ref{theo-IV-1.4}. Note that the  decay estimate in the theorem can be  generalized to the
case when the initial lies in $Z_q(\RR^6)$ with $1< q <2$,
where $Z_q(\RR^6)=L^q(\RR^3_x;L^2_v(\RR^3))$.

The compensating function is useful in deriving $L^p-L^q$ estimates
for linear dissipative kinetic equations in the form of
\begin{equation}\label{con-1}
g_t +v\cdot \nabla_x g +\cL g =h,
\end{equation}
where $h$ is a given function and $\cL$ is  the linearized Boltzmann collision operator.

Let us now recall the definition of compensating function introduced
by Kawashima \cite{kawashima}.
\begin{defi}
 $S(\omega)$ is called a compensating function
if it has the following properties:

(i) $S(\cdot)$ is $C^\infty$ on $\SS^2$ (the unit sphere in $\RR^3$)
 with values in the space
of bounded linear operators on $L^2({\RR}^3)$, and
$S(-\omega)=-S(\omega)$ for all $\omega\in \SS^2$.

(ii) $iS(\omega)$ is self-adjoint on $L^2({\RR}^3)$ for all
$\omega\in \SS^2 $.

(iii) There exist constants $\lambda>0$ and $c_0>0$ such that for all $g\in
L^2(\RR^3)$ and $\omega\in \SS^2$,
\begin{equation}\label{comp-fct}
Re( S(\omega)(v\cdot\omega)g,g
)_{L^2(\RR_v^3)}+( \cL g,g)_{L^2(\RR_v^3)}\geq c_0(\| {\pP}g\|_{L^2(\RR^3_v)}^2+||| ({\bf
I-P})g|||^2).
\end{equation}
\end{defi}

The construction of $S(\omega)$ was given in \cite{kawashima}, but
for completeness and the convenience of the readers, we recall some
basic derivation and estimates.

 Let $\cW$ be the subspace spanned by the thirteen
moments containing the null space $\mathcal{N}$ of
$\cL$
and the
images of $\mathcal{N}$ under the mappings $g(v)\mapsto v_jg(v)$
$(j=1,2,3)$ denoted by:
$$
\cW=\mbox{span}\{e_j|j=1,2, \cdots,13\}.
$$
Here, the orthonormal set of functions $e_j$ is given by
$$
e_1=\mu^{\frac 12}, \quad e_{i+1}=v_i\mu^{\frac 12},\quad i=1,2,3,\quad
e_5=\frac{1}{\sqrt{6}}(|v|^2-3)\mu^{\frac 12},
$$
and
\begin{eqnarray*}
e_{j+4}&=&\sum_{i=1}^3 \frac{c_{ji}}{\sqrt{2}}(v_i^2 -1)\mu^{\frac 12}, j=2,3\\
e_8&=&v_1v_3\mu^{\frac 12},\quad e_9=v_2v_3\mu^{\frac 12},\quad e_{10} =v_3v_1 \mu^{\frac 12},\\
e_{i+10}&=&\frac{1}{\sqrt{10}}(|v|^2-5)v_i\mu^{\frac 12},\quad i=1,2,3,
\end{eqnarray*}
where the constant vectors
$c_i=(c_{i1}, c_{i2}, c_{i3})$, $i=2,3$ together with $c_1=(\frac{1}{\sqrt 3},
\frac{1}{\sqrt 3},\frac{1}{\sqrt 3})$ form an orthonormal basis of $\RR^3$.

Let $\pP_0$ be the orthogonal projection from $L^2({\RR}^3_v)$ onto
$\cW$, that is,
$$
\pP_0g=\sum_{k=1}^{13}( g,e_k)_{L^2(\RR^3_v)} e_k.
$$

Set $ { W_k=\langle f,e_k\rangle} $, $k=1,2,\cdots, 13$, and
$W=[W_1,...,W_{13}]^T$. For later use, set
$W_I=[W_1,...,W_{5}]^T$, and $W_{II}=[W_6,...,W_{13}]^T$. Then we have
$$
\partial_tW+\sum_{j}V^j\partial_{x_j}W+\overline{L}W=\overline{h}+R,
$$
where $V^j$ $(j=1,2,3)$ and $\overline{L}$ are the symmetric
matrices defined by
$$
\overline{L}=\{( \cL e_l,e_k)_{L^2(\RR^3_v)}\}_{k,l=1}^{13},\ \ \ V^j=\{(
v_je_k,e_l)_{L^2(\RR^3_v)}\}_{k,l=1}^{13},
$$
and $\overline{h}=[( h,e_1)_{L^2(\RR^3_v)},...,( h,e_{13})_{L^2(\RR^3_v)}]^T$. Here $R$ denotes the remaining
term which  contains  the factor $({\bf I}-\pP_0)g$. Straightforward
calculation gives
$$
V(\xi)
=\sum_{j=1}^3V^j\xi_j=\begin{pmatrix}
 V_{11} & V_{12}\\
 V_{21} & V_{22} \end{pmatrix},$$
with
$$V_{11}(\xi)=\begin{pmatrix}
 0 & \xi_1 & \xi_2 & \xi_3&0\\
 \xi_1 & 0 & 0 & 0 &a_1\xi_1  \\
 \xi_2& 0 & 0 & 0 & a_1\xi_2\\
 \xi_3& 0 & 0 & 0 &a_1\xi_3\\
0&a_1\xi_1&a_1\xi_2&a_1\xi_3&0\end{pmatrix},$$
and
$$V_{21}(\xi)=V_{12}(\xi)^T=\begin{pmatrix}
 0 & a_{21}\xi_1 & a_{22} \xi_2 & a_{23}\xi_3&0\\
 0 & a_{31}\xi_1 & a_{32} \xi_2 & a_{33}\xi_3&0\\
 0 & \xi_2 & \xi_1 & 0 & 0  \\
 0 & 0 & \xi_3 & \xi_2 & 0\\
 0& \xi_3 & 0 & \xi_1 & 0\\
0& 0 &0 & 0&a_4\xi_1\\
0& 0 &0 & 0&a_4\xi_2\\
0& 0 &0 & 0&a_4\xi_3\\
\end{pmatrix},$$
where $a_1=\sqrt{\frac 23}, a_{kj}=\sqrt{2} c_{kj}$, $k=2,3,$ $j=1,2,3$,
and $a_4=\sqrt{\frac 35}$. By setting
$$
R(\xi)=\sum_{j=1}^3R^j\xi_j=\begin{pmatrix}
 \alpha \tilde{R}_{11} & V_{12}\\
 -V_{21} & 0\end{pmatrix},
$$
with
$$\tilde{R}_{11}=\begin{pmatrix}
 0 & \xi_1 & \xi_2 & \xi_3&0\\
 -\xi_1 & 0 & 0 & 0 &0  \\
 -\xi_2& 0 & 0 & 0 & 0\\
 -\xi_3& 0 & 0 & 0 &0\\
0&0&0&0&0\end{pmatrix} .$$
It was shown in \cite{kawashima} that
 there exist positive constants $c_1$ and $c_2$ such
that
$$
Re \langle R(\omega)V(\omega)W,W\rangle\geq c_1|W_I|^2-c_2\sum_{k=2}^4|W_{II}|^2,
$$
for any $\omega\in \SS^2$ with the constant $\alpha$ suitably chosen.
Here $\langle\cdot,\cdot\rangle$ represents the standard inner product in $\CC^{13}$.

Hence,  a compensating
function $S(\omega)$ can be defined as follows.
 For any given $\omega\in \SS^2$, set $R(\omega)\equiv\{r_{ij}(\omega)\}_{i,j=1}^{4}$ and let
$$
S(\omega)g\equiv\sum_{k,\ell=1}^{4}\lambda
r_{k\ell}(\omega)( g,e_\ell)_{L^2(\RR^3_v)}e_k, \qquad f\in L^2(\RR^3).
$$
When the parameter $\lambda>0$ is chosen small enough,
it was shown in \cite{kawashima}
 that the estimate (\ref{comp-fct}) holds because of the dissipation
 of $\cL$ on the space $\mathcal{N}^\perp$.

To obtain the $L^p-L^q$ estimate,
 by taking the Fourier transform in the variable x, the equation
(\ref{con-1}) yields
\begin{equation}\label{con-2}
\hat{g}_t +i|\xi| (v\cdot\omega) \hat{g} +\cL \hat{g}= \hat{h},
\end{equation}
where $\omega=\frac{\xi}{|\xi|}$. Take the inner product of \eqref{con-2} with
$((1+|\xi|^2) -i\kappa S(\omega))\hat{g}$ and use the properties of the
compensating function, to get
\begin{eqnarray*}
((1+|\xi|^2)\|\hat{g}\|^2_{L^2(\RR^3_v)}&-&\kappa|\xi|(i S(\omega) \hat{g},\hat{g})_{L^2(\RR^3_v)})_t +\delta_0((1+|\xi|^2)|||({\bf I}-\pP) \hat{g}|||^2+
|\xi|^2 ||\pP \hat{g}\|^2_{L^2(\RR^3_v)})\nonumber\\
&\le& C(1+|\xi|^2)Re(\hat{g}, \hat{h})_{L^2(\RR^3_v)},
\end{eqnarray*}
which implies that
\begin{equation*}
E(\hat{g})_t +\delta_0 \frac{|\xi|^2}{1+|\xi|^2} E(\hat{g})\le C\|\hat{h}\|^2_{L^2(\RR^3_v)},
\end{equation*}
where
$$
E(\hat{g})=\|\hat{g}\|^2_{L^2(\RR^3_v)} -\kappa\frac{|\xi|}{1+|\xi|^2}(i S(\omega) \hat{g},\hat{g})_{L^2(\RR^3_v)}\sim \|\hat{g}\|^2_{L^2(\RR^3_v)},
$$
when $\kappa$ is chosen to be small. And this estimate yields
\begin{equation}\label{con-5}
\|\hat{g}\|^2_{L^2(\RR^3_v)}\le C\exp\{-\frac{\delta_0|\xi|^2|t}{1+|\xi|^2}\}
\|\hat{g_0}\|^2_{L^2(\RR^3_v)}+C\int_0^t\exp\{-\frac{\delta_0|\xi|^2|(t-s)}{1+|\xi|^2}\}\|\hat{h}\|^2_{L^2(\RR^3_v)}(s) ds.
\end{equation}

Based on \eqref{con-5}, we have the following $L^p-L^q$ estimate on the
solution operator of (\ref{con-1}) obtained in \cite{kawashima}.

\begin{lemm}
Let $k\geq k_1\geq 0$ and $N\geq 4$. Assume that

 $(i)$  $g_0\in H^N(\RR^6)\cap Z_q$,

 $(ii)$ $h\in C^0([0,\infty[; H^N\cap Z_q)$,

 $(iii)$ $\pP h(t,x,v)=0$ for all $(t,x,v)\in[0,\infty)\times{\RR}^3\times{\RR}^3$.

 $(iv)$
 $g(t,x,v)\in C^0([0,\infty[; H^N(\RR^6))\cap C^1([0,\infty[; H^{N-1}(\RR^6)) $ is a solution of
(\ref{con-1}).

Then we have
$$\|\nabla_x^kg\|^2_{L^2(\RR^6)}\leq
C(1+t)^{-2\sigma_{q,m}}(\|\nabla_x^{k_1}g_0\|_{Z_q(\RR^6)}+\|\nabla_x^{k}g_0\|_{L^2(\RR^6)})^2
$$
$$+\int_0^t(1+t-s)^{-2\sigma_{q,m}}(\|\nabla_x^{k_1}h\|_{Z_q(\RR^6)}+\|\nabla_x^{k}h\|_{L^2(\RR^6)})^2ds
,\eqno(2.11)$$ for any integer $m=k-k_1\geq 0$, where $q\in[1,2]$
and
$$\sigma_{q,m}=\frac 32\Big(\frac 1q-\frac 12\Big)+\frac
m2.$$
\end{lemm}

We now recall the energy estimates obtained for the global existence
of solutions for the hard potential in \cite{amuxy4-3}.
Firstly, we have when $N\ge 6$ and $l>3/2+2s +\gamma$,
\begin{equation*}
\frac{d}{dt} \mathcal{E} +D\le 0,
\end{equation*}
where $\mathcal{E} =\|g\|^2_{H^N_l(\RR^6)}$ and $D=\|\nabla_x \pP g\|^2_{H^{N-1}(\RR^6)}
+|||({\bf I}-\pP) g|||^2_{\mathcal{B}^N_l(\RR^6)}$. We claim that
 the following energy estimate also holds
\begin{equation}\label{con-7}
\frac{d}{dt} \mathcal{E}_1 +D\le C\|\nabla_x \pP g\|^2_{L^2_{x,v}(\RR^6)},
\end{equation}
where $\mathcal{E}_1=\|\nabla_x\pP g\|^2_{H^{N-1}(\RR^6)} +\|({\bf I}-\pP) g\|^2_{H^N_l(\RR^6)}$. In fact, for the energy estimate on the macroscopic component of
the solution, by Lemma 4.4 of  \cite{amuxy4-3} and by taking the sum over
$1\le |\alpha|\le N-1$, we have
\begin{align}\label{pabc-1}
\|\nabla_x  \cA\|_{H^{N-1}(\RR^3_x)}^2
&\lesssim
-\frac{d}{dt}\sum_{1\le |\alpha|\le N-1}\Big\{(\partial^\alpha r,\nabla_x \partial^\alpha (a, - b, c))_{L^2(\RR^3_x)}
+(\partial^\alpha b, \nabla_x \partial^\alpha a)_{L^2(\RR^3_x)}\Big\}
\\& +\notag
\|g_2\|_{H^N(\RR^3_x;L^2(\RR^3_v))}^2+\mathcal{E}^{\frac 12} D,
\end{align}
where we use the same notations used in \cite{amuxy4-3} except that we replace
$E_{N,1}D_{N,0}$ by $\mathcal{E}^{\frac 12}D$ in an obvious way.

On the other hand, for the energy estimate on the microscopic component
without weight, one can follow the proof of Lemma 4.5 in \cite{amuxy4-3} except
for $\alpha=0$, we take the $L^2(\RR^6_{x,v})$ inner product of
\eqref{con-1} with $(\iI-\pP)g=g_2$ to have
$$
\frac{d}{dt} \|g_2\|^2_{L^2(\RR^6_{x,v})} +
|||g_2|||^2_{{\cB}^0_0(\RR^6)}\lesssim \|\nabla_x\pP g\|^2_{L^2(\RR^6_{x,v})} +\mathcal{E}^{\frac 12}D.
$$
This together with the estimate given in Lemma 4.5 of \cite{amuxy4-3} for $1\le
|\alpha|\le N$ gives
\begin{align}\label{micro-1}
\frac{d}{dt}\big(\|g_2\|^2_{L^2(\RR^6_{x,v})}+ \|\nabla_x g\|^2_{H^{N-1}(\RR^3_x;L^2(\RR^3_v))}\big) +|||g_2|||^2_{{\cB}^0_0(\RR^6)}
+\sum_{1\le |\alpha|\le N}|||\partial^\alpha g|||^2_{{\cB}^0_0(\RR^6)}
\lesssim
\|\nabla_x\pP g\|^2 +\mathcal{E}^{\frac 12}D.
\end{align}
Moreover, for $\ell\ge 0$ and $\beta\neq 0$, (4.10) in \cite{amuxy4-3} gives
\begin{align}\notag
 \frac{d}{dt}&\| W_\ell\partial^\alpha_\beta g_2\|^2_{L^2(\RR^6)}+
|| W_\ell\partial^\alpha_\beta g_2 ||_{\cB^0_0}
\\&\lesssim\label{beta-1}
\|g_2\|_{H^N(\RR^3_x, L^2(\RR^3_v))}^2
+\cE_{N,\ell}^{1/2}\cD_{N,\ell}
\\&\notag+\delta_0\|\nabla_x\cA\|_{H^{N-1}(\RR^3)}^2
+\sum_{|\alpha'|=|\alpha|+1, |\beta'|=|\beta|-1}\|\partial^{\alpha'}_{\beta'}g_2\|_{\cB^{0}_\ell}^2.
\end{align}
Here $\delta_0>0$ is a small constant.
By using induction on $|\beta|$ and $|\alpha|+|\beta|$, a suitable linear
combination of \eqref{pabc-1}, \eqref{micro-1} and \eqref{beta-1} gives
\eqref{con-7} for sufficiently small $\mathcal{E}$.

In the following, we also need some $L^p$ estimate on the nonlinear
collision operator. Recall that
Lemma \ref{upp-l-2} implies that
\[
\|\Gamma( f, \, g)\|_{L^2(\RR^3_v)}\lesssim
\|f\|_{L^2(\RR^3_v)}
\|g\|_{H^{2s}_{(\gamma+2s)^+}(\RR^3_v)}.
\]
Hence, by using the fact that $N\ge 6$ and $\ell>3/2+2s+\gamma$,
Sobolev imbedding implies
\begin{eqnarray*}
\|\Gamma(g,g)\|^2_{L^2(\RR^6_{x,v})}+\|\nabla_x\Gamma(g,g)\|^2_{L^2(\RR^6_{x,v})}&\lesssim& \mathcal{E}^2\lesssim
\mathcal{E}_1\mathcal{E} +\|\pP g\|^4_{L^2(\RR^6_{x,v})},\nonumber\\
\|\Gamma(g,g)\|^2_{Z_1}&\lesssim& \mathcal{E}_1\mathcal{E} +\|\pP g\|^4_{L^2(\RR^6_{x,v})}.
\end{eqnarray*}
Define
$$
M(t)=\sup_{0\le s\le t}\{(1+s)^{\frac 52} \mathcal{E}_1(s)\},
\quad M_0(t)=\sup_{0\le s\le t}\{(1+s)^{\frac 32}\|g(s)\|^2_{L^2(\RR^6_{x,v})}\}.
$$
Then by the $L^p-L^q$ estimate, we have
\begin{eqnarray*}
\|\nabla_x g(t)\|^2_{L^2(\RR^6_{x,v})}&\lesssim & (1+t)^{-\frac 52}(\|g_0\|^2_{Z_1(\RR^6)} + \|\nabla_x g_0\|^2_{L^2(\RR^6_{x,v})})\nonumber\\
&&+ \int_0^t (1+t-s)^{-\frac 52}(\|\Gamma(g,g)\|_{Z_1(\RR^6)} +\|\nabla_x\Gamma(g,g)\|_{L^2(\RR^6_{x,v})})^2 ds\nonumber\\
&\lesssim & \eta(1+t)^{-\frac 52} + \int_0^t(1+t-s)^{-\frac 52}(\mathcal{E}\mathcal{E}_1 +\|\pP g\|^4_{L^2(\RR^6_{x,v})})(s) ds\nonumber\\
&\lesssim & \eta (1+t)^{-\frac 52} +\delta M(t)\int_0^t (1+t-s)^{-\frac 52}(1+s)^{-\frac 52}ds \nonumber\\
&&+M_0^2(t)\int_0^t (1+t-s)^{-\frac 52}(1+s)^{-3} ds \nonumber\\
&\lesssim & \eta(1+t)^{-\frac 52} +\delta (1+t)^{-\frac 52}M(t) +(1+t)^{-\frac 52}M_0^2(t),
\end{eqnarray*}
where $\eta=\|g_0\|^2_{Z_1(\RR^6)} + \| g_0\|^2_{H^N_\ell(\RR^6)}$.
Here, we use $\delta>0$ to denote the upper bound
of  $\mathcal{E}$ for all time.

Thus, we have
\begin{eqnarray*}
\mathcal{E}_1(t)&\le & \mathcal{E}_1(0)e^{-t}
 +\int_0^t e^{-(t-s)}\|\nabla_x g\|^2_{L^2(\RR^6_{x,v})} (s) ds\nonumber\\
&\lesssim &\delta e^{-t} +\eta (1+t)^{-\frac 52} +\delta (1+t)^{-\frac 52}M(t)
+(1+t)^{-\frac 52}M_0^2(t),
\end{eqnarray*}
that is,
\begin{eqnarray*}
M(t)\lesssim (\delta+\eta) +\delta M(t) +M^2_0(t).
\end{eqnarray*}
By applying the $L^p-L^q$ estimate again, we have
\begin{eqnarray*}
\|g(t)\|^2_{L^2(\RR^6_{x,v})}&\lesssim& (1+t)^{-\frac 32}(\|g_0\|^2_{Z_1(\RR^6)} +\|g_0\|^2_{L^2(\RR^6_{x,v})})\nonumber\\
&&+\int_0^t (1+t-s)^{-\frac 32}(\|\Gamma(g,g)\|_{Z_1(\RR^6)}+\|\Gamma(g,g)\|_{L^2(\RR^6_{x,v})})^2(s) ds\nonumber\\
&\lesssim& \eta (1+t)^{-\frac 32} +\int_0^t (1+t-s)^{-\frac 32}(\mathcal{E}\mathcal{E}_1+\|\pP g\|^4_{L^2(\RR^6_{x,v})})(s) ds\nonumber\\
&\lesssim& \eta (1+t)^{-\frac 32} +\delta(1+t)^{-\frac 32}M(t) +(1+t)^{-\frac 32} M^2_0(t).
\end{eqnarray*}
Hence,
\begin{eqnarray*}
M_0(t)&\lesssim& \eta + \delta M(t) +M^2_0(t)\nonumber\\
&\lesssim & (\eta +\delta)+ M^2_0(t).
\end{eqnarray*}
By assumption, $\eta +\delta$ is small. The above estimate and the
continuity argument give
$
M_0(t)\le C_{\eta,\delta},
$
and then
$
M(t)\le \bar{C}_{\eta,\delta},
$
where $C_{\eta,\delta}$ and $\bar{C}_{\eta,\delta}$ are two constants depending
on $\eta$ and $\delta$ only.
This completes the proof of the first part in Theorem
\ref{theo-IV-1.4}.

\subsection{Soft Potential.}
Finally, in this subsection, we will
prove the second part of Theorem \ref{theo-IV-1.4} t for the soft potential case, that is, when $2s+\gamma\le 0$.

As for the case with angular cutoff, here we need to apply the
following basic inequality from \cite{Deckelnick}.

\begin{lemm}\label{6-l-1}
 Let $f(t)\in C^1([t_0,\infty))$ such
that $f(t)\geq 0,\
A={\displaystyle\int^\infty_{t_0}}f(t)dt<+\infty$ and $f'(t)\le
a(t)f(t)$ for all $t\geq t_0$. Here $a(t)\geq 0,\
B={\displaystyle\int^\infty_{t_0}}a(t)dt<+\infty$. Then
$$
f(t)\leq \frac{(t_0f(t_0)+1)\exp(A+B)-1}{t},\quad \mbox{ for all} \quad t\geq
t_0.
$$
\end{lemm}

Now it remains to find the appropriate functionals $f(t)$ and $a(t)$
that satisfy the above differential inequality.

First of all, the basic energy estimate derived in \cite{amuxy4-2} for the
global existence is
$$
\frac{d}{dt} \cE_{N,\ell} + c_0 \cD_{N,\ell}\le 0,
$$
where $c_0>0$ is a constant. Here,
\begin{align*}
\cE_{N,\ell}&\sim \|\cA\|^2_{H^N(\RR^3)}+\|g_2\|_{\widetilde{\cH}^N_\ell(\RR^6)}^2,
 \\ \cD_{N,\ell}& =\|\nabla_x \cA\|_{H^{N-1}(\RR^3)}^2+
\|g_2\|_{\widetilde{\cB}^N_\ell(\RR^6)}^2,
\end{align*}
and
\begin{align*}
{\widetilde\cB\,}^N_\ell(\RR^6)=\Big\{ g\in\cS'(\RR^6);\,\,
||g||^2_{{\widetilde\cB\,}^N_\ell(\RR^6)}
=\sum_{|\alpha|+|\beta|\leq N}\int_{\RR^3_x}|||\tilde W_{\ell-|\beta|}\,
\partial^\alpha_{\beta} g(x,\, \cdot\,)|||^2_{\Phi_\gamma}dx <+\infty\Big\}\, .
\end{align*}

Later we introduce
 another functional $\bar{\cE}_{N-1,\ell-1}$ that has the following
property
\begin{align}\label{6-3}
\bar{\cE}_{N-1,\ell-1} \sim \|\nabla_x\cA\|_{H^{N-2}(\RR^3)}^2
+\|\nabla_x g_2\|^2_{\widetilde{\cH}^{N-2}_{\ell-1}(\RR^6)}.
\end{align}
Clearly, by the property of the $|||\cdot|||$, $\bar{\cE}_{N-1,\ell-1}\lesssim
\cD_{N,\ell}$ so that $\int_0^\infty\bar{\cE}_{N-1,\ell-1}(t) dt<\infty$.
Note that this functional contains spatial differentiation of at least
one order, and the maximum order of differentiation is $N-1$. The reason
for this is to have the time integrability
of the functional coming from the dissipation effect. And the functional excludes
the $N$-th order differentiation because we need to estimate the term like
\begin{align*}
\int_{\RR^3} \|\cA\|_{L^2(\RR^3_x)}\|\partial^\alpha_x \cA\|_{L^2(\RR^3_x)}|||\partial^\alpha_x g_2||| dx\lesssim
&\|\cA\|_{L^6(\RR^3_x)}\|\partial^\alpha_x \cA\|_{L^3(\RR^3_x)}||\partial^\alpha_x g_2||_{\cX^0(\RR^6)}\\
\lesssim \|\nabla_x\cA\|_{L^2(\RR^3_x)}\|\partial^\alpha_x \cA\|_{H^1(\RR^3_x)}||\partial^\alpha_x g_2||_{\cX^0(\RR^6)},
\end{align*}
where
$$
\cX^{N}(\RR^6)=\Big\{ g\in\cS'(\RR^6);\,\,\
||g||_{\cX^N(\RR^6)}^2=\sum_{|\alpha|\le N}
\int_{\RR^3_x}|||\partial_x^\alpha g|||^2_{\Phi_\gamma}dx<+\infty\Big \}.
$$
Hence, since the maximum order of differentiation is $N$, the above
estimate requires that $|\alpha|\le N-1$.

We now construct $\bar{\cE}_{N-1,\ell-1}$ following the argument used in
Lemma 6.3, Lemma 6.4 and (6.15) in \cite{amuxy4-2}. Firstly, by taking
$1\le |\alpha|\le N-2$ in Lemma 6.3 of \cite{amuxy4-2}, we have
\begin{align}\label{6-4}
\|\nabla_x \partial^\alpha \cA\|_{L^2(\RR^3_x)}^2
\le&
-\frac{d}{dt}\Big\{(\partial^\alpha r,\nabla_x \partial^\alpha (a, - b, c))_{L^2(\RR^3_x)}
+(\partial^\alpha b, \nabla_x \partial^\alpha a)_{L^2(\RR^3_x)}\Big\}
\\ &\quad +\notag\|\nabla_x g_2\|_{\cX^{N-2}(\RR^6)}^2+
\big(\|\cA\|_{H^{N-1}(\RR^3_x)}+\|g_2\|_{H^N(\RR^3_x; L^2(\RR^3_v))}\big)\bar{\cD}_{N-1}\\
&\quad +\notag \|\nabla_x \cA\|^2_{H^{N-1}(\RR^3_x)}
\big(\|\nabla_x \cA\|^2_{H^{N-2}(\RR^3_x)}
+\|g_2\|^2_{L^2(\RR^6)}\big),
\end{align}
where
$$
\bar{\cD}_{N-1}=\sum_{1\le \alpha\le N-2}\|\nabla_x\partial^\alpha \cA\|^2_{L^2(\RR^3_x)} +\|\nabla_x g_2\|^2_{\cX^{N-2}}.
$$
Note that $\bar{\cD}_{N-1}$ is different from $\cD_{N-1}$ defined in \cite{amuxy4-2}.
In particular, in $\bar{\cD}_{N-1}$,
the usual dissipation terms $\|\nabla_x \cA\|^2_{L^2(\RR^3_x)}$ and
$\|g_2\|^2_{L^2(\RR^6)}$ are not included. And this is also why there
is the last term on the right hand side of \eqref{6-4}.

Next, following the argument used for Lemma 6.4 in \cite{amuxy4-2}, we can derive
\begin{align*}
\frac{d}{dt} \bar{\cE}_{N-1} +c_0 \|\nabla_x g_2\|^2_{\cX^{N-2}}\lesssim
\cE^{\frac 12}_N \bar{\cD}_{N-1} +\sum_{0\le |\alpha|\le N-2}\|\nabla_x\partial^\alpha \cA\|^4_{L^2(\RR^3_x)},
\end{align*}
where ${\cE}_N=\sum_{0\le |\alpha|\le N-1}\|\partial^\alpha_x g\|^2_{L^2(\RR^6)}$ and $\bar{\cE}_{N-1}=\sum_{1\le |\alpha|\le N-2}\|\partial^\alpha_x g\|^2_{L^2(\RR^6)}$.

Finally, corresponding to the weighted estimate (6.15) in \cite{amuxy4-2},
one can show that for $|\alpha|\ge 1$ and $|\beta|\ge 1$ with
$|\alpha+\beta|\le N-1$, it holds
\begin{align*}
\frac{d}{dt}\Big(
\|\partial^\alpha_{\beta}&g_2\|_{L^2_{\ell-1-|\beta|}(\RR^6)}+
(\partial^\alpha r,\nabla_x \partial^\alpha (a, -b,c))_{L^2(\RR^3_x)}
+(\partial^\alpha b, \nabla_x \partial^\alpha a)_{L^2(\RR^3_x)}
\Big)
\\&\notag\hspace{1cm}
+c_0\| \
||| \tilde W_{\ell-1-|\beta|}\partial^\alpha_{\beta} g_2 |||_{\Phi_\gamma} \ \|_{L^2(\RR^3_x)}^2
\\&\lesssim\notag
 || \partial^\alpha g_2 ||^2_{ L^2_{s+\gamma/2}(\RR^6)}+
\cE_{N,\ell-1}^{1/2}\bar{\cD}_{N-1,\ell-1}
\\&\notag+\sum_{|\alpha'|=|\alpha|+1,|\beta'|=|\beta|-1}\| \ |||\tilde W^{\ell-1-(|\beta|-1)}\partial_x^{\alpha'}\partial_v^{\beta'}g_2|||_{\Phi_\gamma} \ \|_{L^2(\RR^3_x)}^2
+\delta_0
\|\nabla_xg_2\|_{\widetilde{\cB}^{N-1}_\ell(\RR^6)}^2\\
&\notag +\|\nabla_x \cA\|^2_{H^{N-1}(\RR^3_x)}\big(\|\nabla_x \cA\|^2_{H^{N-2}(\RR^3_x)}
+\|g_2\|^2_{L^2(\RR^6)}\big),
\end{align*}
where $\delta_0>0$ is a small constant, and
\begin{align*}
 \bar{\cD}_{N-1,\ell-1} =\sum_{1\le |\alpha|\le N-2}\|\nabla_x\partial
 \cA\|_{L^{2}(\RR^3_x)}^2+
\|\nabla_x g_2\|_{\widetilde{\cB}^{N-2}_{\ell-1}(\RR^6)}^2.
\end{align*}
Here, we have used the assumption that $\ell-1\ge N$.

Now we can define the functional $\bar{\cE}_{N-1,\ell-1}$ as follows:
\begin{align*}
\bar{\cE}_{N-1,\ell-1}&=\bar{\cE}_{N-1} +c_1\sum_{1\le |\alpha|\le N-2}
\Big\{(\partial^\alpha r,\nabla_x \partial^\alpha (a, - b, c))_{L^2(\RR^3_x)}
+(\partial^\alpha b, \nabla_x \partial^\alpha a)_{L^2(\RR^3_x)}\Big\}\\+
&\sum_{|\alpha|, |\beta|\ge 1, |\alpha+\beta|\le N-1} c_{\alpha,\beta}\Big(
\|\partial^\alpha_{\beta}g_2\|_{L^2_{\ell-1-|\beta|}(\RR^6)}+
(\partial^\alpha r,\nabla_x \partial^\alpha (a, -b,c))_{L^2(\RR^3_x)}
+(\partial^\alpha b, \nabla_x \partial^\alpha a)_{L^2(\RR^3_x)}
\Big),
\end{align*}
where $c_1>0$ and $c_{\alpha,\beta}>0$ are small constants which
can be  chosen so that $\bar{\cE}_{N,\ell-1}$ satisfies
\eqref{6-3}. It is straightforward
to check that by induction on $|\beta|$, we have
\begin{align*}
\frac{d}{dt} \bar{\cE}_{N-1,\ell-1} &+\eta_0 \bar{\cD}_{N,\ell-1}\lesssim
\|\nabla_x \cA\|^2_{H^{N-1}(\RR^3_x)}\big(\|\nabla_x \cA\|^2_{H^{N-2}(\RR^3_x)}
+\|g_2\|^2_{L^2(\RR^6)}\big)\nonumber\\
&\lesssim \|\nabla_x \cA\|^2_{H^{N-1}(\RR^3_x)}\bar{\cE}_{N-1,\ell-1},
\end{align*}
where $\eta_0>0$ is a  constant.
Here, we have used the fact that $\cE_{N,\ell}(t)$ is sufficiently small
for all time from the global existence.
Since
$$
\int_0^\infty(\bar{\cE}_{N-1,\ell-1}+ \|\nabla_x \cA\|^2_{H^{N-1}(\RR^3_x)})dt<\infty,
$$
Lemma \ref{6-l-1} implies that
$$
\bar{\cE}_{N-1,\ell-1}\lesssim (1+t)^{-1}.
$$
By using the fact that $\ell-1\ge N$, the Sobolev imbedding theorem
implies the decay estimate given in the second part of Theorem \ref{theo-IV-1.4}.

\end{document}